\documentclass[12pt,a4paper,twoside]{article}
\usepackage{amsmath,amsfonts,amssymb,amsthm,graphics}
\usepackage[mathscr]{euscript}
\newtheorem{Theorem}[equation]{Th\'eor\`eme}
\newtheorem{Corollary}[equation]{Corollaire}
\newtheorem{Proposition}[equation]{Proposition}
\newtheorem{Lemma}[equation]{Lemme}

\newtheorem{Remark}[equation]{Remarque}

\newtheorem{Definition}[equation]{D\'efinition}

\newtheorem{Notation}[equation]{Notation}

\def\Section#1{\section{#1}\setcounter{equation}{0}}

\newenvironment{resume}{%
\begin{abstract}
}{\end{abstract}}

\newenvironment{biblio}{%
\begin{thebibliography}
}{\end{thebibliography}}

\font\smc=cmcsc10

\def\bdots{\mathinner{\mkern1mu\raise1pt\hbox{.}\mkern2mu\raise4pt\hbox{.}
\mkern2mu\raise7pt\vbox{\kern7pt\hbox{.}}\mkern1mu}}

\def\oF{{\mathfrak o}_F}

\def\pF{{\mathfrak p}_F}

\def\oE{{{\mathfrak o}_E}}

\def\pE{{\mathfrak p}_E}

\def\fA{{\mathfrak A}}

\def\vF{{\nu_F}}

\def\End{{\hbox{\rm End}}}

\def\Hom{{\hbox{\rm Hom}\,}}
\def\max{\hbox{\rm max}\,}
\def\min{\hbox{\rm min}\,}
\def\dim{\hbox{\rm dim}\,}
\def\ker{\hbox{\rm ker}\,}

\def\det{\hbox{\rm det}\,}

\def\cInd{c{\hbox{\rm-Ind}}}

\def\Lie{\hbox{\rm Lie}\,}

\def\vol{\hbox{\rm vol}\,}

\DeclareMathOperator{\GL}{GL}
\DeclareMathOperator{\SL}{SL}

\def\su{{\mathfrak {su}}}

\def\mfu{{\mathfrak u}}
\def\msl{{\mathfrak {sl}}}
\DeclareMathOperator{\lg2}{{\mathfrak g}_{2}}
\DeclareMathOperator{\SO}{SO}
\def\O{{\mathrm O}}
\DeclareMathOperator{\SU}{SU}
\DeclareMathOperator{\U}{U}
\DeclareMathOperator{\G2}{G_2}
\DeclareMathOperator{\Spin}{Spin}

\def\tr{{\hbox{\rm tr}}}

\def\CC{{\mathcal C}}

\def\BZ{{\mathbb Z}}

\def\b{{\beta}}

\def\L{{\Lambda}}

\def\tJ{{\tilde J}}
\def\tP{{\tilde P}}

\def\oG{{\bar G}}
\def\oH{{\bar H}}
\def\oJ{{\bar J}}
\def\oP{{\bar P}}

 \def\bs{\mathrm}

\def\un{{\mathit 1}}

\def\today{\number\day\space
 \ifcase\month\or
 janvier\or f\'evrier\or mars\or avril\or mai\or juin\or
 juillet\or ao\^ut\or septembre\or octobre\or novembre\or d\'ecembre\fi
\space\number\year}

\begin{document}

\title{Caract\`eres semi-simples  de $\G2(F)$, \\  $F$ corps local non archim\'edien}

\author{Laure Blasco et Corinne Blondel
}
\date{\today}
\maketitle
%
%
%
\begin{resume}
On d\'efinit   strates
 semi-simples,   caract\`eres semi-simples et types semi-simples de $\G2(F)$. 
 A partir de certains de ces types, on construit des repr\'esentations supercuspidales de ce groupe.  
 

\end{resume}

\markboth{\smc L. Blasco et C. Blondel}{Caract\`eres semi-simples de $\G2(F)$}

\section*{Introduction}\label{introduction}

Gr\^ace aux travaux de C. J.  Bushnell et P. C. Kutzko, nous connaissons une description de nature arithm\'etique des repr\'esentations complexes irr\'eductibles supercuspidales des groupes lin\'eaires d\'efinis sur un corps local non archim\'edien $F$ \cite{BK1}.  
Plus r\'ecemment, S. Stevens a d\'ecrit de mani\`ere semblable les repr\'esentations  complexes irr\'eductibles supercuspidales des groupes classiques en supposant que la caract\'eristique r\'esiduelle de $F$ est diff\'erente de 2  \cite{S5}. 
Ces des\-criptions reposent sur la notion de {\it caract\`ere simple} ou {\it semi-simple}, selon le groupe consid\'er\'e, et s'effectuent principalement en trois \'etapes~: la d\'efinition des caract\`eres simples ou semi-simples et l'\'etude de leurs propri\'et\'es~; la cons\-truction de repr\'esentations irr\'eductibles du normalisateur d'un caract\`ere simple ou semi-simple contenant ce caract\`ere~; la d\'emonstration que toute repr\'esentation irr\'eductible supercus\-pidale s'obtient par induction compacte (modulo le centre) \`a partir d'une des repr\'esentations pr\'ec\'edentes.\\
L'int\'er\^et de telles descriptions est au moins double. D'une part, elles sont \'etablies par des arguments alg\'ebriques et se pr\^etent donc \`a des constructions dans le cadre plus large des repr\'esentations \`a coefficients dans un anneau o\`u  la caract\'eristique r\'esiduelle de $F$ est inversible. C'est ainsi que les carac\-t\`eres semi-simples des groupes lin\'eaires ou classiques interviennent de mani\`ere cruciale dans la d\'emonstration  du deuxi\`eme th\'eor\`eme d'adjonction con\c cue par J.-F. Dat \cite{D}. D'autre part, ces descriptions en paral\-l\`eleÊ sur diff\'erents groupes offrent l'espoir d'obtenir une bonne notion de transfert de carac\-t\`eres semi-simples qui aiderait \`a d\'ecrire certaines fonctorialit\'es.

Dans cet article, nous reprenons cette d\'emarche afin de d\'ecrire les repr\'esentations complexes supercuspidales du groupe exceptionnel $\G2(F)$. Nous r\'ealisons les deux premi\`eres \'etapes et  obtenons une s\'erie particuli\`ere de repr\'esentations irr\'eductibles de sous-groupes ouverts compacts, les {\it types semi-simples de $\G2(F)$}.
Nous concluons par l'\'enonc\'e de conditions suf\-fisantes portant sur les types semi-simples pour que leurs induites compactes \`a $\G2(F)$ soient des repr\'esentations supercuspidales (th\'eor\`eme \ref{cuspidales}). Nous esp\'erons montrer, 
dans un article ult\'erieur,  que toutes les repr\'esentations irr\'eductibles supercuspidales de $\G2(F)$ sont bien de cette forme. \\
Notons que K.-C. Yu a d\'ej\`a donn\'e une construction g\'en\'erale de repr\'esentations supercus\-pidales des groupes r\'eductifs $p$-adiques \cite{Yu} dont  Ju-Lee Kim a montr\'e l'exhaustivit\'e \cite{Kim} par des arguments analytiques. L'ensemble n\'ecessite que la caract\'eristique r\'esiduelle $p$ de $F$ soit suffisamment grande. Ici, nous supposons simplement qu'elle est diff\'erente de 2 et 3.

Le groupe $\G2(F)$ est le groupe des automorphismes d'une $F$-alg\`ebre d'octonions $V$ munie de sa norme et s'identifie ainsi \`a un sous-groupe du groupe orthogonal d\'eploy\'e $\SO_{F}(V)$. Il est aussi le groupe des points fixes de $\Spin_{F}(V)$ sous l'action d'un groupe d'automorphismes d'ordre $6$ dit groupe de trialit\'e. L'action du groupe de trialit\'e n'est pas d\'efinie sur $\SO_{F}(V)$ mais peut l'\^etre sur ses pro-$p$-sous-groupes. 
L'id\'ee est alors de construire les caract\`eres semi-simples de $\G2(F)$ \`a partir de ceux de $\SO_{F}(V)$ \`a l'aide d'une correspondance de Glauberman pour le groupe de trialit\'e (th\'eor\`eme \ref{carfixes}), tout comme S. Stevens a construit les ca\-rac\-t\`eres semi-simples des groupes classiques \`a partir de ceux du groupe lin\'eaire en utili\-sant l'automorphisme d'adjonction \cite{S4}. D\`es cette \'etape,
l'exclusion des carac\-t\'eristiques r\'esiduelles 2 et 3 s'impose. \\
La deuxi\`eme \'etape se d\'eroule parall\`element \`a \cite[\S\S  3.2 \`a 4.1]{S5}. Les quotients r\'eductifs intervenant sont tous des groupes classiques d\'efinis sur le corps r\'esiduel et nous utilisons encore une fois les r\'esultats de \cite{S5}. 
Notons tout de m\^eme que les calculs d'entrelacement exigent de nouvelles m\'ethodes puisqu'il n'existe pas pour $\G2 (F)$ une transform\'ee de Cayley aux bonnes propri\'et\'es (remarque \ref{Moy}).

Ce sch\'ema simple n\'ecessite un grand nombre de pr\'eparatifs et d\'etours que nous pr\'esentons maintenant avec plus de d\'etails dans le plan de l'article.

      La premi\`ere partie d\'efinit et \'etudie les  {\it strates semi-simples $[\Lambda, n, 0 ,\beta]$ de $\lg2 (F)$,} l'alg\`ebre de Lie de $\G2(F)$. Ce sont les strates semi-simples de $\End_F(V)$ dont la suite deÊ r\'eseaux $\Lambda$ cor\-res\-pond \`aÊ unÊ point de l'immeuble de $G_2(F)$ et l'\'el\'ementÊ $\beta$ appartient \`a $\lg2 (F)$. Ce dernier est semi-simple et une d\'erivation sur $V$ donc son noyau $V^0$ est une sous-alg\`ebre de composition de dimension paire (\ref{strates.2}, \ref{composition}). On dispose alors d'une d\'ecomposition de $V$ en somme directe de $V^0$ et de son orthogonal $W$ et, lorsque  $V^0$ est  d\'eploy\'ee de dimension 2, d'une polarisation compl\`ete de $W$, $W=W^+ \oplus W^-$. La d\'ecomposition de $V$ ainsi obtenue est stable par $\beta$ et scinde la strate $[\Lambda, n, 0 ,\beta]$ (\S \ref{strates.2}). Elle gouverne toute l'\'etude. \\ 
 Bien \'evidemment, lorsque la strate est nulle, cette d\'ecomposition est triviale. Ce cas corres\-pond au niveau z\'ero d\'ej\`a \'etudi\'e par L. Morris \cite{M2} et est oubli\'e, ou peu s'en faut, jusqu'au dernier paragraphe (\S \ref{types.cuspidales}). Supposons donc la strate $[\Lambda, n, 0 ,\beta]$ non nulle. Sa ``restriction  \`a $W$ ou $W^+$'' (ici confondus sous le nom $W'$), c'est-\`a-dire la strate $[\Lambda \cap W', n, 0, \beta_{W'}]$ o\`u $\beta_{W'}$ est la restriction de $\beta$ \`a $W'$, est une strate semi-simple  de l'alg\`ebre de Lie d'un sous-groupe $\bar L$ de $G_{2}(F)$ stabilisant $W'$. Celui-ci est isomorphe \`a $\SL(3,F)$ si $V^0$ est d\'eploy\'ee de dimension 2,  $\SU(2,1)(F)$ si $V^0$ est anisotrope de dimension 2 et $\SO(W)$ si $V^0$ est de dimension 4.  De plus, le centralisateur de $\beta$ dans $\G2(F)$ s'identifie au centralisateur de $\beta_{W'}$ dans $\bar L$ (\S  \ref{strates.3}).\\
  La classification des strates semi-simples de $\lg2(F)$ consiste alors \`a \'etudier cette  ``application de restriction \`a $W'$''. Le point crucial est la construction de suites de r\'eseaux de $V$ cor\-res\-pon\-dant \`a un point de l'immeuble de $\G2(F)$ dont  les suites de r\'eseaux de $W'$, cor\-res\-pon\-dant  \`a un point de l'immeuble de $\bar L$, sont facteurs directs (\S  \ref{strates.normes}). Mais ici, le langage  appropri\'e est celui des {\it normes} : normes de volume nul dans le cas de 
 $\SL(3,F)$ \cite{BT0}, normes autoduales dans celui d'un groupe classique \cite{BT} et normes autoduales d'alg\`ebre dans celui de $\G2(F)$ \cite{GY}. On obtient ainsi un plongement canonique de l'immeuble de $\bar L$ dans celui de $\G2(F)$ (propositions \ref{volume0sl}, \ref{volume0su}, \ref{dimension4}).  On conclut sur une classification compl\`ete des strates semi-simples de $\lg2(F)$ et un proc\'ed\'e d'approximation de telles strates (\S \ref{strates.4}).  
 
  Suit la construction des caract\`eres semi-simples de $\G2(F)$ associ\'es \`a la strate $[\Lambda, n,0,\beta]$ qui ne s'ach\`eve qu'au paragraphe \ref{par3}. Afin d'utiliser une correspondance de Glauberman, nous devons au pr\'ealable \'etudier l'action du groupe de trialit\'e $\Gamma$ sur les caract\`eres de $\SO(V)$ associ\'e \`a cette strate (\S \ref{par2}). Notons d'abord que, la strate $[\Lambda, n,0,\beta]$ de $\G2 (F)$ \'etant une strate semi-simple autoduale de $\SO(V)$ (\S \ref{strates.2}), la d\'efinition des caract\`eres semi-simples de $\SO(V)$ doit \^etre \'elargie comme dans \cite{D}. A plusieurs reprises dans l'article, on \'etend les r\'esultats de \cite{S5} \`a ces strates et leurs caract\`eres semi-simples associ\'es.\\
  Ceci dit, l'action du groupe de trialit\'e est complexe et son \'etude occupe tout le paragraphe \ref{par2}.  Une premi\`ere raison est que l'action de $\Gamma$ ne se refl\`ete point sur l'espace $V$ et il n'est plus clair que les filtrations de l'ordre ${\mathfrak A}(\Lambda)$ et du sous-groupe parahorique $P(\Lambda)$ soient stables sous cette action. Une deuxi\`eme raison est que la transform\'ee de Cayley ne commute aux actions de $\Gamma$ que ``localement''~:  la description des caract\`eres des quotients des filtrations des sous-groupes parahoriques \`a l'aide d'\'el\'ements de $\lg2(F)$ n'est pas imm\'ediate, ni le fait que les caract\`eres obtenus \`a partir d'un \'el\'ement de $\lg2(F)$ soient fixes sous l'action de $\Gamma$. Ces propri\'et\'es sont \'etablies aux paragraphes \ref{par22} et \ref{par225}. Elles assurent que les sous-groupes $H^1(\beta,\Lambda)$ et $J^1(\beta,\Lambda)$ associ\'es \`a la strate $[\Lambda,n,0,\beta]$ sont stables sous l'action de $\Gamma$ (lemme \ref{ex2.1}) et que le caract\`ere $\psi_{\beta}$ de $P^{[\frac{n}{2}]+1}(\Lambda)$ est fixe par trialit\'e. Ceci est le premier pas vers une caract\'erisation des caract\`eres semi-simples de $\SO(V)$ associ\'es \`a la strate $[\Lambda, n, 0,\beta]$ qui restent semi-simples sous l'action de $\Gamma$. Les pas suivants exigent, en outre, une \'etude de l'action de la trialit\'e sur le centralisateur de $\beta$ (\S \ref{par24}). On d\'emontre alors que les seuls caract\`eres semi-simples de $SO(V)$ associ\'es \`a $[\Lambda, n,0,\beta]$  qui restent semi-simples sous l'action de $\Gamma$ sont ceux qui sont fixes sous cette action (th\'eor\`eme \ref{carfixes}). On les nomme {\it caract\`eres semi-simples sp\'eciaux de $\SO(V)$ associ\'es \`a $[\Lambda, n,0,\beta]$}.

On est maintenant en mesure de d\'efinir les {\it caract\`eres semi-simples de $\G2(F)$} associ\'es \`a $[\Lambda, n,0,\beta]$ comme l'image par la correspondance de Glauberman pour le groupe $\Gamma$ des carac\-t\`eres semi-simples sp\'eciaux de $\SO(V)$ associ\'es \`a cette m\^eme strate. Il s'agit simplement des restrictions \`a $\oH^1(\beta,\Lambda):=H^1(\beta,\Lambda)\cap \G2(F)$ des caract\`eres semi-simples sp\'eciaux de $H^1(\beta,\Lambda)$.

Dans le paragraphe \ref{par3}, on aborde la deuxi\`eme partie de la construction des types semi-simples dont le d\'eroulement est parall\`ele au cas classique.  Les caract\`eres semi-simples de $\G2 (F)$ associ\'es \`a la strate $[\Lambda, n,0,\beta]$ jouissent des m\^emes propri\'et\'es que ceux de $\SO(V)$ et admettent une extension de Heisenberg $\bar \eta$ \`a $\oJ^1(\beta,\Lambda):=J^1(\beta,\Lambda)\cap \G2(F)$(\S \ref{par31}). A leur tour, les repr\'esentations de Heisenberg obtenues poss\`edent la  propri\'et\'e cruciale ``d'entrelacement simple''~:  la dimension des espaces d'entrelacements est toujours \'egale \`a 0 ou 1. Pour l'\'etablir, on traduit cette propri\'et\'e en termes d'\'egalit\'e entre certains sous-groupes de $\oJ^1(\beta,\Lambda)$ et de ses analogues dans $\GL(V)$ et $\SO(V)$  (proposition \ref{essentiel}) ce qui permet   de ``descendre'' les r\'esultats de $\GL(V)$ \`a $\G2(F)$ en consid\'erant les invariants sous l'action de l'adjonction puis sous celle du groupe de trialit\'e (\S \ref{par33}).\\
Il ne reste plus qu'\`a terminer la construction en \'etendant la repr\'esentation $\bar \eta$ au groupe $\oJ(\beta,\Lambda):=\oP_{\beta}(\Lambda)\oJ^1(\beta,\Lambda)$ o\`u $\oP_{\beta}(\Lambda) = P(\Lambda)\cap \oG_{\beta}$. Ceci est r\'ealis\'e au dernier paragraphe. Gr\^ace aux propri\'et\'es obtenues sur $\bar \eta$ et sur l'immeuble de $\G2(F)$ (en particulier le lemme \ref{lemme28}) et en ajoutant que le quotient $\oJ(\beta,\Lambda)/\oJ^1(\beta,\Lambda)$ est un groupe classique sur le corps r\'esiduel, la m\^eme m\'ethode que \cite{S5} conduit \`a la d\'efinition des types semi-simples de $\G2(F)$ (\S \ref{types.cuspidales}). On conclut en donnant des conditions suffisantes sur ces types semi-simples pour que leurs induites \`a $\G2(F)$ soient cuspidales (th\'eor\`eme \ref{cuspidales}).
 
 \bigskip

Cette \'etude a d\'emarr\'e \`a la suite d'un groupe de travail sur $\G2$ en 2003/2005 dont nous remercions les participants, en particulier Fran\c cois Sauvageot qui r\'ealisa \`a cette occasion une \'etude sur les tores qui, bien qu'invisible dans ce qui suit, en a inspir\'e le contenu. 
 Nous esp\'erons qu'une partie de la jubilation de ce groupe de travail et du plaisir que nous avons eu \`a mener \`a son terme la pr\'esente \'etude transpara\^it dans ce qui suit.     

\Section{Strates semi-simples de $\lg2(F)$}\label{strates}
\subsection{D\'efinitions et notations relatives \`a $\G2$}\label{strates.1}

Soit $F$ un corps local non archim\'edien de caract\'eristique r\'esiduelle $p$
diff\'erente de $2$ et $3$.
Soit $V$ l'alg\`ebre des octonions sur $F$~; on notera $\un$ son unit\'e, $Q$ sa norme, 
qui est multiplicative :  $   Q(xy)= Q(x)Q(y) \  (x,y \in V)$,  et
$f$ la forme bilin\'eaire associ\'ee, de sorte que
$ f(x,x)= 2 Q(x) \  (x  \in V)$.  
L'anti-automorphisme d'adjonction de $\mathfrak{gl}_F(V)$ sera not\'e
$X \mapsto \sigma( X)$ et l'automorphisme correspondant de $\GL_F(V)$
sera not\'e $\tau$ : $\tau(g)={\sigma( g)}^{-1}$.

Les propri\'et\'es de l'alg\`ebre d'octonions sont d\'ecrites par exemple dans \cite{SV}, nous rappelons simplement ici les points  essentiels pour fixer les notations. Le conjugu\'e d'un octonion $x$ est
$\bar x = f(x,\un) \un - x$, sa trace est $\tr \, x = x + \bar x$, sa norme est telle que
$Q(x) \un = x \bar x = \bar x x $, et l'on a pour $x, y, z \in V$ : $\overline{xy} = \bar y \bar x$ et
$f(xy,z)=f(y, \bar x z)= f(x, z \bar y)$.

On consid\'erera toujours $\G2(F)$, not\'e $\oG $, comme le sous-groupe de
$\GL_F(V)= \widetilde G$ form\'e des automorphismes d'alg\`ebre de $V$.
C'est aussi un sous-groupe du groupe $\SO_F(V)$, not\'e $G$, form\'e des \'el\'ements de d\'eterminant $1$  du groupe  $\O_F(V)=\GL_F(V)^\tau$ des isom\'etries de la forme quadratique $Q$. On notera
$\mathfrak g_2(F)$ l'alg\`ebre de Lie de $\G2(F)$, sous-alg\`ebre de Lie de $\mathfrak{s  o}_F(V)$ et de $\mathfrak{gl}_F(V)$.
 
\subsection{D\'efinitions et notations relatives aux strates}\label{strates.2}

Soit $[\L, n, r, \beta]$ une strate de $\End_F(V)$ \cite[\S 3.1]{BK2}. Supposons 
$\beta \in \End_F(V)$  
  semi-simple : le polyn\^ome minimal de $\beta$ est un  produit $\prod_{i=0}^l
\Psi_i$ de  polyn\^omes irr\'eductibles sur $F$ deux \`a deux premiers entre eux. Posons $V^i = \ker \Psi_i(\beta)$.
Cela d\'efinit une   d\'ecomposition de $V$ en somme directe
$V = \oplus_{i=0}^l V^i$, unique \`a l'ordre pr\`es, telle que
$\beta=\sum_{i=0}^l \beta_i$ o\`u $\beta_i$ est la restriction de
$\beta$ \`a $V^i$.

Par d\'efinition  \cite[Definition 3.2]{S4}, la  strate  $[\L, n, r, \beta]$ est   {\it semi-simple} dans $\End_F(V)$  si 
\begin{itemize}
\item
  $\L = \oplus_{i=0}^l \L^i \ $ o\`u $ \ \L^i(t) = \L(t)\cap V^i \ $  ($t \in \mathbb Z$) ;  
  \item  pour $0 \le i \le l$ la strate $[\L^i, n_i, r, \beta_i]$ est simple ou nulle, avec 
  $n_i=r$ si $\b_i =0$, $n_i = - v_{\L^i}(\b_i)$ sinon ;
  \item   pour $0 \le i , j \le l$, $i \ne j$,  la strate $[\L^i\oplus \L^j, \max\{n_i,n_j\}, r, \beta_i+ \beta_j]$ n'est pas \'equivalente \`a une strate simple ou nulle.
    \end{itemize}

On dit qu'une strate $[\L, n, r, \beta]$ est une strate de $\mathfrak{so}_F(V)$, ou   {\it strate  autoduale},   si la suite de r\'eseaux  $\L$ correspond \`a un point  rationnel de l'immeuble de $\SO_F(V)$ (voir \cite{BS} et \S \ref{strates.normes})
et si $\beta$ appartient \`a $\mathfrak{so}_F(V)$.  Soit alors  $[\L, n, r, \beta]$  une    {\it strate  semi-simple autoduale}. 
Le   polyn\^ome minimal de $\beta $ est alors pair et, quitte \`a renum\'eroter, 
on peut   r\'epartir les polyn\^omes irr\'eductibles $\Psi_i$ en deux sous-ensembles 
v\'erifiant~:
\begin{itemize}
    \item pour $0\le i\le s$ le polyn\^ome $\Psi_i$ est pair.
    \item pour $1\le j \le k$ on a l'\'egalit\'e $\Psi_{s + 2j-1}(-X)= \Psi_{s + 2j}(X)$.
\end{itemize}
Chaque sous-espace  $V^i$ pour $i \le s$ est non d\'eg\'en\'er\'e  et orthogonal \`a tous les autres.
Pour $1\le j \le k$  les sous-espaces
$V^{s + 2j-1} $ et $V^{s + 2j} $ sont totalement isotropes en dualit\'e et orthogonaux aux autres, ce qui nous donne la d\'ecomposition
\begin{equation}\label{decomposition}
 V = \left[\perp_{i=0}^s V^i \right] \perp
\left[\perp_{j=1}^k ( V^{s + 2j-1} \oplus V^{s + 2j})\right].
\end{equation}
Rappelons comme en \cite[\S 8.2]{D}   que cette d\'efinition est plus large que  celle de strate semi-simple gauche dans  \cite{S5}, qui correspond \`a  $k=0$ dans la somme pr\'ec\'edente.

On dit enfin qu'une telle strate est une {\it strate semi-simple de $\mathfrak g_2(F)$ } si  $\L$ correspond \`a un point rationnel de l'immeuble de $\G2(F)$ (voir \cite{GY} et  \S \ref{strates.normes})
et si $\beta$ appartient \`a $\mathfrak g_2(F)$.
Dans toute la suite ces conditions sont suppos\'ees v\'erifi\'ees.

Soit donc $\beta$ un \'el\'ement semi-simple de  $\mathfrak g_2(F)$. Alors $\beta$
est  une d\'erivation de $V$, donc s'annule en $\un$.  Notons $V^0$ l'espace propre associ\'e \`a la valeur propre $0$ de $\beta$ ; il contient $\un$  et tous les produits de deux de ses \'el\'ements~:
$$\forall x,y\in V^0, \quad \beta(xy)=\beta(x)y+x\beta(y)=0.$$
De plus, puisque $\beta$ est aussi un \'el\'ement semi-simple de $\mathfrak{so}_F(V)$, le sous-espace $V^0$ est non isotrope et de dimension paire.
 C'est donc une sous-alg\`ebre de composition de dimension paire de  $V$.

\subsection{Sous-alg\`ebres de composition}\label{composition}

Un bref rappel sur les sous-alg\`ebres de composition de $V$ s'impose ici ;
  les faits cit\'es se trouvent soit dans  la r\'ef\'erence de base   \cite[Ch. 1 et 2]{SV},
soit dans \cite[\S 1]{RS} ou \cite[\S 8]{GY}.
Commen\c cons par le proc\'ed\'e de doublement, qui est fondamental et fournit une r\`egle de calcul largement utilis\'ee dans tout ce paragraphe \ref{strates}.

\begin{Proposition}{\bf \cite[Proposition 1.5.1]{SV}}\label{double}
Soit $C$ une alg\`ebre de composition et $D$ une sous-alg\`ebre de composition de $C$, propre et de dimension finie. Alors pour tout $a \in D^\perp$ de norme non nulle la somme
$D \perp Da$ est une sous-alg\`ebre de composition de $C$ dans laquelle le produit est donn\'e par:
$$
(x+ya) (u+va) = (xu - Q(a) \bar v y ) + (vx+y\bar u) a \qquad (x, y, u ,v \in D).
$$
\end{Proposition}
On remarque sur cette formule que la multiplication \`a gauche (resp. \`a droite) par $D$ sur $Da$ v\'erifie 
$u(vx) = (uv) x$ (resp.  $(xu)v=x(uv)$) pour tous $x \in Da$ et $u,v \in D$ si et seulement si $D$ est commutative, auquel cas $Da$ est un bimodule sur $D$.

Une  sous-alg\`ebre de composition $D$ de $V$ distincte du centre $F \un$ de $V$ et de $V$ 
elle-m\^eme est de dimension   $2$ ou $4$ ; sa description  d\'epend de la restriction de $Q$ \`a $D$.

\begin{enumerate}
    \item Si $D$ est de  dimension $2$ et contient des vecteurs isotropes pour $Q$, elle peut s'\'ecrire sous la forme $D = F \un \oplus F a$ pour un \'el\'ement
    $a$ v\'erifiant $Q(a)=-1$ et $f(a, \un) = 0$.
Alors $e_D^+ = \frac 1 2 (\un + a )  $ et $e_D^- = \frac 1 2 (\un - a )  $
forment la seule paire d'idempotents de $D$ v\'erifiant $e_D^+ e_D^- =0$. On a  
$Q(e_D^+)= Q(e_D^-)=0 $, $\overline{e_D^+}=e_D^-$,  
 $e_D^+ + e_D^- = \un$   et $f(e_D^+, e_D^-)=1$.

Soit $W = D^\perp$. 
Les sous-espaces $W_D^+= e_D^+ W$ et $W_D^-=e_D^-W$ forment une polarisation compl\`ete de $W$. On a $W_D^+ =  \{ x\in V \vert  e_D^+\cdot x=x \text{ et } x\cdot e_D^+=0\}$.   

Le groupe $\SL_F(W_D^+) \simeq \SL(3,F)$ s'injecte naturellement dans $\oG$
(l'action sur $W_D^-$ \'etant donn\'ee par dualit\'e et l'action sur $D$ triviale);  son image est le
 fixateur (point par point) de $D$  dans $\oG$.

 \begin{Notation}\label{betachapeau1}
 Une telle sous-alg\`ebre $D$ \'etant fix\'ee,
 on note  $\phi \mapsto \check\phi$ l'injection naturelle de
 $\mathfrak{sl}(W_D^+) \simeq \mathfrak{sl}(3,F)$ dans  $\mathfrak g_2(F)$ ;  $\check\phi$ agit par $0$ sur $D$ et par l'oppos\'e du transpos\'e de $\phi$ sur $W_D^-$.
 \end{Notation}

\item Si $D$ est anisotrope de  dimension $2$, c'est une extension quadratique de $F$
de la forme  $D = F \un \oplus F a$,  pour un \'el\'ement
    $a$ orthogonal \`a $\un$ tel que $-Q(a) $ n'est pas un carr\'e.
    Alors $D^\perp$ est un  $D$-espace vectoriel de dimension $3$
    que $Q$ munit d'une structure hermitienne sur $D$.
    Le groupe $\SU(D^\perp) \simeq \SU(2,1)(D/F)$ s'injecte naturellement dans $\oG$
    (via l'action triviale sur $D$) ;  son image est
    le fixateur de $D$  dans $\oG$.

 \begin{Notation}\label{betachapeau2}
 Une telle sous-alg\`ebre $D$ \'etant fix\'ee,
 on note $\phi \mapsto  \vec\phi$ l'injection naturelle de
 $\su(D^\perp) \simeq \mathfrak{su}(2,1)(D/F)$
  dans  $\mathfrak g_2(F)$ ;  $\vec\phi$ agit par $0$ sur $D$.
 \end{Notation}

\item Si $D$ est de  dimension $4$, $D$ est une alg\`ebre de quaternions. Si elle contient des \'el\'ements isotropes, elle est isomorphe  \`a l'alg\`ebre des matrices $2\times 2$ \`a coefficients dans $F$ et  la restriction de $Q$ \`a $D$ co\"\i ncide avec le d\'eterminant. Sinon, $D$ est un corps de quaternions et la restriction de $Q$ \`a $D$ co\" \i ncide avec sa norme. 

Dans les deux cas, notons $D^1$ le groupe des \'el\'ements de norme 1 de $D$. D'apr\`es \cite[57.13]{OM} toute rotation de $D$
fixant $\un$ est de la forme $x \mapsto cxc^{-1}$ pour un $c \in D^\times$, de sorte que 
le groupe $\SO_F(D)$ des automorphimes sp\'eciaux orthogonaux de $D$ est form\'e des applications  
  $x \mapsto h(x) = p cxc^{-1}$ ($x \in D$) avec $c \in D^\times$ et $p = h(1) \in D^1$. 

Pour tout choix de $a\in D^\perp$ de norme $Q(a) $ non nulle,  $V$ est somme orthogonale de $D$ et $Da$.
Un \'el\'ement  $g$ du stabilisateur $\oG_{D}$ de $D$ dans $\oG$ est d\'ecrit  par deux \'el\'ements   de $D$, $c$ de norme non nulle  et $p$ de norme 1, de la fa\c con suivante \cite[(2.2)]{SV}~:
\begin{equation*}
 \forall x, y\in D, \quad g(x+ya)=cxc^{-1} +(pcyc^{-1})a.
\end{equation*}
Puisque la  multiplication \`a droite par $a$ d\'efinit une similitude de rapport $Q(a)$ du $F$-espace vectoriel $D$ dans le $F$-espace vectoriel $Da$ qui induit un isomorphisme de $SO_{F}(D)$ dans $SO_{F}(Da)$, on conclut que la restriction au sous-espace $W$ identifie 
$\bar G_{D}$ au groupe $SO_{F}(D^\perp)$.
\end{enumerate}

Il est bien connu que 
les deux classes de conjugaison de sous-groupes de Levi maximaux de 
 $G$ consistent en des sous-groupes 
 $\GL(2,F)$ associ\'es \`a des racines courtes pour l'une, longues pour l'autre.  On peut les identifier  en termes de 
l'\'etude qui pr\'ec\`ede. Soit en effet $D$ une sous-alg\`ebre de composition d\'eploy\'ee de dimension $4$ de $V$. Toute polarisation compl\`ete de $D^\perp$, soit 
$D^\perp= W \oplus W'$, donne lieu \`a une injection naturelle de $\GL_F(W)$ dans 
$SO_{F}(D^\perp)$ donc dans $\bar G_{D}$. Il n'est pas difficile de voir, par exemple en examinant les preuves des propositions  \ref{volume0sl} et \ref{dimension4}, que l'on obtient ainsi la famille des ``$\GL(2)$ courts'' si le produit $W.W$ est nul, celle des ``$\GL(2)$ longs'' si le produit $W.W$ est non nul. Lorsque  $W.W$ est non nul, l'orthogonal de 
$WW \oplus W'W'$ dans $D$ est une sous-alg\`ebre de composition d\'eploy\'ee $X$ de dimension $2$ de $V$ et  $\GL_F(W)$ est un sous-groupe de Levi de  $\SL_F(W_X^+)$. 

Pour terminer ce paragraphe, donnons une construction \'el\'ementaire dont nous aurons besoin \`a plusieurs reprises. 
\begin{Lemma}\label{idempotents}
Soient $L$ et $L'$ deux droites isotropes de $V$, orthogonales \`a $\un$,  de  g\'en\'erateurs $h$ et $h'$    v\'erifiant $f(h,h')=1$. Les \'el\'ements $h-h'$ et $h+h'$ sont de norme $-1$ et $ 1$  respectivement et on a : 
$$
\begin{aligned}
 &(h+h') (h-h')= (h+h') (h-h')^{-1} = h'h-h h' \\  \text{ avec }   &-h'h-h h' =\un, \ 
 (-h'h)^2=  -h'h, \             (-h h')^2=-h h' . 
 \end{aligned}
$$
Ces \'el\'ements ne d\'ependent pas du choix des g\'en\'erateurs $h$ et $h'$, on notera donc 
$$e^+(L, L')= -hh',    \  e^-(L, L')= -h' h , \  c(L, L')=  (h+h') (h-h')=  e^+(L, L')-e^-(L, L').$$ 
On a $c(L, L')h=h, c(L, L')h'=-h'$. 
Le sous-espace 
$  F[c(L, L')]$ est 
  une sous-alg\`ebre de composition de dimension $2$ de $V$  dont la paire d'idempotents standard est 
$(e^+(L, L'),e^-(L, L'))$. Il est orthogonal \`a $L \oplus L'$ et leur somme est une sous-alg\`ebre de composition de dimension $4$ de $V$. 
\end{Lemma}
\begin{proof}
Il s'agit de calculs standard dans l'alg\`ebre $V$, utilisant les formules de 
\cite[\S 1]{SV}. De fait $h^2 = h'^2=0$, $hh'+h'h+f(h,h')\un =0$ et 
$(h'h)(hh')= h'(h^2h')=0$. 
\end{proof}

\subsection{Les \'el\'ements semi-simples de $\mathfrak{g}_{2}(F)$}\label{strates.3}

Soit $\beta$ un \'el\'ement semi-simple non nul de $\mathfrak{g}_{2}(F)$ et $V^0$
son noyau, sous-alg\`ebre  de composition de dimension $2$ ou $4$ de  $V$. On note $W$ l'orthogonal de $V^0$ ; il est stable par multiplication  \`a droite et \`a gauche par  $V^0$, et lorsque $V^0$ est commutative c'est   un bimodule sur $V^0$. La d\'ecomposition  $V=V^0\perp W$ est stable par $\beta$   et la restriction de $\beta$ \`a $W$ induit un endomorphisme de $W$ not\'e $\beta_{W}$.
La d\'erivation $\beta$ commute aux multiplications  \`a droite et \`a gauche par  $V^0$ ; en particulier, si $V^0$ est de dimension $2$ 
elle est $V^0$-lin\'eaire
  :
\begin{equation*}\begin{aligned}
\forall w\in W, v,v'\in V^0, &\quad  f(v' w,v)=f(w, \bar{v'} v)=0 \text{ et } f(wv',v)=f(w,v\bar v')=0, \\
&\quad  \beta(vw)=v\beta(w) \text{ et } \beta(wv)=\beta(w)v.
\end{aligned}
\end{equation*}

Notons    $\widetilde G_\beta$  le centralisateur de $\beta$ dans $\GL_F(V)$ puis
$G_\beta  $, $G_\beta^\prime  $ et $\bar G_\beta   $ ses centralisateurs dans $\O_F(V)$, $\O^\prime_F(V)$ et $\G2(F)$ respectivement.

\begin{Lemma}\label{dim2}
 On suppose que $V^0$ est de dimension $2$.
\begin{enumerate}
    \item
  Si $V^0$ est un plan hyperbolique, les sous-espaces $W^\pm = W_{V^0}^\pm$
sont  stables par $\beta$. La restriction $\beta_{W^+}$ de $\beta$ \`a $W^+$ est un \'el\'ement de $\mathfrak{sl}(W^+) \simeq \mathfrak{sl}(3,F)$ et $\bar G_{\beta}$ est le centra\-li\-sateur de $\beta_{W^+}$ dans $\SL_F(W^+) \simeq \SL(3,F)$.
\item Si  $V^0 $ est une extension quadratique $F' $ de $F$,   la restriction $\beta_{W}$ appartient \`a l'alg\`ebre de Lie de $\SU(W)$, isomorphe \`a  $\mathfrak{su}(2,1)(F^\prime/F)$, et $\bar G_{\beta}$ est le centralisateur de $\beta_{W}$ dans $\SU(W)\simeq \SU(2,1)(F'/F)$.
\end{enumerate}
\end{Lemma}

\begin{proof}
Dans le cas (i), les sous-espaces $W^\pm$ sont   stables par $\beta$ qui est $V^0$-lin\'eaire,  et  $\beta_{W^+}$ est un \'el\'ement de $\msl(W^+)$.
Tout \'el\'ement $g$ de $\bar G_{\beta}$ stabilise $V^0$ et $W$. La restriction de $g$ \`a $V^0$ est un automorphisme d'alg\`ebre donc $g$ fixe $e_{V^0}^+$ et $e_{V^0}^-$ ou les \'echange. Mais dans le deuxi\`eme cas, $g$ \'echangerait les sous-espaces vectoriels $W^+$ et $W^-$ ce qui entra\^\i nerait que $\beta_{W^+}$ et son oppos\'e soient conjugu\'es, ce qui est impossible en dimension impaire.\\
Par suite, tout \'el\'ement de $\bar G_{\beta}$ est l'identit\'e sur $V^0$ et stabilise $W^+$ et $W^-$. L'application de restriction \`a $W^+$ d\'efinit un plongement de $\bar G_{\beta}$ dans $\SL_F(W^+) $  dont l'image est le  centralisateur de $\beta_{W^+}$.

Dans le cas (ii), l'\'el\'ement $\beta_W$ appartient \`a $\su(2,1)(F'/F)$ et le sous-groupe des automorphismes de $V$ triviaux sur $V^0$ est isomorphe \`a $\SU(2,1)(F'/F)$ via l'application de restriction \`a $W$  (\S \ref{composition}). Pour finir la d\'emonstration, il suffit de montrer que tout \'el\'ement de $\bar G_{\beta}$ est trivial sur $V^0$. \\
Pour cela, on consid\`ere la $F'$-alg\`ebre $V'=F'\otimes_{F}V$ dont le groupe d'automorphismes $\bar G'$ est $\G2 (F')$. L'\'el\'ement $\beta'=1\otimes \beta$ appartient \`a $\Lie \bar G'$ et son noyau est ${V'}^0=F'\otimes_{F}V^0$ qui est maintenant une sous-alg\`ebre d\'eploy\'ee de dimension 2. Par (i), tous les \'el\'ements de $\bar G'_{\beta'}$ sont triviaux sur ${V'}^0$ donc en particulier, tous les \'el\'ements de $1\otimes \bar G_{\beta}$. On en d\'eduit que  tout \'el\'ement de $\bar G_{\beta}$ est trivial sur $V^0$.
\end{proof}

\begin{Lemma}\label{dim4} On suppose que $V^0$ est de dimension $4$. Alors $\beta_{W}$ est un endomorphisme de $W$ dont le polyn\^ome minimal est de la forme $X^2-u$, $u\in F^\times$.
\begin{enumerate}
	\item Si le polyn\^ome minimal de $\beta_{W}$ est scind\'e, il poss\`ede deux racines oppos\'ees $\lambda$ et $-\lambda$. Les deux sous-espaces propres de $\beta_{W}$, not\'es $W_\lambda$ et $W_{-\lambda}$, d\'efinissent une polarisation compl\`ete de $W$. Le groupe $\bar G_{\beta}$ s'identifie par la restriction \`a $W_\lambda$ \`a $GL(W_\lambda)\simeq \GL(2,F)$. Il est en fait un sous-groupe de Levi de $\bar G$ attach\'e \`a une racine courte. Dans ce cas, l'alg\`ebre $V^0$ est n\'ecessairement d\'eploy\'ee.
	\item  Si le polyn\^ome minimal de $\beta_{W}$ est irr\'eductible, $\beta_{W}$ engendre  une extension quadratique $F[\beta_{W}]$ de $F$ et $W$ est un $F[\beta_{W}]$-espace vectoriel muni de la forme hermitienne $\Phi_{W}$ v\'erifiant~: $\tr_{F[\beta_{W}]/F} \circ \Phi_{W}=f_{\vert W\times W}$. Ainsi,  $\beta_W$ est un \'el\'ement scalaire de ${\mathfrak u}(\Phi_{W}, W)$. Le centra\-li\-sateur $\bar G_{\beta}$ de $\beta$ s'identifie par restriction \`a $W$  au groupe $\U(\Phi_{W},W)$, groupe unitaire anisotrope ou quasi-d\'eploy\'e selon que $V^0$ est anisotrope ou d\'eploy\'ee.
\end{enumerate}
\end{Lemma}

\begin{proof} Soit $a$ un \'el\'ement  de norme non nulle  dans l'orthogonal de $V^0$. D'apr\`es   la d\'emonstration du lemme 2.4.3 de \cite{SV}, il existe un \'el\'ement $c$ de $  V^0$, de carr\'e non nul et de trace nulle tel que : $ \  \beta(v+v'a)=v'(c a)=(c v')a \  $ pour tous $v,v'\in V^0$.  Il s'ensuit que  $\beta_{W}$ n'est pas scalaire et $\beta_W^2$ est la multiplication par $c^2$, non nul puisque   $\beta$ est   semi-simple non nul. Le polyn\^ome minimal de $\beta_{W}$ est donc de la forme indiqu\'ee.

Le centralisateur $\oG_{\beta}$ de $\beta$ s'identifie, via la restriction \`a $W$, au sous-groupe des \'el\'ements de $SO_{F}(W)$ qui commutent \`a $\beta_{W}$ (\ref{composition}(iii)).

Dans le cas  (ii), $\oG_{\beta}$ s'identifie donc au groupe des automorphismes $F[\beta_{W}]$-lin\'eaires de $W$ qui conservent la forme hermitienne $\Phi_{W}$, c'est-\`a-dire \`a $U(\Phi_{W}, W)$.

 Dans le cas   (i),   $\beta_W$ est diagonalisable de valeurs propres $\lambda$ et $-\lambda$, chacune de multiplicit\'e $2$ puisque $\beta_{W}$ est de trace nulle. Si $e^+_{F[c]}$ et $e^-_{F[c]}$ sont les idempotents orthogonaux de $F[c]$ (\S \ref{composition}),  les sous-espaces propres 
de $\beta_W$ sont   $W_\lambda=(e^+_{F[c]}V^0)a$ et  $W_{-\lambda}=(e^-_{F[c]}V^0)a$ et forment une polarisation compl\`ete de $W$. Le groupe  $\oG_\beta$ s'identifie alors \`a $\GL(W_\lambda) \simeq GL(2,F)$ par la restriction \`a $W_\lambda$. Pour voir qu'il s'agit d'un Levi attach\'e \`a une racine courte on peut s'appuyer sur le paragraphe \ref{composition} et remarquer soit que $W_\lambda W_\lambda$ est nul, car $2\lambda$ n'est pas valeur propre de $\beta$, soit que $W_\lambda $ et  $W_{-\lambda}$ 
 se d\'ecomposent non trivialement en somme directe de leurs intersections avec $W^+_{F[c]}$ et  $W^+_{F[c]}$. 

On peut aussi identifier 
  $\bar G_\beta$   au centralisateur du tore de rang 1 de $\GL(W_\lambda)$ form\'e des automorphismes scalaires. Ce dernier est l'image par la restriction \`a $W_\lambda$ d'un tore $\bar T$ de rang 1 de $\bar G$ contenu dans $\oG_{\beta}$ dont les \'el\'ements agissent trivialement sur $V^0$ et dont le centra\-li\-sateur est pr\'ecis\'ement $\bar G_{\beta}$. A conjugaison pr\`es dans $\bar G$ (par un \'el\'ement de $\bar G$ envoyant $c$ sur un multiple de $X_0$ dans les notations de \cite{RS}, voir  {\it loc. cit.} Theorem 1) on reconna\^\i t le tore standard de $\bar G$ dont le centralisateur est facteur de Levi d'un sous-groupe parabolique maximal de $\bar G$ attach\'e \`a une racine courte \cite[\S 1.4]{RS}.  
\end{proof}

{\bf Remarque.}
 
Reprenons les notations du lemme \ref{dim2} et  consid\'erons le cas o\`u $\beta_{W^+}$ (i)  ou $\beta_W$  (ii) est non nul mais a une valeur propre nulle. 
	Alors  le sous-espace propre correspondant est forc\'ement de dimension $1$ (sur $F$  pour le cas (i), sur $F'$   pour le cas (ii))  et l'on obtient dans les deux cas un sous-espace $W_0$  de $W$, de dimension $2$ sur $F$, sur lequel $\beta_W$ est nul. Ce sous-espace est non d\'eg\'en\'er\'e et orthogonal \`a $V^0$ : par le proc\'ed\'e de doublement, 
	$V^0 \perp W_0$ est une sous-alg\`ebre de composition. Les cas \'etudi\'es dans le lemme \ref{dim4} peuvent ainsi \^etre consid\'er\'es comme des cas limites de ceux du lemme \ref{dim2}, on y reviendra plus loin. On notera toutefois la diff\'erence importante dans l'\'etude du centralisateur.

\subsection{Propri\'et\'es des \'el\'ements semi-simples de $\mathfrak g_2(F)$}\label{strates.35} 

De temps \`a autre dans ce travail, on a besoin d'un argument au cas par cas ; apr\`es tout $G_2$ est un groupe exceptionnel. Nous avons choisi de r\'eunir dans ce paragraphe des propri\'et\'es tr\`es pr\'ecises dont l'utilit\'e appara\^\i tra plus loin. On garde les notations du paragraphe pr\'ec\'edent.

 \begin{Lemma}\label{central}  
Supposons   $V^0$   anisotrope  de dimension $2$. 
Le centralisateur de $\beta$ dans $\O_F(V)$ est contenu dans    
   $\O_F(V^0)   \times U(W, V^0/F)$.
   \end{Lemma} 
\begin{proof}  Le centralisateur de $\beta$ dans $\O_F(V)$ est  produit de 
   $\O_F(V^0)$ par le centra\-lisateur de $\beta_W$ dans $\O_F(W)$ : il s'agit de 
  v\'erifier que   
l'alg\`ebre $F[\beta_W]$ contient $V^0$. Le fait que $\beta_W$, \'el\'ement  
semi-simple de    $\mathfrak{su} (W, V^0/F)$,
 soit injectif et de trace nulle sur $V^0$  intervient de fa\c con essentielle. 
  Soit $P_{V^0}$ le polyn\^ome minimal de $\beta_W$ sur $V^0$.  

Si  $P_{V^0}$ est irr\'eductible,  
$\beta_W$  engendre une extension de $V^0$ de degr\'e   $3$ (le degr\'e $1$ est impossible car $\beta_W$ est de trace nulle sur $V^0$ et $p \ne 3$), soit $[V^0[\beta_W] : F]=6$.    
Son polyn\^ome minimal sur $F$ est   multiple de $P_{V^0}$ et de degr\'e pair car l'involution adjointe agit non trivialement sur $F[\beta_W]$. Il est donc 
  de degr\'e~$6$ et $V^0[\beta_W]= F[\beta_W]$. 
 
Si  $P_{V^0}$  est produit de  trois facteurs de degr\'e $1$,  de 
   racines $\beta_1$, $ \beta_2 $, $ \beta_3 \in V^0$,    il faut voir que $\beta_i \notin F$.  Si un vecteur   propre correspondant \`a $\beta_i$ est anisotrope, c'est le cas 
  car $\sigma(\beta_i) = -\beta_i$. 
 Sinon les droites propres de $\beta_1$ et $\beta_2$ (par exemple) sont isotropes en dualit\'e et 
 on ne peut avoir $\beta_1 \in F$ sinon $\beta_2 = -\beta_1$ et par trace  $\beta_3=0$. 
 C'est le m\^eme chose si  $P_{V^0}$  est produit de deux facteurs de degr\'e $1$. 
   
   Si  $P_{V^0}$  a un facteur  irr\'eductible $Q$ de degr\'e $2$ et une racine $\beta_1 \in V^0$, la droite propre pour  $\beta_1$ est anisotrope et $\beta_1 \notin F$. Soit $\beta_2$ la restriction de $\beta_W$ au noyau de $Q(\beta_W)$. Son polyn\^ome minimal sur $F$ est multiple de $Q$ mais ne peut \^etre \'egal \`a $Q$ car $\tr_{V^0}\beta_2 = - \beta_1 \notin F$, on a donc  $ F[\beta_2] = V^0[\beta_2]$.   
\end{proof} 

   \begin{Lemma}\label{centralGL}  
Supposons   $V^0$  d\'eploy\'ee  de dimension $2$. Soit  $W = W^+ \oplus W^-$ la 
polarisation compl\`ete canonique de $W$ et soit   $\iota$ l'injection correspondante de $\GL_F(W^+)$ dans $\SO_F(W)$. 
  Le centralisateur de $\beta$ dans $\SO_F(V)$ est contenu dans    
   $\SO_F(V^0) \times \iota(\GL_F(W^+))$.
\end{Lemma} 
\begin{proof} Le centralisateur de $\beta$ dans $\O_F(V)$  est produit de $\O_F(V^0)$ et du centra\-lisateur de $\beta_W$ dans $\O_F(W)$. 
  Un \'el\'ement du centralisateur de $\beta_W$ stabilise les sous-espaces propres de $\beta_W$, il suffit donc de v\'erifier que chacun d'entre eux est enti\`erement contenu soit dans $W^+$ soit dans $W^-$, ou encore que les polyn\^omes minimaux $P^+$ et $P ^-$ de $\beta_{W^+}$ et $\beta_{W^-}$, restrictions de $\beta$ \`a 
  $W^+$ et $W^-$ respectivement,  sont premiers entre eux. L\`a encore, le fait que $\beta_{W^+}$ soit injectif et de trace nulle est essentiel. 
  Comme $\beta_{W^-} = - ^t \beta_{W^+}$ on a $P^-(X)= \pm P^+ (-X)$. 
  
  Si $P^+$ est irr\'eductible il est de degr\'e $3$, car la trace d'un scalaire non nul est non nulle 
  ($p \ne 3$), et $P^-(X)= - P^+ (-X)$ est irr\'eductible et distinct de $P^+$ (les traces de 
   $\beta_{W^-} $ et $ \beta_{W^+}$ sont oppos\'ees), ils sont donc premiers entre eux. 
  
  Si $P^+$ a deux ou trois facteurs irr\'eductibles de degr\'e $1$, de racines $\lambda_1$, $\lambda_2$, $\lambda_3$, les racines de $P^-$ sont les oppos\'ees des pr\'ec\'edentes 
  et on ne peut avoir d'\'egalit\'e $\lambda_i =-\lambda_j$ car les $\lambda_i$ sont non nuls et de somme nulle. 
  
  Enfin, si $P^+(X)=(X-\lambda)Q(X)$ avec $\lambda \in F $ et $Q$ irr\'eductible de degr\'e $2$, 
  on a $P^-(X)=(X+\lambda) Q(-X)$. Il est premier  \`a $P^+$ car $\lambda$ est non nul et 
  le c\oe fficient de $X$ dans $Q$ est $ \lambda$.  
\end{proof}

   \begin{Lemma}\label{centraldim4}  
   Soit $[\Lambda, n, r ,\beta]$, $0 \le r < n$, une strate semi-simple non nulle de $\mathfrak{g}_{2}(F)$.  
   Dans les cas suivants : 
   \begin{description}
\item[--]    $V^0$ est  de dimension $4$ ;  
\item[--]   $V^0$  est anisotrope de dimension $2$ et $\beta_{W}$ engendre une extension cubique de $V^0$; 
\item[--]   $V^0$  est d\'eploy\'ee de dimension $2$ et $\beta_{W^+}$ engendre une extension cubique de $F$ ;   
\end{description}
 la strate  $[\Lambda, n, n-1 ,\beta]$ est semi-simple \'egalement.  
\end{Lemma} 
\begin{proof} Si  $V^0$ est  de dimension $4$,  nous  reprenons les notations du lemme 
\ref{dim4}. Si $u$ est un carr\'e, alors  $-n$ est la  valuation de $\lambda$ et les polyn\^omes caract\'eristiques des strates simples $[\L\cap W_\lambda,n,n-1, \lambda]$ et 
$[\L\cap W_{-\lambda},n,n-1, -\lambda]$ sont premiers entre eux. Si $u$ n'est pas un 
carr\'e, $\beta_W$ est un \'el\'ement de trace nulle dans l'extension quadratique de $F$ qu'il engendre, il est donc   {\it minimal } sur $F$ au sens de \cite[(1.4.14)]{BK1} et la strate 
$[\L\cap W,n,n-1, \beta_W]$ est simple par  \cite[(1.4.15)]{BK1}. 

Supposons   $V^0$  de dimension $2$.  Une extension cubique $E$ de $F$ est ou bien totalement ramifi\'ee ou bien non ramifi\'ee et on v\'erifie imm\'ediatement que dans les deux cas, pour  $p\ne 3$, 
les \'el\'ements de trace nulle de $E$ sont minimaux sur $F$. Si $V^0$  est anisotrope, 
$\beta_{W}$ est minimal sur $V^0$,  
 la strate  $[\L_W,n,n-1, \beta_W]$ est donc simple sur $V^0$ et a fortiori sur $F$. Si 
 $V^0$  est d\'eploy\'ee, 
 $\beta_{W^+}$   
est  minimal sur $F$ et la strate $[\L_{W^+},n,n-1, \beta_{W^+}]$ est simple. Par un raisonnement semblable \`a celui du lemme pr\'ec\'edent on constate que
 le polyn\^ome caract\'eristique de cette strate est premier \`a celui de
  $[\L_{W^-},n,n-1, \beta_{W^-}]$, leur somme est donc semi-simple.

\end{proof}

\subsection{Normes d'alg\`ebre autoduales}\label{strates.normes}

Un point rationnel de l'immeuble \'etendu de $\GL_F(V)$ peut \^etre vu
sous diff\'erents aspects~: comme une suite de r\'eseaux de $V$, comme une fonction de r\'eseaux de $V$ \`a points de discontinuit\'e rationnels, ou comme une {\it norme} \`a valeurs rationnelles
sur $V$. Les d\'efinitions et le passage d'un  langage \`a un autre sont d\'etaill\'es dans \cite[I]{BL}. De m\^eme, un point rationnel de l'immeuble de $\SO_F(V)$ peut \^etre vu
  comme une suite de r\'eseaux autoduale de $V$, comme une fonction de r\'eseaux autoduale de $V$ \`a points de discontinuit\'e rationnels, ou comme une {\it  norme  autoduale}  \`a valeurs rationnelles
sur $V$  : voir \cite[\S 3]{BS}.
Consid\'erant l'immeuble de  $\G2(F)$ comme sous-ensemble simplicial de celui de $\SO_F(V)$,
le langage le plus adapt\'e est celui des  normes : les points rationnels de l'immeuble de $\G2(F)$ sont alors les {\it   normes d'alg\`ebre autoduales} de $V$ \cite[Proposition 4.2]{GY}.  Pour \'etablir dans la suite les propri\'et\'es dont nous avons besoin, nous passons constamment d'un langage \`a l'autre.

   Rappelons maintenant les d\'efinitions essentielles. Une {\it norme } (ou $F$-norme)  rationnelle sur $V$ est une application $\alpha$ de $V$ dans $\mathbb Q \cup \infty$ v\'erifiant pour tous $x, y$ dans $V$ et $\xi$ dans $F$ :
$$\alpha(x+y) \ge \inf(\alpha(x), \alpha(y))  ; \ \
    \alpha(\xi x) = v_F(\xi) + \alpha(x)   ; \  \
   \alpha(x) = \infty \iff x = 0.
$$
   Si $V$ est muni d'une forme quadratique ou hermitienne non d\'eg\'en\'er\'ee $f$,  la {\it norme duale} d'une norme $\alpha$ sur $V$ est la norme $\alpha^\ast$ sur $V$ d\'efinie par
  $$\alpha^\ast (v) = \inf_{x \in V}(v_F(f(v,x)) - \alpha(x))  \quad (v \in V).
  $$
   La norme $\alpha$ est dite {\it autoduale} si $\alpha=\alpha^\ast$. Enfin, c'est une {\it norme d'alg\`ebre} si elle v\'erifie $$\alpha(xy) \ge \alpha(x) + \alpha(y) \  \text{ pour tous }   x, y \in V.$$  Une norme d'alg\`ebre autoduale de l'alg\`ebre d'octonions $V$ v\'erifie toujours
  $\alpha(\un)=0$ et $\alpha(\bar x )= \alpha(x)$ ($x \in V$) \cite[Corollary 7.5]{GY}.

  Soit $\alpha$ une norme  rationnelle sur $V$. La fonction de r\'eseaux correspondante est la fonction $\lambda_\alpha$ d\'efinie par $\lambda_\alpha(r) = \{v \in V / \ \alpha(v) \ge r \}$
  $(r \in \mathbb R)$. Inversement la norme associ\'ee \`a une fonction de r\'eseaux $\lambda$ est
  $v \mapsto \alpha_\lambda (v) = \sup \{r \in \mathbb R / \ v \in \lambda(r)\} $ $(v \in V)$.
  Si $f$ est une forme sur $V$ comme ci-dessus, on d\'efinit le dual d'un r\'eseau $L$
  de $V$ par
  $L^\ast = \{x \in V / \  f(x, L) \subset \pF \}$, et la fonction duale d'une fonction de r\'eseaux $\lambda$ de $V$ par
  $$
  \lambda^\ast (r)  = \left( \bigcup_{s > -r } \lambda(s) \right)^\ast
     \quad (r \in \mathbb R).
  $$
  La norme $\alpha$ est autoduale si et seulement si la fonction de r\'eseaux associ\'ee l'est, c'est-\`a-dire co\"\i ncide avec la fonction duale $\lambda_\alpha^\ast$ \cite[Corollaire 3.4]{BS}.

  Si $\alpha_1$ et $\alpha_2$ sont des normes sur $V_1$ et $V_2$ respectivement,
  alors $$\alpha_1 \wedge \alpha_2 : v_1 +v_2 \mapsto \min\{ \alpha_1(v_1), \alpha_2(v_2)\} \quad
  (v_1 \in V_1,  \ v_2 \in V_2) $$
   est une norme sur $V_1 \oplus V_2$. Une norme $\alpha$  sur $V$ est dite {\it scind\'ee} par une d\'ecomposition en somme directe $V = V_1 \oplus V_2$
  si $ \  \alpha = \alpha_{|V_1} \wedge \alpha_{|V_2}$. La notion correspondante pour
  une fonction  de r\'eseaux $\lambda$ est : $ \ \lambda(r) = \lambda(r) \cap V_1 \oplus \lambda(r) \cap V_2 \  $ pour tout $r \in \mathbb R$. Enfin une base $(e_i)_{i \in I}$ de $V$
  est une {\it base scindante} pour $\alpha$ si la d\'ecomposition en somme de droites
  associ\'ee scinde $\alpha$. Si c'est le cas, la base duale $(e_i^\ast)_{i \in I}$
  scinde $\alpha^\ast $ et, pour tout $i \in I$,  $\alpha^\ast(  e_i^\ast) =
  -\alpha (  e_i ) $.

  Supposons  $V$   muni d'une forme non d\'eg\'en\'er\'ee.
\begin{itemize}
    \item Si $V$ est somme orthogonale de $V_1$ et $V_2$
  alors la norme duale de $\alpha_1 \wedge \alpha_2$ est $\alpha_1^\ast \wedge \alpha_2^\ast$~;
  la norme  $\alpha_1 \wedge \alpha_2$ est autoduale si et seulement si $\alpha_1$ et $\alpha_2$ le sont.

  Si $\alpha$ est une norme autoduale sur $V$, elle est scind\'ee par $V = V_1 \perp V_2$ si et seulement si ses restrictions \`a $V_1$ et $V_2$ sont autoduales. (Condition suffisante :   $\alpha_{V_1} \wedge \alpha_{V_2}$ est alors autoduale et   minore $\alpha$, d'o\`u l'\'egalit\'e car toutes deux sont maximinorantes.)

  \item Si $V$ est somme directe de  $W^+$ et $W^-$, sous-espaces   totalement isotropes en dualit\'e, et si $\alpha^+$ est une norme sur $W^+$, $\alpha^-$  une norme sur $W^-$,
  alors la norme duale de $\alpha^+ \wedge \alpha^-$ est  $  (\alpha^-)^\sharp \wedge (\alpha^+)^\sharp$ o\`u $ (\alpha^+)^\sharp$ est d\'efinie par
  \begin{equation}\label{150}
   (\alpha^+)^\sharp(v) = \inf_{x \in W^+}(v_F(f(v,x)) - \alpha^+(x))  \quad (v \in W^-).
\end{equation}
  et de m\^eme pour $ (\alpha^-)^\sharp$. La norme $\alpha^+ \wedge \alpha^-$ est autoduale si et seulement si $\alpha^- = (\alpha^+)^\sharp$.
\end{itemize}

  Examinons \`a pr\'esent les propri\'et\'es des normes d'alg\`ebre autoduales de
  l'alg\`ebre d'octonions $V$ scind\'ees par une d\'ecomposition
  $V = V^0 \perp W$  comme au paragraphe pr\'ec\'edent. Notre objectif est de montrer qu'une telle norme est enti\`erement d\'etermin\'ee par sa restriction \`a $W$ et de caract\'eriser cette restriction. 

\begin{Lemma}\label{151} Soit $V^0$ une sous-alg\`ebre de composition de $V$, $W$ son orthogonal et
  $\alpha$ une norme  d'alg\`ebre autoduale  de
   $V$ scind\'ee  par la d\'ecomposition
  $V = V^0 \perp W$. Alors $\alpha_{|V^0}$ est une norme d'alg\`ebre autoduale de $V^0$ 
  et $\alpha_{|W}$   est une norme   autoduale de $W$.
  De plus :
\begin{itemize}
    \item  si $V^0$ est anisotrope  on a  $ \  \alpha_{|V^0}= \dfrac 1 2 \,  v_F \circ Q$ ;
\item si $V^0$ est de dimension $2$ et d\'eploy\'ee, d'idempotents orthogonaux $e^+$ et $e^-$,
  on a \\
  $\alpha(\xi e^+ + \mu e^- ) = \min\{v_F(\xi),  v_F(\mu)\}$ ($\xi, \mu \in F$).
  \end{itemize}
\end{Lemma}

\begin{proof} La premi\`ere assertion d\'ecoule des rappels ci-dessus. Pour la seconde rappelons qu'en caract\'eristique r\'esiduelle diff\'erente de $2$, une norme est autoduale si et seulement si elle est maximinorante, i.e. \'el\'ement maximal de l'ensemble des normes v\'erifiant
$$\alpha(x) + \alpha(y) \le v_F(f(x,y))  \qquad (x, y \in V) $$  (voir \cite[\S 2]{GY}  ou
\cite[Proposition 3.2]{BS}, ou la r\'ef\'erence originelle \cite[Proposition 2.5]{BT}).
La norme $\alpha_{|V^0}$ v\'erifie donc $\alpha_{|V^0}(x) \le \frac 1 2 v_F(Q(x))$ pour tout $x \in V^0$. Il reste \`a v\'erifier que  si $V^0$ est anisotrope,  $\frac 1 2 v_F \circ Q$ est bien une norme sur $V^0$ et v\'erifie la condition de minoration ci-dessus, ce qui est facile.

Supposons maintenant $V^0$ d\'eploy\'ee de dimension $2$.  On a pour $\xi, \mu \in F$ :
$$
\aligned
\alpha( \xi e^+ + \mu e^- )  + \alpha(e^+) &\le
\alpha((\xi e^+ + \mu e^- )e^+)= \alpha(\xi e^+ ) = v_F(\xi ) + \alpha(e^+)
\\
\text{et de m\^eme } \
\alpha( \xi e^+ + \mu e^- )  + \alpha(e^-) &\le  v_F(\mu ) + \alpha(e^-)
\endaligned
$$
donc $\alpha( \xi e^+ + \mu e^- ) \le \min\{ v_F(\xi ),  v_F(\mu )\}$.
Mais  $\alpha$ est maximinorante et   $ \xi e^+ + \mu e^-  \longmapsto \min\{v_F(\xi),  v_F(\mu)\}$ d\'efinit bien une norme autoduale (imm\'ediat via  \ref{150}), d'o\`u l'\'egalit\'e.
\end{proof}

\begin{Lemma}\label{152}
Soit $E$ une extension quadratique de $F$ d'indice de ramification $e$.
Soit $W $ un espace vectoriel sur $E$ muni d'une forme quadratique $f$ pour sa structure d'espace vectoriel sur $F$ et soit $\Phi$ la forme hermitienne correspondante, caract\'eris\'ee par $f = \tr_{E/F} \circ \Phi$. Alors $\alpha \mapsto \tfrac 1 e \, \alpha$ est une bijection de l'ensemble des $E$-normes sur $W$   autoduales relativement \`a $\Phi$ et l'ensemble
des $F$-normes $\gamma$ sur $W$  autoduales relativement \`a $f$ et v\'erifiant,
pour $x \in W$ et $\xi \in E$ :
$\gamma (\xi x)=    \tfrac 1 e \,  v_{E} (\xi) + \gamma(x)$.
\end{Lemma}

\begin{proof}  Si l'on supprime de part et d'autre la condition d'autodualit\'e l'\'enonc\'e est imm\'ediat. Examinons les conditions d'autodualit\'e en termes des fonctions de r\'eseaux associ\'ees.  Si $L$ est un $\oE$-r\'eseau de $W$,  son dual $L^\ast$  relativement \`a $f$ et son dual $$L^\sharp = \{x \in W / \  \Phi(x, L) \subset \pE \}$$ relativement \`a $\Phi$ co\"\i ncident. Si $\alpha$ est une norme autoduale relativement \`a $\Phi$,
on v\'erifie que $\lambda_\alpha^\sharp (er) = \lambda_{\tfrac 1 e \alpha}^\ast (r)$. On en d\'eduit  : $\lambda_\alpha^\sharp =  \lambda_\alpha \iff
\lambda_{\tfrac 1 e \alpha}^\ast =  \lambda_{\tfrac 1 e \alpha} $, d'o\`u l'\'enonc\'e.
\end{proof}

Rappelons (\cite[\S 1.9]{BT0}) que le  volume d'un r\'eseau $L$ d'un $F$-espace vectoriel de dimension finie  $W$ est d\'efini moyennant le  choix d'un r\'eseau particulier $L_0$ de $W$
par $\vol(L)= v_F(\det g)$ o\`u $g \in GL_F(W)$ v\'erifie $g(L_0)=L$.
Le {\it volume d'une norme} $\alpha$ sur $W$ est alors d\'efini au moyen d'une base scindante
$(e_i )_{i \in I}$ de $\alpha$ et du $\oF$-r\'eseau $L$ engendr\'e par cette base par
$$
\vol \alpha = \vol L - \sum_{i \in I} \alpha(e_i).
$$
Cette quantit\'e est ind\'ependante du choix d'une base scindante.

\begin{Proposition}\label{volume0sl}
Soit $\alpha$ une norme d'alg\`ebre autoduale de $V$ scind\'ee par $V = V^0 \perp W$ o\`u
$V^0$ est de dimension $2$ et d\'eploy\'ee, d'idempotents orthogonaux $e^+$ et $e^-$.
Soit $W = W^+ \oplus W^-$ la polarisation compl\`ete de $W$ associ\'ee :  $W^+ = e^+ W$ et $W^- = e^- W$.
\begin{enumerate}
    \item La restriction de $\alpha$ \`a $W$ est scind\'ee par $W = W^+ \oplus W^-$,  donc de la forme $\alpha^+ \wedge \alpha^- $ avec $\alpha^- = (\alpha^+)^\sharp$ (cf \ref{150}).
    \item Pour toute base
       $(f_1^+, f_2^+, f_3^+)$ de $W^+$ on note  $(f_1^-, f_2^-, f_3^-)$
    la base duale de $W^-$. Alors $f_1^- f_2^-$ est un multiple non nul de $f_3^+$.
    Appelons {\em base sp\'eciale} de $W^+$ toute base $(f_1^+, f_2^+, f_3^+)$ telle que
    $f_3^+ = f_1^- f_2^-$. Les r\'eseaux engendr\'es par les bases sp\'eciales de $W^+$ ont tous m\^eme volume.
    \item On normalise le volume de sorte que les r\'eseaux engendr\'es par les bases sp\'eciales de $W^+$ soient de volume nul. Alors la restriction de $\alpha$ \`a $W^+$ est de volume nul.
    \item R\'eciproquement, soit $\alpha^+$ une norme de $W^+$ de volume nul pour la normalisation ci-dessus, soit $\alpha^- = (\alpha^+)^\sharp$, norme de $W^-$, et soit    $\alpha_0$ l'unique norme
    d'alg\`ebre autoduale de $V^0$. Alors
    $\alpha_0 \wedge \alpha^+ \wedge \alpha^-$ est une norme d'alg\`ebre autoduale de $V$, 
    l'unique telle norme scind\'ee par   $V = V^0 \perp W$ et prolongeant $\alpha^+$.
\end{enumerate}
\end{Proposition}
\begin{proof}
(i) Soit $w \in W$, il s'\'ecrit $w= e^+ w + e^- w$ avec $e^+ w \in W^+$ et $e^- w \in W^-$.
On a $  \  \alpha (e^+ w) \ge \alpha(e^+) + \alpha (w)= \alpha(w) \  $ (Lemme \ref{151}) et de m\^eme pour $\alpha (e^- w)$,  donc :
$\alpha (w) \ge \min (\alpha (e^+ w), \alpha (e^- w)) \ge \alpha(w) \ $ d'o\`u l'\'egalit\'e.

(ii) Pour $i=1,2,3$,  l'\'el\'ement $a_i= f_i^+ + f_i^-$ est de norme $1$, on peut donc utiliser le proc\'ed\'e de doublement (Proposition \ref{double}) pour calculer dans $V^0 \perp V^0 a_i$
(qui est une alg\`ebre associative) les produits suivants :
$f_i^+ f_i^- = -e^+$, $f_i^- f_i^+ = -e^-$, $f_i^+ f_i^+ = f_i^- f_i^- = 0$. Par ailleurs on v\'erifie que pour $x,y \in W^+$ (resp. $W^-$) le produit $xy$ appartient \`a $W^-$ (resp. $W^+$). On en d\'eduit imm\'ediatement que $f_1^- f_2^-$ est un \'el\'ement de $W^+$ orthogonal \`a $f_1^-$ et $f_2^-$. Toujours par doublement, l'espace engendr\'e par
$f_2^+, f_2^-, f_3^+, f_3^-$ est l'orthogonal de $V^0 \perp V^0 a_1$,  \'egal \`a
$(V^0 \perp V^0 a_1) a_2$. De fait $ \ f_1^- f_2^- = (e^- a_1) (e^- a_2) = (e^-  e^- a_1) a_2 = f_1^- a_2 \ $ est non nul.

Posons $f_1^- f_2^- = \mu f_3^+   $ avec $\mu \in F^\times$ et soit $v$ le volume du r\'eseau de base $(f_1^+, f_2^+, f_3^+)$. Le r\'eseau engendr\'e par la base sp\'eciale $(f_1^+, f_2^+, f_1^- f_2^-)$ a pour volume $v + \vF(\mu)$. Un r\'eseau arbitraire est image par $g \in \GL(3,F)$ du premier et d\'etermine une base sp\'eciale $(gf_1^+, g f_2^+, ( ^t g^{-1}f_1^-)
( ^t g^{-1}f_2^-) )$  engendrant un r\'eseau de volume
$v + \vF(\det g) + \vF(\mu')$ avec $\mu'= f(  ^t g^{-1}f_3^- , (^t g^{-1}f_1^-)( ^t g^{-1}f_2^-))$.

Posons $d = \det g^{-1}$ ; alors   $g'= g \left(\begin{smallmatrix} 1&0&0 \cr 0&1&0\cr 0&0& d \end{smallmatrix}\right)$ appartient \`a $\SL(3,F)$ et d\'etermine un \'el\'ement $\tilde g$ de $G$ valant l'identit\'e sur $V^0$, $g'$ sur $W^+$ et $ ^t g'^{-1}$ sur $W^-$. Ainsi :

$$\mu' = d \,   f(\tilde g (f_3^-), \tilde g (f_1^-)\tilde g (f_2^-) )
=   d \,  f(\tilde g (f_3^-), \tilde g (f_1^- f_2^-) ) =d \,  \mu .$$

(iii) Partant d'une base $(f_1^+, f_2^+, f_3^+)$ scindant $\alpha$, la base sp\'eciale
$(f_1^+, f_2^+, f_1^- f_2^-)$ scinde $\alpha$ et
$\vol \alpha^+ = -\alpha^+(f_1^+) - \alpha^+(f_2^+) - \alpha^+(f_1^- f_2^-) $.
Rappelons  que $\alpha^+(f_i^+) = - \alpha^-(f_i^-)$ \cite[(18)]{BT0}. Comme
  la base duale de $(f_1^+, f_2^+, f_1^- f_2^-)$ est $(f_1^-, f_2^-, f_1^+ f_2^+)$, on a
$$\alpha^+(f_1^- f_2^-) = - \alpha^-(f_1^+ f_2^+) \le -\alpha(f_1^+) -\alpha(f_2^+) ,$$
soit $\vol \alpha^+ \ge 0$.
  Echangeons les r\^oles de $W^+$ et $W^-$ : on obtient $$ -\alpha^-(f_1^-) - \alpha^-(f_2^-) - \alpha^-(f_1^+ f_2^+) \ge 0,$$  c'est-\`a-dire
$\vol \alpha^+ \le 0$, c.q.f.d.

(iv)  Bien s\^ur     $\alpha = \alpha_0 \wedge \alpha^+ \wedge \alpha^-$ est une norme   autoduale de $V$ ; il s'agit de montrer que c'est une norme d'alg\`ebre. Soit $(f_1^+, f_2^+, f_1^- f_2^-)$
une  base sp\'eciale de $W^+$ scindant $\alpha^+$. D'apr\`es \cite[Lemma 2.1]{GY}
il suffit de montrer que la propri\'et\'e $\alpha (xy) \ge \alpha(x) + \alpha(y)$ est valide lorsque $x$ et $y$ appartiennent \`a la base $(e^+, e^-, f_1^+, f_2^+, f_1^- f_2^-,
f_1^-, f_2^-, f_1^+ f_2^+)$ de $V$. C'est une v\'erification facile en utilisant les propri\'et\'es d\'ej\`a mentionn\'ees et le proc\'ed\'e de doublement, qui permet d'\'etablir que les produits $f_i^+ f_j^-$ et $f_i^- f_j^+$ sont nuls pour $i \ne j$ et que l'image d'une base sp\'eciale par permutation circulaire reste sp\'eciale.
\end{proof}

 {\begin{Proposition}\label{volume0su}
Soit $\alpha$ une norme d'alg\`ebre autoduale de $V$ scind\'ee par $V = V^0 \perp W$ o\`u
$V^0$ est une extension quadratique  $F'$ de $F$, d'indice de ramification $e$  ; notons $\alpha_W$ sa restriction \`a $W$.
Soit $\Phi$ la forme hermitienne sur $V$ caract\'eris\'ee par $f = \tr_{F'/F} \circ \Phi$.
\begin{enumerate}
\item Pour toute base orthogonale       $(f_1 , f_2 , f_3 )$ de $W$ sur $F'$, $ f_1  f_2   $ est un multiple non nul de $f_3 $. 
    Appelons {\em base sp\'eciale} de $W $ toute base de Witt $(e_{-1}, e_{ 0} , e_{1} )$ sur $F'$ telle que   $Q(e_{-1})=Q(e_{ 1})=0$, $\Phi(e_{-1}, e_1)=1$, $e_{0} = (e_{-1}+ e_{ 1})(e_{-1}-e_{ 1})$. 
    Les $ \mathfrak o_{F'}$-r\'eseaux engendr\'es par les bases sp\'eciales de $W $ ont tous m\^eme volume et si l'on 
    normalise le volume de sorte que les $ \mathfrak o_{F'}$-r\'eseaux engendr\'es par les bases sp\'eciales de $W $ soient de volume nul, alors les $F'$-normes autoduales de $W$ sont de volume nul.

    \item  La norme   $ \ e \alpha_W \  $ est une $F'$-norme autoduale   sur $W$.
    \item R\'eciproquement, soit $\alpha'$ une $F'$-norme autoduale  de $W $  et soit    $\alpha_0$ l'unique norme
    d'alg\`ebre autoduale de $V^0$. Alors
    $\alpha_0 \wedge \frac 1 e \alpha'$ est une norme d'alg\`ebre autoduale de $V$, l'unique telle norme scind\'ee par   $V = V^0 \perp W$ et prolongeant  $\frac 1 e \alpha'$. 
\end{enumerate}
\end{Proposition}
\begin{proof}
(i) Le proc\'ed\'e de doublement (\ref{double}) donne
$W = F' f_1 \perp (F' \perp F' f_1) f_2$, donc $(F' f_1) f_2= F' f_3$ (de sorte que 
le d\'eterminant de $\Phi_{|W}$ est  une norme de $F'$ sur $F$). 
Le r\'eseau engendr\'e par la base sp\'eciale $(e_{-1}, e_{ 1} , e_{0} )$ est identique au
r\'eseau engendr\'e par la base orthogonale $(e_{-1}+ e_{ 1},e_{-1}-e_{ 1} , e_{0} )$
($p\ne 2$). Il suffit donc de   montrer que le volume est constant sur l'ensemble des $ \mathfrak o_{F'}$-r\'eseaux engendr\'es par des bases orthogonales $(f_1 , f_2 , f_3 )$ telles que $Q(f_1) Q(f_2) Q(f_3) = 1$. Soient $(f_1 , f_2 , f_3 )$ et $(f'_1 , f'_2 , f'_3 )$ deux telles bases. Quitte \`a permuter les vecteurs et \`a multiplier le premier par $a \in F'^\times$ et le second par $a^{-1}$, on peut supposer $Q(f_1)=Q(f'_1)=1$. Il y a alors une isom\'etrie de $W$ envoyant $f_1$ sur $f'_1$, ce qui permet de se ramener \`a la question analogue en dimension $2$. Soit alors $g \in GL(2, F')$ envoyant $f_2$, $f_3$ sur $f'_2$, $f'_3$. En \'ecrivant les conditions $\Phi(f'_2, f'_3)= 0$ et $\Phi(f'_2, f'_2) \Phi(f'_3, f'_3)=1$ on obtient imm\'ediatement $\det g^2 = 1$.

 Une norme autoduale $\alpha'$ sur $W$ a  m\^eme volume que sa norme duale. Or elle est scind\'ee par une base de Witt, donc aussi par une base sp\'eciale  $(e_{-1}, e_{ 1} , e_{0} )$, 
 dont
la base duale  $(e_{ 1}, e_{- 1} , t e_{0} )$ ($t \in \mathfrak o_{F'}$) engendre le m\^eme
r\'eseau. La norme duale v\'erifie $( \alpha')^\ast (e_i^\ast) = -   \alpha'(e_i)$, son volume est donc l'oppos\'e de celui de $  \alpha'$, d'o\`u la nullit\'e.

(ii) Pour $u \in F^{\prime\times}$ et $w \in W$ on a
$\alpha(u) + \alpha(w) \le \alpha (uw) \le -\alpha(u^{-1}) + \alpha(w)$. On en d\'eduit
par le lemme  \ref{151} que $\alpha (uw) =  \dfrac 1 2 \,  v_F \circ Q(u) + \alpha(w)$,
puis par le lemme \ref{152}   que
 $ e \alpha_W$ est une $F'$-norme autoduale   sur $W$. 

(iii) Soit $(\un, \epsilon)$ une base de $V^0$ scindante pour $\alpha_0$ et   $(e_{-1}, e_{ 1} , e_{0} )$ une base sp\'eciale de $W$ scindant $\alpha'$. Posons $\alpha =\alpha_0 \wedge \frac 1 e \alpha'$.  Il suffit d'\'etablir   $\alpha (xy) \ge \alpha(x) + \alpha(y)$   lorsque $x$ et $y$ appartiennent \`a la base $(\un, \epsilon, e_{-1} , \epsilon e_{-1},  e_{ 1}  , \epsilon e_{ 1} ,  e_{0}, \epsilon e_{0})$ de $V$ \cite[Lemma 2.1]{GY}. Ce sont des v\'erifications faciles, bas\'ees sur le lemme \ref{idempotents}, le fait que     $\alpha(\xi v)= \frac 1 e v_{F'} (\xi) + \alpha(v)$ ($\xi \in F'$, $v \in V$) et l'autodualit\'e de $\alpha'$, soit  :
$ 
\alpha'(e_{-1} ) + \alpha'(e_{ 1} ) = 0 , \quad \alpha'(e_{0}) = 0.
$ 
\end{proof}

{\begin{Proposition}\label{dimension4}
Soit 
$V^0$  une sous-alg\`  ebre de composition de dimension $4$ de $V$ et $W$ son orthogonal. 
\begin{enumerate}
 \item Si $V^0$ est anisotrope, alors $W$ l'est aussi et chacun poss\`ede une unique norme autoduale, soit    $\alpha_0= \frac 1 2 v_F\circ Q$ et  $\alpha_W= \frac 1 2 v_F\circ Q$  respectivement. La norme $\alpha_0 \wedge \alpha_W$ est une norme d'alg\`ebre autoduale de $V$. 
 La norme $\alpha_0$ est scind\'ee   par toute d\'ecomposition $V^0 = F' \perp F'^\perp$ o\`u $F'$ est un sous-corps de $V^0$ de degr\'e $2$ sur $F$. 

    \item  Si $V^0$ est d\'eploy\'ee, pour toute norme   autoduale $\alpha_W$  de $W$ 
    il existe une et une seule norme d'alg\`ebre autoduale $\alpha_0$ de $V^0$ telle que 
   $\alpha_0 \wedge \alpha_W$ soit une norme d'alg\`ebre autoduale de $V $.
   Si en outre il existe une extension quadratique $F'$ de $F$ dans $V^0$ telle que  $\alpha_W$ soit de la forme $\frac 1 e \alpha'$ pour une 
   $F'$-norme autoduale  $\alpha'$  de $W $, alors
    $e \alpha_0  $ est une $F'$-norme autoduale de $V^0$.
\end{enumerate}
\end{Proposition}
\begin{proof}
La premi\`ere assertion de  (i) est imm\'ediate via le lemme \ref{151}. Quant au scindage, il suffit de v\'erifier que les restrictions de $\alpha_0$ \`a $F'$ et \`a son orthogonal sont autoduales, ce qui est facile. 

Passons \`a (ii). La norme $\alpha_W$ est scind\'ee par une base de Witt $(h, h', k, k')$ avec $f(h,h')=f(k,k')=1$, les autres valeurs de $f$ sur cette base  \'etant nulles. Aux droites $L=Fh$ et $L'=Fh'$ on attache 
les \'el\'ements  $e^+=-hh' $, $ e^-= -h'h $ et $c=e^+-e^-$  du lemme \ref{idempotents}~; ils appartiennent \`a $V^0$ par doublement, puisque $W=V^0(h-h')$, ainsi que l'\'el\'ement $b= (k+k')(h-h')$.  En outre $b$ est de norme $-1$ et orthogonal \`a $c$,  de sorte que 
\begin{equation}\label{fc}
V^0 = F[c] \perp F[c]b = (Fe^+ \oplus F e^-) \perp (F e^+b \oplus F e^-b).
\end{equation}

Les droites $F (e^+b)(h-h')\subset e^- W$ et $ F (e^-b)(h-h')\subset e^+W$ sont les deux droites isotropes de l'orthogonal de $L \oplus L'$ dans $W$ et l'on a 
$f((e^+b)(h-h'),(e^-b)(h-h')) = 1$  ; quitte \`a \'echanger $k$ et $k'$ on peut  supposer 
que $(e^+b)(h-h')$ est multiple de $k'$ et on voit ais\'ement qu'on a alors \'egalit\'e : 
$(e^+b)(h-h')=k' $ et $(e^-b)(h-h')=k$. Alors : 
$$
 -h'k' = [e^-(h-h')][(e^+b)(h-h')] = \overline{(e^+b)}e^-= \bar b e^- e^-  = -b e^- $$
 et de m\^eme $hk = -b e^+$.  
 
Ces petits calculs vont nous permettre de montrer que si $\alpha_0$ est une norme d'alg\`ebre autoduale de $V^0$ telle que  $\alpha = \alpha_0 \wedge \alpha_W$ soit une norme d'alg\`ebre autoduale de $V $, alors elle est scind\'ee par les d\'ecompositions \ref{fc}. 
De fait :
\begin{itemize}
	\item $  \alpha_0 ( (e^+)^2) \ge 2 \alpha_0(e^+)$ entra\^ine $\alpha_0(e^+) \le 0$ mais 
$ 
\alpha_0 (e^+) = \alpha_0 (h h') \ge \alpha_W (h  ) + \alpha_W  ( h') 
$,  qui vaut $0$ par autodualit\'e. Donc $\alpha_0 (e^+)=0$ et de m\^eme $\alpha_0(e^-)=0$. 
\item Pour tout $v \in V$ on a donc $\alpha (v) = \min \{\alpha (e^+v), \alpha (e^- v)\}$
(cf. preuve de \ref{volume0sl} (i)). 
\item Enfin $(e^-b)(e^+b)= \overline{ e^+   }e^-= e^-$ donc 

$ 0 \ge  \alpha_0  (e^-b)+ \alpha_0 (e^+b) =  \alpha_0  (h'k')+ \alpha_0 (h k) 
\ge  \alpha_W  (h') + \alpha_W(k')+ \alpha_W (h ) + \alpha_W(k)$  

qui vaut $0$ par autodualit\'e, soit : $\alpha_0  (e^-b)+ \alpha_0 (e^+b)=0$. 
\end{itemize}

Ainsi ${\alpha_0}_{|F[c]}$ est l'unique norme d'alg\`ebre autoduale de $F[c]$ (lemme \ref{151} (i)) et 
${\alpha_0}_{|F[c]b}$ est autoduale, $\alpha_0$ est donc autoduale scind\'ee par 
$F[c] \perp F[c]b$ et $\alpha$   scind\'ee par $V =   F[c] \perp F[c]^\perp$. 
Il reste \`a utiliser la proposition \ref{volume0sl} : $\alpha $ est une norme d'alg\`ebre 
autoduale si et seulement si 
$\alpha_0(e^+b) + \alpha_W(h) + \alpha_W(k)=0$,  ce qui d\'etermine $\alpha_0$. 

Pour la derni\`ere assertion, il suffit de remarquer que, notant $X^0$ l'orthogonal de $F'$ dans $V^0$, l'espace $X = X^0 \perp W$ est hermitien de dimension $3$ sur $F'$, agissant par multiplication \`a gauche. Si $(h, h')$ est une base de Witt de $W$ sur $F'$ scindant $\alpha'$, 
l'\'el\'ement  $(h + h')(h-h')$ appartient \`a $X^0$ et $\alpha'$ se prolonge de fa\c con unique en une $F'$-norme autoduale de $X$, d'o\`u par la proposition \ref{volume0su} (iii) une unique norme d'alg\`ebre autoduale  $\alpha$ scind\'ee par $V = F' \perp X$ et prolongeant $\frac 1 e \alpha'$. Par unicit\'e c'est la m\^eme que la pr\'ec\'edente.  
 \end{proof}

Revenons pour terminer au langage des suites de r\'eseaux.
Soit  $\alpha$  une  norme  \`a valeurs rationnelles sur $V$ et $\lambda_\alpha$  la fonction de r\'eseaux associ\'ee. Soit $m_\alpha$ le p.p.c.m.  des d\'enominateurs des valeurs de $\alpha$ \'ecrites sous forme de fractions irr\'eductibles, c'est-\`a-dire
$$m_\alpha = \min \{ t \in \mathbb N / \alpha(V-\{0\}) \subseteq \dfrac 1 t \mathbb Z\}.$$
On d\'efinit une suite de r\'eseaux  $\Lambda_\alpha$  de p\'eriode minimale, par
$$\Lambda_\alpha(i)= \lambda_\alpha(\frac i m_\alpha)  \quad (i \in \mathbb Z). $$
 Si $\alpha$ est autoduale, la suite est autoduale et d'invariant $1$ : 
 $\Lambda_\alpha(i)^\ast =  \Lambda_\alpha(1-i)$, conform\'ement aux conventions de \cite{S4} et 
 \cite{S5}.

Il est \'el\'ementaire de v\'erifier que, dans la situation de la proposition
\ref{volume0sl} (iv) et en posant $\alpha = \alpha_0 \wedge \alpha^+ \wedge \alpha^-$,
on a pour tout $i \in \mathbb Z$ : $$\Lambda_\alpha(i) \cap W^+ = \Lambda_{\alpha^+}(i).$$
De m\^eme, dans la situation de la proposition
\ref{volume0su} (iii) et en posant $\alpha = \alpha_0 \wedge \frac 1 e \alpha'$,
on a pour tout $i \in \mathbb Z$ : $$\Lambda_\alpha(i) \cap W  = \Lambda_{\alpha'}(i),$$ 
suite autoduale pour la structure $F'$-hermitienne de $W$. Ainsi le passage par les normes r\'esout le probl\`eme habituel de normalisation des suites de r\'eseaux dans une sommation. Pour faciliter les r\'ef\'erences nous r\'esumons cette discussion   ci-dessous.

\begin{Notation}\label{Lambdachapeau1}
 Fixons une  sous-alg\`ebre de composition $D$ de dimension $2$ et hyperbolique de $V$.
 Pour toute suite de r\'eseaux $\Lambda$  de
 $ W_D^+ $ attach\'ee \`a une norme de volume nul,
 on note  $ \check\Lambda$ l'unique suite  de r\'eseaux de $V$   correspondant \`a une norme d'alg\`ebre autoduale de $V $ scind\'ee par
 $V = D \perp (W_D^+ \oplus W_D^-)$ et v\'erifiant  :
 $\forall i \in \mathbb Z$, $ \check\Lambda(i) \cap W_D^+  = \Lambda(i)$.
 \end{Notation}

 \begin{Notation}\label{Lambdachapeau2}
  Fixons une  sous-alg\`ebre de composition $D$ de dimension $2$ et anisotrope de $V$.
   Pour toute suite de ${\mathfrak o}_{D}$-r\'eseaux $\Lambda$  de
 $ D^\perp $ attach\'ee \`a une $D$-norme autoduale,
 on note  $ \vec\Lambda$ l'unique suite  de r\'eseaux de $V$ correspondant \`a une norme d'alg\`ebre autoduale de $V $ scind\'ee par
 $V = D \perp D^\perp $ et v\'erifiant  :
 $\forall i \in \mathbb Z$, $ \vec\Lambda(i) \cap D^\perp  = \Lambda(i)$.
 \end{Notation}

\begin{Notation}\label{Lambdachapeau3}
  Fixons une  sous-alg\`ebre de composition $D$ de dimension $4$   de $V$.
   Pour toute suite de  r\'eseaux $\Lambda$ de
 $ D^\perp $, autoduale et d'invariant $1$,  
 on note  $ \hat\Lambda$ l'unique suite  de r\'eseaux de $V$ correspondant \`a une norme d'alg\`ebre autoduale de $V $ scind\'ee par
 $V = D \perp D^\perp $ et v\'erifiant  :
 $\forall i \in \mathbb Z$, $ \hat\Lambda(i) \cap D^\perp  = \Lambda(i)$.
 \end{Notation}
 On notera que, d'apr\`es la d\'emonstration de la proposition \ref{dimension4}, les suites de r\'eseaux de \ref{Lambdachapeau3} sont toujours \'egalement de la forme \ref{Lambdachapeau1} ou 
 \ref{Lambdachapeau2}.  D'ailleurs, toute  suite  de r\'eseaux de $V$ attach\'ee  \`a une norme d'alg\`ebre autoduale   est de la forme $ \check\Lambda$ (\ref{Lambdachapeau1}) 
 par \cite[Proposition 8.1]{GY}. Toutefois c'est le point de vue ci-dessus qui nous sera utile dans ce qui suit, et pour commencer dans le lemme \ref{lemme28}. Il s'agit d'une version du lemme 2.8 de \cite{S5} pour le groupe $G_2(F)$~; remarquons que le lemme initial se g\'en\'eralise sans probl\`eme du cas des strates gauches \`a celui   des strates autoduales.

 \begin{Lemma}\label{lemme28}
 Soit $\beta$ un \'el\'ement semi-simple non nul de $\mathfrak{g}_2(F)$  dont on note  
 $V^0$ le noyau, d'orthogonal $W$.
 Soit $B_{\beta_W} $ le commutant  de $\beta_W$ dans $\mathfrak{gl}_F(W)$.  
Soient $\L$ et $\L'$ deux suites de r\'eseaux autoduales de $V$ correspondant \`a des points de l'immeuble de $G_2(F)$ et normalis\'ees par $F[\beta]^\times$. Supposons que  $\widetilde{\mathfrak A}_0(\L'_W) \cap B_{\beta_W} 
 \subseteq \widetilde{\mathfrak A}_0(\L_W) \cap B_{\beta_W} $.   Alors il existe  
  une suite de r\'eseaux autoduale $\L^{''}$ de $V$ correspondant \`a un  point de l'immeuble de $G_2(F)$ et normalis\'ee par $F[\beta]^\times$, telle que : 
  $$
 \widetilde{\mathfrak A}_0(\L^{''}_W)  \cap B_{\beta_W}  = \widetilde{\mathfrak A}_0(\L'_W)  \cap B_{\beta_W}  
  \text{ et }  \quad 
 \widetilde{\mathfrak A}_0(\L^{''}) \subseteq \widetilde{\mathfrak A}_0(\L). 
  $$
 \end{Lemma}

 \begin{proof}
 Par d\'efinition d'une strate semi-simple, 
  les suites de r\'eseaux $\L$ et  $\L'$ sont scind\'ees par la d\'ecomposition 
 $V = V^0 \perp W$  et le centralisateur de $\beta$ est produit de   $\mathfrak{gl}_F(V^0)$ 
 et  de son   intersection avec $\mathfrak{gl}_F(W)$.
 Distinguons le cas (i)~:  $V^0$  d\'eploy\'ee de dimension $2$, des autres cas~: 
  $V^0$  anisotrope ou d\'eploy\'ee de dimension $4$.
  Pour  le cas (i)  on  
 d\'ecompose davantage via 
 la polarisation compl\`ete   $W= W^+ \oplus W^-$  qui scinde 
 les suites de r\'eseaux donn\'ees~;  le centralisateur de $\beta_W$ est produit de  ses  intersections avec 
 $\mathfrak{gl}_F(W^+)$ et $\mathfrak{gl}_F(W^-)$ (lemme \ref{centralGL}), de sorte que :  
 $$
   \widetilde{\mathfrak A}_0(\L'_{W }) \cap B_{\beta_{W}} 
 \subseteq \widetilde{\mathfrak A}_0(\L_{W}) \cap B_{\beta_{W}}
 \iff     
  \widetilde{\mathfrak A}_0(\L'_{W^+}) \cap B_{\beta_{W^+}} 
 \subseteq \widetilde{\mathfrak A}_0(\L_{W^+}) \cap B_{\beta_{W^+}}.
 $$ 
 Appliquons alors \cite[Lemma 2.8]{S5} dans  $W^+$ (cas (i)) ou $W$   : il existe une suite de r\'eseaux $\L^\diamond_{W^+}$  de  $W^+$ (cas (i)) ou 
 $\L^\diamond_W$ de $W$, normalis\'ee par  $F[\beta_{W^+}]^\times$  (cas (i)) ou 
  $F[\beta_{W}]^\times$, telle que : 
  \begin{equation}\label{equation28}
  \begin{aligned}
   &\widetilde{\mathfrak A}_0(\L^\diamond_{W^+}) \cap B_{\beta_{W^+}} 
 = \widetilde{\mathfrak A}_0(\L'_{W^+}) \cap B_{\beta_{W^+}} 
 \quad  \text{ et }  \     \widetilde{\mathfrak A}_0(\L^\diamond_{W^+})  
 \subseteq  \widetilde{\mathfrak A}_0(\L_{W^+})  
 \quad  \text{ (cas (i))} 
 \\
  &\widetilde{\mathfrak A}_0(\L^\diamond_W) \cap B_{\beta_W} 
 = \widetilde{\mathfrak A}_0(\L'_W) \cap B_{\beta_W}
 \quad   \text{ et }  \    
  \widetilde{\mathfrak A}_0(\L^\diamond_W) 
 \subseteq \widetilde{\mathfrak A}_0(\L_W)  
 \quad   \text{ (autres cas)} .
 \end{aligned}
  \end{equation} 
Dans le cas (i) on peut translater pour que la norme associ\'ee \`a 
$\L^\diamond_{W^+}$ soit de volume nul 
en base sp\'eciale, cela ne change pas 
$ \widetilde{\mathfrak A}_0(\L^\diamond_{W^+}) $. 
Dans les autres cas les suites  $\L_W$ et $\L'_W$ sont des suites de $V^0$-r\'eseaux autoduales pour une structure hermitienne ou orthogonale   et par {\it loc. cit.} on peut choisir 
$\L^\diamond_W$ de m\^eme.

  Tout choix d'une telle suite  $\L^\diamond_{W^+}$   ou 
 $\L^\diamond_W$  d\'etermine une unique suite de r\'eseaux $\L^\diamond$ de $V$ corres\-pondant \`a un point de l'immeuble de $\G2(F)$, normalis\'ee par $F[\beta]^\times$~: c'est le proc\'ed\'e   canonique qu'on a d\'evelopp\'e plus haut,  
  qui fournit 
une inclusion naturelle de l'immeuble~$\mathscr I$  de 
  $\SL_F(W^+)$  (\ref{Lambdachapeau1}), $SU(W, V^0/F)$  (\ref{Lambdachapeau2}) 
  ou $SO_F(W)$  (\ref{Lambdachapeau3}) dans l'immeuble de $G_2(F)$.  Ces inclusions conservent le barycentre puisqu'elles s'obtiennent par prolongement unique de normes, elles sont donc continues. 
   Pour la m\^eme raison, 
    l'inclusion naturelle de l'immeuble de $\G2(F)$  dans l'immeuble \'etendu de
     $\GL_F(V)$,  donn\'ee par le fait qu'une norme d'alg\`ebre autoduale sur $V$  est une norme sur $V$,  est continue.

 Or les conditions (\ref{equation28}) restent v\'erifi\'ees pour tout point de l'intersection avec~$\mathscr I$  
 de la facette $\widetilde{\mathscr F}$ de  $\L^\diamond_{W^+}$   ou 
 $\L^\diamond_W$   dans l'immeuble \'etendu de
     $\GL_F(W^+)$ ou   $\GL_F(W)$ et affirment que  $\Lambda_{W^+}$   ou   $\Lambda_{W}$  est adh\'erent \`a la facette  $\widetilde{\mathscr F}$.  
     Consid\'erons le segment g\'eod\'esique $[\L^\diamond_{W^+}, 
     \L_{W^+}]$  ou    $[\L^\diamond_W, \L_W]$  dans~$\mathscr I$ qui s'identifie  
     par l'inclusion canonique avec le segment $[\L^\diamond , \L]$ de l'immeuble de $\G2(F)$.  Un voisinage $\mathscr V$  assez petit  de $\L$ dans l'immeuble de $\G2(F)$ ne contient que des points $\L^{'''}$ tels que  $ \widetilde{\mathfrak A}_0(\L^{'''})  \subseteq \widetilde{\mathfrak A}_0(\L) $. L'intersection de $\mathscr V$ avec 
     $[\L^\diamond , \L[$ est form\'ee de points  $\L^{''} $ tels que 
        $\Lambda^{''}_{W^+}$ 
ou $\L^{''}_W$ appartienne \`a $\widetilde{\mathscr F} \cap \mathscr I$
(car les facettes sont convexes donc  $[\L^\diamond_{W^+}, 
     \L_{W^+}[$, resp.     $[\L^\diamond_W, \L_W[$, est tout entier contenu dans la facette  $\widetilde{\mathscr F}$). Tout tel point   $\L^{''} $   v\'erifie les conditions voulues. 
     \end{proof}

\subsection{Classification des strates semi-simples de $\mathfrak g_2(F)$}\label{strates.4}

   L'\'etude du   paragraphe  \ref{strates.3}, dont on reprend les notations, nous permet d'\'etablir une liste des  formes  possibles de strates semi-simples non nulles de $\mathfrak g_2(F)$. Soit $[\L, n, r, \beta]$, $0\le r < n$,  une telle strate,  soit $V^0$    le sous-espace propre de $\beta$ pour la valeur propre nulle et soit $W$ son orthogonal. Lorsque $V'$ est un sous-espace de $V$  stable par $\beta$, on  note  $\beta_{V'}$ la restriction de $\beta$ \`a $V'$ (en tant qu'endomorphisme de $V'$) et $\L_{V'}$ la suite d\'efinie par : $\L_{V'}(i)=\L(i)\cap V'$ pour tout $i\in \BZ$. La strate a l'une des formes suivantes :
\begin{enumerate}
\item\label{planhyperbolique} $V^0$ est un plan hyperbolique. Alors $W$ poss\`ede une polarisation compl\`ete uniquement d\'etermin\'ee par $V^0$, $W=W^+\oplus W^-$, qui est stable par $\beta$ et la strate $[\Lambda_{W^+}, n,r,\beta_{W^+}]$ est une strate semi-simple dans $\msl(3,F)$  telle que  $\beta_{W^+}$ n'ait aucune valeur propre nulle~;

\item\label{extensionquadratique} $V^{0}$ est une extension quadratique $F'$ de $F$. Alors $W$ est un $F'$-espace vectoriel muni d'une forme hermitienne et la strate $[\Lambda_{W}, n,r, \beta_{W}]$  est une strate semi-simple dans $\su(2,1)(F'/F)$  telle que  $\beta_{W}$ n'ait aucune valeur propre nulle~;

\item\label{scalairegl2} $V^{0}$ est de dimension 4 et   $\beta_W^2=\lambda^2 \, {\rm Id}_W$ pour un $\lambda\in F^\times$. Alors $V^0$ est d\'eploy\'ee, les espaces propres $W_\lambda$ et $W_{-\lambda}$ de $\beta$ associ\'es aux valeurs propres non nulles d\'efinissent une polarisation compl\`ete de $W$ et la strate $[\L_{ W_\lambda}, n, r, \beta_{W_\lambda}]$ est une strate scalaire dans $\mathfrak{gl}(2,F)$~;

\item\label{scalairesu} $V^0$ est de dimension 4 et $\beta_W^2=u \, {\rm Id}_W$ o\`u $u\in F$ n'est pas un carr\'e. Alors $\beta_W$ munit $W$ d'une structure de $F[\beta_W]$-espace vectoriel  hermitien  et la strate $[\Lambda_{W}, n,r, \beta_{W}]$  est une strate scalaire dans $\mfu(W, F[\beta_W]/F)$. 

 \end{enumerate}

    \begin{Theorem}\label{remontee}
     La  liste ci-dessus est exhaustive et   toutes ces formes sont obtenues. 
     
     Pr\'ecis\'ement, 
  soit $D$ une sous-alg\`ebre de composition de dimension $2$ de $V$.
\begin{itemize}
    \item Si $D$ est un plan hyperbolique et $D^\perp= W^+ \oplus W^-$ la polarisation compl\`ete associ\'ee, pour toute strate semi-simple  $[\Lambda_{W^+}, n,r,\beta_{W^+}]$ de $\msl(W^+)$,
    la strate  $[\check\Lambda_{W^+}, n,r,\check\beta_{W^+}]$ est une strate semi-simple  de $\mathfrak g_2(F)$, dite {\em ``strate de type $D$''}.
    
    En outre, si  $[\Lambda_{W^+}, n,r+1,\gamma_{W^+}]$  est une strate semi-simple    de $\msl(W^+)$ \'equivalente \`a $[\Lambda_{W^+}, n,r+1,\beta_{W^+}]$, alors la strate semi-simple $[\check\Lambda_{W^+}, n,r+1,\check\gamma_{W^+}]$  de $\mathfrak g_2(F)$  est \'equivalente \`a
    $[\check\Lambda_{W^+}, n,r+1,\check\beta_{W^+}]$.
    
        \item Si $D$ est une extension quadratique de $F$ et $W=D^\perp$, pour tout strate semi-simple  $[\Lambda_{W }, n,r,\beta_{W }]$ de $\su(W )$,
    la strate  $[\vec\Lambda_{W }, n,r,\vec\beta_{W }]$ est une strate semi-simple  de $\mathfrak g_2(F)$, dite {\em ``strate de type $D$''}.
    
    En outre, si  $[\Lambda_{W }, n,r+1,\gamma_{W }]$  est une strate semi-simple    de $\su(W )$ \'equivalente \`a $[\Lambda_{W }, n,r+1,\beta_{W }]$, alors la strate semi-simple  $[\vec\Lambda_{W }, n,r+1,\vec\gamma_{W }]$  de $\mathfrak g_2(F)$  est \'equivalente \`a   
     $[\vec\Lambda_{W }, n,r+1,\vec\beta_{W }]$.
\end{itemize}
  Toute    strate  semi-simple de $\mathfrak g_2(F)$ est obtenue  par l'un ou l'autre de ces proc\'ed\'es.
     \end{Theorem}

    \begin{proof}
    Le proc\'ed\'e donn\'e  associe \`a une strate  de $\msl(W^+)$ ou
   $\su(W )$ une strate   de   $\mathfrak g_2(F)$  : c'est une cons\'equence directe des injections d'alg\`ebres de Lie rappel\'ees au paragraphe  \ref{strates.1} (notations \ref{betachapeau1} et \ref{betachapeau2})  d'une part et des constructions de normes faites au paragraphe
   \ref{strates.normes}  (notations \ref{Lambdachapeau1} et \ref{Lambdachapeau2}) d'autre part. Puisque les \'el\'ements semi-simples consid\'er\'es sont prolong\'es par $0$ sur $D$,
   ces constructions (\ref{Lambdachapeau1}) et (\ref{Lambdachapeau2}) assurent que  la valuation  relativement \`a la suite de r\'eseaux consid\'er\'ee est conserv\'ee : c'est $-n$.
    Ce proc\'ed\'e est
  compatible \`a l'\'equivalence de strates par \cite[Proposition 2.9]{BK2}.

    Il nous faut encore v\'erifier que la strate obtenue est semi-simple pour le m\^eme entier $r$, $0 \le r <  n$ (le cas des strates nulles, $r=n$, est imm\'ediat), en revenant \`a la
     d\'efinition de 
    \cite{S4} rappel\'ee au paragraphe \ref{strates.2}. Remarquons    que la somme d'une strate semi-simple et d'une strate nulle est toujours une strate semi-simple 
    et qu'en outre une strate semi-simple de $\su(W, D/F )$ est aussi une strate semi-simple de $\End_F(W)$, avec une d\'ecomposition peut-\^etre moins fine de $W$. 
  Le r\'esultat est donc clair sauf si     $D$ est un plan hyperbolique~: il faut alors montrer que 
    si   $[\Lambda_{W^+}, n,r,\beta_{W^+}]$ est une strate semi-simple de $\msl(W^+)$, 
    alors $[\Lambda_{W^+}\oplus \Lambda_{W^-}, n,r,\beta_{W^+}+ \beta_{W^-}]$ est une strate semi-simple de 
    $\mathfrak{s  o}_F( W^+\oplus W^-)$, dans les notations du lemme \ref{centralGL} que nous utilisons ci-apr\`es. Il s'agit d'une d\'emonstration au cas par cas, le r\'esultat ne serait pas vrai si $\beta_{W^+}$  n'\'etait pas de trace nulle. 
    Nous raisonnons comme dans le lemme \ref{centralGL} sur les polyn\^omes minimaux $P^+$ et $P ^-$ de $\beta_{W^+}$ et $\beta_{W^-}$.  
        \begin{description}
\item[--] Si $P^+$ est irr\'eductible, il est de degr\'e $3$ et l'\'el\'ement $\beta_{W^+}$
est minimal (lemme \ref{centraldim4}) ; les polyn\^omes caract\'eristiques des strates 
$[\Lambda_{W^+} , n,n-1,\beta_{W^+} ]$ et $[  \Lambda_{W^-}, n,n-1,\beta_{W^-}]$ sont premiers entre eux, la somme est bien semi-simple. 
\item[--] Si $P^+$ a deux diviseurs, ils correspondent \`a une d\'ecomposition de $\beta $ en somme des \'el\'ements simples 
$\beta_1^+$, de sous-espace caract\'eristique $W_1^+$  de dimension $1$ et $\beta_2^+$, de sous-espace caract\'eristique $W_2^+$ de dimension $2$ ; $\beta_2^+$ ne peut pas \^etre nul. Si 
$\beta_1^+$ est nul, alors $\beta_2^+$ est de trace nulle, donc minimal dans l'extension quadratique qu'il engendre, et  dans la d\'ecomposition correspondante de $W^-$, 
$\beta_2^-$ est conjugu\'e \`a $\beta_2^+$. Alors la strate 
$$[\Lambda_{W^+}\oplus \Lambda_{W^-}, n,n-1,\beta_{W^+}+ \beta_{W^-}]$$
est semi-simple, somme d'une strate nulle en dimension $2$ et d'une strate simple en dimension $4$, attach\'ee \`a $\beta_2^+$.

Si 
$\beta_1^+$ est non  nul, alors $\beta_{W^+}+ \beta_{W^-}$ a quatre sous-espaces caract\'eristiques distincts. Les polyn\^omes caract\'eristiques des strates attach\'ees \`a 
$\beta_1^+$ et $\beta_1^- = - \beta_1^+$ sont premiers entre eux, et de m\^eme pour celles attach\'ees \`a $\beta_2^+$ et $\beta_2^- $, conjugu\'e de $- \beta_2^+$, 
donc un d\'efaut de semi-simplicit\'e de la strate $[\Lambda_{W^+}\oplus \Lambda_{W^-}, n,r,\beta_{W^+}+ \beta_{W^-}]$  ne pourrait provenir que d'une somme ``mixte'' 
$\beta_1^+ + \beta_2^-$ (ou $\beta_1^- + \beta_2^+$ ce qui revient au m\^eme). 
Soient $-n_1$ et $-n_2$ les valuations de $\beta_1^+ $ et $ \beta_2^+$ par rapport aux suites de r\'eseaux $\L \cap W_1^+$ et  $\L \cap W_2^+$ respectivement. Ces suites ont m\^eme p\'eriode sur $F$ donc $\tr( \beta_2^+ )= -\beta_1^+$ entra\^ine $-n_1 \ge -n_2$. 
Par semi-simplicit\'e de la strate initiale on a $r < n_1$. Si donc la strate attach\'ee \`a $\beta_1^+ + \beta_2^-$ \'etait \'equivalente \`a une strate simple il s'agirait d'une strate scalaire attach\'ee \`a un \'el\'ement $\lambda$ congru \`a $\beta_1^+$  d'une part, \`a 
$\frac 1 2 \tr( \beta_2^- ) = \frac 1 2 \beta_1^+$ d'autre part : c'est impossible. 
\item[--] Si $P^+$ a trois diviseurs, ils correspondent \`a trois valeurs propres scalaires 
 $\beta_i^+$,  $i=1,2,3$, de somme nulle et sous-espaces propres
  $W_i^+$. 
 Si l'une est nulle, les deux autres sont oppos\'ees de sorte que les trois valeurs propres dans $W^-$, oppos\'ees des pr\'ec\'edentes, sont globalement les m\^emes : 
  la strate somme sera semi-simple avec d\'ecomposition en trois sous-espaces propres, chacun de dimension $2$. 
  
  Si aucune n'est nulle, soit $-n_i$   la valuation de $\beta_i^+ $   par rapport \`a la suite de r\'eseaux $\L \cap W_i^+$ ; notons que ces suites ont m\^eme p\'eriode sur $F$. A renum\'erotation pr\`es on peut supposer 
  $-n_1= -n_2 \le -n_3$, de sorte que la strate initiale est semi-simple pour  $0 \le r < n_3$.  Un d\'efaut de semi-simplicit\'e de la strate $[\Lambda_{W^+}\oplus \Lambda_{W^-}, n,r,\beta_{W^+}+ \beta_{W^-}]$  ne peut provenir que d'une somme  
$\beta_i^+ + (-\beta_j^+)$ avec $i \ne j$ donnant une strate \'equivalente \`a une strate simple. C'est alors forc\'ement une strate scalaire attach\'ee \`a un \'el\'ement $\lambda$ congru \`a $\beta_i^+$  et \`a $-\beta_j^+$. 
Les valuations de ces deux \'el\'ements sont alors \'egales et la somme $\beta_i^+ + \beta_j^+$ est congrue \`a $0$, il ne peut donc s'agir que de $\{i, j\} = \{1,2\} $ et on a $n_3 < n_1$. Ecrivons $\beta_2^+ = -\beta_1^+ - \beta_3^+$ : la congruence \`a $\lambda$ ne peut avoir lieu que pour $r \ge n_3$, c'est donc impossible. 

\end{description}

  Reste \`a voir que toute strate semi-simple non nulle  de $\mathfrak g_2(F)$ s'obtient ainsi.
  Pour une strate $[\L, n, r, \beta]$ telle que le noyau de $\beta$ est  de dimension $2$, c'est d\'ej\`a clair. Examinons  les cas o\`u le noyau $V^0$ de $\beta$ est de dimension $4$, c'est-\`a-dire les cas (iii) et (iv).
  Notons $W$ l'orthogonal de $V^0$ et rappelons que $\beta_W^2$ op\`ere par un scalaire $u $.  
  Notons enfin $\alpha$ la norme d'alg\`ebre autoduale correspondant \`a la suite de r\'eseaux $\Lambda$, $\alpha_0$ et $\alpha_W$ ses restrictions \`a $V^0$ et $W$ respectivement.  
\begin{itemize}
 \item Si $V^0$ est  isotrope  et si $u$  est le carr\'e de $\lambda \in F^\times$,
    les sous-espaces propres $W_\lambda$ et $W_{-\lambda}$ de $\beta_W$ forment une polarisation compl\`ete de $W$ qui scinde $\Lambda\cap W$. Soit $W_\lambda = F h \oplus F k$ un scindage de $\Lambda \cap W_\lambda $ et soit $(h', k')$ la base duale de $W_{-\lambda}$. 
   D'apr\`es  le lemme \ref{idempotents} on a $\beta(h-h') = c(h-h')$ avec 
   $c = \lambda c(Fh, Fh')$ et 
  d'apr\`es 
    la d\'emonstration de la proposition \ref{dimension4} (ii), la norme   $\alpha$ est scind\'ee 
    par $V = F[c] \perp F[c]^\perp$, c.q.f.d.

          \item Si $V^0$ est  isotrope  et si $u$ n'est pas un carr\'e de $F$, soit 
          $(h, h')$ une base de Witt de $W$ sur $F[\beta_W]   $ scindant $\alpha_W$. Elle d\'etermine $d=c(Fh, Fh') \in V^0$ (lemme \ref{idempotents}) et la d\'emonstration de    \ref{dimension4} (ii) montre que $\alpha$ est scind\'ee par $V = F[d] \perp F[d]^\perp$. D'autre part $\beta$, qui commute \`a la multiplication \`a gauche par $ V^0$,  est    $F[d]$-lin\'eaire, il conserve 
          $e^+(Fh, Fh')F[d]^\perp $ et   $e^-(Fh, Fh')F[d]^\perp $. La strate est donc aussi une strate de type $F[d]$.

      \item Si $V^0$ est anisotrope, soit $D$ une sous-extension quadratique de $F$ dans $V^0$  non isomorphe \`a $F[\beta_W ]$ et soit     $a \in W$ de norme $1$. 
      L'application $x \mapsto xa$ est une isom\'etrie de $V^0$ sur $W$ telle que 
      $\alpha_0(x)=\alpha_W(xa)$, ces normes sont donc scind\'ees 
      respectivement par $V^0= D \perp D^\perp\cap V^0$ et $W= Da \perp (D^\perp\cap V^0)a$
      (proposition \ref{dimension4} (i)).  La strate $[\Lambda_{W}, n,r, \beta_{W}]$ est aussi bien une strate simple dans $\su(W, D/F)$. 
      On prolonge $\beta_W$ par $0$ sur $D^\perp\cap V^0$ d'o\`u une strate semi-simple   de 
      $\su(W \perp D^\perp\cap V^0, D/F)$ dont l'image par le proc\'ed\'e (ii) de la proposition est exactement notre strate de d\'epart. 
    \end{itemize}
    \end{proof}

Le proc\'ed\'e ci-dessus  est compatible \`a l'\'equivalence des strates, donc \`a l'approximation, dont il nous faut dire un mot. Dans le premier cas, on sait par   \cite[Proposition 3.4]{S4}
qu'il existe  une strate semi-simple $[\L_{W^+}, n, r+1, \gamma_{W^+}]$ dans
$\mathfrak{gl}(W^+)$ \'equivalente \`a la strate $[\Lambda_{W^+}, n,r+1,\beta_{W^+}]$
  et telle que la d\'ecomposition  de
 $W^+$  associ\'ee \`a $\gamma_{W^+}$  s'obtienne \`a partir de  celle associ\'ee \`a $\beta_{W^+}$
   par regroupement \'eventuel de blocs. Puisque $\beta_{W^+}-\gamma_{W^+}$ est un \'el\'ement de $\fA_{-(r+1)}(\Lambda)$ et que $\beta_{W^+}$ est de trace nulle, l'\'el\'ement $\tr (\gamma_{W^+})I$ appartient aussi \`a $\fA_{-(r+1)}(\Lambda)$
 (cf. \cite[Proposition 2.9]{BK2}) donc on peut remplacer $\gamma_{W^+}$
 par  $\gamma_{W^+}- \frac 1 3 \tr (\gamma_{W^+})$ qui appartient \`a $\msl(3,F)$, d'o\`u l'existence d'une strate \'equivalente dans $\mathfrak{sl}(W^+)$.

Dans le second cas  on trouve de m\^eme un $\gamma_{W}$ dans $\mfu(W)$
via  \cite[\S 3.6]{S4} puis on le
 remplace  par $\gamma_{W}- \frac 1 3 \tr (\gamma_{W})$ qui appartient \`a    $\su(W)$, d'o\`u l'existence d'une strate \'equivalente dans $\su(W)$.

Ces remarques, en compl\'ement du th\'eor\`eme,   assurent que :

 \begin{Corollary}\label{raffinement1}
 Le processus d'approximation des strates peut rester interne aux strates semi-simples de $\mathfrak g_2(F)$ : 
 soit $D$ une sous-alg\`ebre de composition de dimension $2$ de $V$ et 
 soit $[\L, n, r, \beta]$, $0\le r < n$,  une   strate semi-simple de $\mathfrak g_2(F)$ de type $D$.
 La strate $[\L, n, r+1, \beta]$ est \'equivalente \`a une strate semi-simple $[\L, n, r+1, \gamma]$ de   $\mathfrak g_2(F)$ de type $D$.

  \end{Corollary}

\begin{Remark} 
{\rm 
 Lorsque $V^0$ est de dimension $4$ et d\'eploy\'ee nous venons de montrer 
que la strate $[\L, n, r, \beta]$ est toujours de type $D$ pour une sous-alg\`ebre de composition   $D$ hyperbolique de dimension $2$ convenable. Elle pourrait aussi \^etre de type $D'$ pour une extension quadratique $D'$ convenable. De toutes fa\c cons, 
pour une strate telle que $V^0$ soit de dimension $4$,  le proc\'ed\'e ci-dessus n'est pas canonique et ne nous dispensera pas d'un traitement particulier, notamment pour    l'\'etude du centralisateur de $\beta$. 
La diff\'erence de nature entre les quatre types de strates ci-dessus appara\^\i t d'ailleurs clairement dans la d\'emonstration du th\'eor\`eme \ref{carfixes}.  }
\end{Remark}

\setcounter{section}{1}
       
\Section{Trialit\'e et caract\`eres semi-simples de $\SO_F(V)$} \label{par2}  

Soit $[\L, n, 0, \beta]$ une strate semi-simple  de $\mathfrak g_2(F)$ au sens du paragraphe \ref{strates.2}. Nous voulons lui associer une famille de caract\`eres semi-simples dans le groupe $\G2(F)$, obtenus \`a l'aide de l'action de la trialit\'e sur une famille de caract\`eres semi-simples dans le groupe $\SO_F(V)$. Ceci n\'ecessite une \'etude pr\'ecise de 
la trialit\'e afin d'en sp\'ecifier l'action  sur les objets relatifs \`a la strate  dans le groupe $\SO_F(V)$ : les sous-groupes ouverts compacts   $P^i(\L)$, $  H^i(\beta, \L)$, $ J^i(\beta, \L)$ ($i \ge 1$)   et  les 
  caract\`eres semi-simples  de $   H^1(\beta, \L)$.  
  
\subsection{La trialit\'e et le groupe $\Spin_F(V)$}\label{par21} 
La r\'ef\'erence pour ce paragraphe est \cite[\S3]{SV}. 
 Rappelons pour m\'emoire  que la {\it norme 
spinorielle} est l'homomorphisme de $\SO_F(V)$ dans $F^\times / (F^\times)^2$ qui \`a une r\'eflexion orthogonale d\'efinie par un \'el\'ement non isotrope $a$ de $V$ associe la classe de $Q(a)$. On notera  $\O^\prime_F(V):=G^\prime $ son noyau, le {\it groupe orthogonal r\'eduit}. Commen\c cons par  la d\'efinition de la trialit\'e :

\begin{Definition} 
Soit $t_1$ un \'el\'ement de $\SO_F(V)$. Il existe $ t_2$ et $t_3$ dans $\SO_F(V)$ v\'erifiant 
\begin{equation}\label{trialite}
t_1(xy) = t_2(x) t_3(y) \quad \text{pour tous } x, y \in V 
\end{equation}
si et seulement si la norme spinorielle de $t_1$ est triviale ; 
dans ce cas $t_2$ et $t_3$ appartiennent eux aussi \`a 
$\O^\prime_F(V)$. On appelle {\em triplet reli\'e}  tout triplet 
$(t_1,t_2,t_3)$ d'\'el\'ements de $\SO_F(V)$ v\'erifiant (\ref{trialite}).
Ce sont alors des \'el\'ements de  $\O^\prime_F(V)$. 
\end{Definition}

 Soit $\Spin_F (V) $   le groupe des points sur $F$ du groupe alg\'ebrique $\mathbf{Spin} (V)$, d\'efini sur $F$
    (le rev\^etement simplement connexe de $\mathbf{SO}(V)$). 
Dans le cas de l'alg\`ebre d'octonions la trialit\'e fournit le mod\`ele   de  $\Spin_F (V) $, 
en termes de triplets reli\'es, que nous utiliserons  : 
  $$
  \Spin_F (V) = \{ (t_1, t_2, t_3) \in \SO_F(V)^3/ \ \forall x, y \in V  \  t_1(xy)=t_2(x)t_3(y) \} 
  $$  
et l'application $\pi $ de premi\`ere projection fournit   la suite exacte
fondamentale \cite[\S 3.6 et 3.7]{SV}
\begin{equation}\label{spin}
1 \longrightarrow \{\pm 1 \}
\longrightarrow \Spin_F (V)
\stackrel{\pi}{\longrightarrow}
 \O^\prime_F(V) \longrightarrow 1
\end{equation}    

Pour tout  $t$ appartenant \`a $ \SO_F(V)$ on d\'efinit $\hat t \in \SO_F(V)$ par $\hat t(x) = \overline{t(\bar x)}$ ($x \in V$). Alors $t \mapsto \hat t$ est un automorphisme involutif de 
$\SO_F(V)$ et si $(t_1, t_2, t_3)$ est un triplet reli\'e, les triplets 
\begin{equation}\label{triplets}
(t_2, t_1, \hat t_3) , \   (t_3, \hat t_2, t_1), \  (\hat t_1, \hat t_3, \hat t_2), \  
 (\hat t_2, t_3, \hat t_1), \  (\hat t_3, \hat t_1, t_2) \end{equation}
  sont eux aussi des triplets reli\'es. 
On d\'efinit le groupe 
  $\Gamma$  des {\it automorphismes   de  trialit\'e } comme le groupe d'automorphismes de 
    $\Spin_F (V) $ engendr\'e par 
    $(t_1, t_2, t_3) \mapsto   (\hat t_2, t_3, \hat t_1)$, 
    d'ordre $3$, et $(t_1, t_2, t_3) \mapsto   (t_2, t_1, \hat t_3)$, d'ordre $2$. 
Il  est   isomorphe
\`a $\mathfrak S_3$ et son groupe de points fixes est isomorphe \`a $\G2(F)$ via la projection $\pi$ :
$$
1  \longrightarrow \Spin_F (V)^\Gamma
\stackrel{\pi}{\longrightarrow}
 \G2(F) \longrightarrow 1
$$
Puisque $p$ est impair, les groupes $\SO_F(V)$ et $\O^\prime_F(V)$ ont les m\^emes 
pro-$p$-sous-groupes et   d'apr\`es \cite[Th\'eor\`eme 9.7.3]{W}, l'extension (\ref{spin}) 
 est scind\'ee au-dessus de tout  pro-$p$-sous-groupe $H$,   et ceci de fa\c con unique :  l'image de l'unique section homomorphe $s_H$ de 
 $H$ dans $\pi^{-1}(H)$ est le pro-$p$-Sylow de $\pi^{-1}(H)$, 
 qui est distingu\'e et d'indice $2$ dans $\pi^{-1}(H)$.  
 Ainsi, si $\pi^{-1}(H)$ est stable par les automorphismes   de  trialit\'e, il en est de m\^eme de $s_H(H)$ et 
on peut  consid\'erer $\Gamma$  comme un groupe d'automorphismes de $H$, engendr\'e par 
$t_1 \mapsto t_2 $ et $t_1 \mapsto \hat t_2 $ par exemple, o\`u  $(t_1, t_2, t_3)$ est un triplet reli\'e d'\'el\'ements de $H$. C'est sous cette forme que la trialit\'e appara\^it 
le plus souvent dans la suite.

Par d\'erivation on obtient un groupe
$\mathbf \Gamma$ d'automorphismes de l'alg\`ebre de Lie
de $\Spin_F (V)$, isomorphe \`a $\mathfrak{s   o}_F(V)$, 
qui peut de m\^eme \^etre d\'efini \`a partir de la notion de {\it triplets reli\'es}.   Pour tout \'el\'ement $t_1$ de 
  $\mathfrak{s   o}_F(V)$, il existe $t_2, t_3 \in  \mathfrak{s   o}_F(V)$, uniques, tels que : 
\begin{equation}\label{trialiteLie}
t_1(xy) = t_2(x)y + x t_3(y) \quad \text{pour tous } x, y \in V . 
\end{equation}  
Le triplet $(t_1, t_2, t_3)$ d'\'el\'ements de  $\mathfrak{s   o}_F(V)$ est alors dit {\it reli\'e}, et les triplets de la liste \ref{triplets} le sont aussi. 
Les automorphismes de trialit\'e sont les automorphismes de la forme $t_1 \mapsto t_i$ ou $\hat t_i$, $i=1,2,3$, o\`u $(t_1, t_2, t_3)$ est un triplet reli\'e. Leur 
 action   a pour 
  points fixes  les d\'erivations de $V$, c'est-\`a-dire l'alg\`ebre de Lie de
$\G2(F)$ : $\mathfrak g_2(F) = {\mathfrak{s   o}_F(V)}^{\mathbf \Gamma}$.

Dans la suite, l'expression ``fixe par trialit\'e'' signifie
 ``fixe par $\Gamma$ ou $\mathbf \Gamma$''.

\subsection{Trialit\'e et base de Witt}\label{par22}

Nous partons du mod\`ele de l'alg\`ebre d'octonions reproduit dans \cite[\S5]{GY}. 
Il comporte un espace vectoriel $W^+$ de dimension $3$ sur $F$, muni d'une base $ \{e_1 ,  e_2 ,  e_3\} $ 
et son dual $W^-$ muni de la base  $\{e_{-1} ,  e_{-2 },  e_{-3} \}$ d\'etermin\'ee par $e_{-i}(e_j)=-\delta_{ij}$. On 
 identifie  $\wedge^3 W^+$ \`a $F$ en convenant de  $e_1 \wedge e_2 \wedge   e_3 = 1$, d'o\`u des identifications $\wedge^2 W^+ \simeq W^-$ et  $(\wedge^2 W^+)^\ast \simeq W^+$. 
L'isomorphisme de $\wedge^2 W^-  $ sur  $(\wedge^2 W^+)^\ast  $ donn\'e par 
$$ (\phi_1\wedge \phi_2, w_1 \wedge w_2) \mapsto \phi_1(w_1) \phi_2 (w_2) - 
\phi_1(w_2)\phi_2(w_1)  
$$
fournit enfin un isomorphisme de $\wedge^2 W^-  $ sur $W^+$. 
Avec ces conventions, l'alg\`ebre $V$ peut \^etre d\'ecrite comme   l'espace des matrices 
$$
\left( \begin{matrix}a & w \cr \phi & b 
\end{matrix}\right) \qquad 
a, b \in F, \quad w \in W^+ , \  \phi \in W^-, 
$$ 
muni du produit 
$$
\left( \begin{matrix}   a & w \cr \phi & b  
\end{matrix}\right)
\left( \begin{matrix}   a' & w' \cr \phi' & b'  
\end{matrix}\right)
=
\left( \begin{matrix}   aa' + \phi'(w) & aw'+b'w - \phi\wedge\phi'  \cr a'\phi +b\phi' + w \wedge w'& 
b b' + \phi(w') 
\end{matrix}\right). 
$$
 Ce mod\`ele est particuli\`erement adapt\'e \`a l'\'etude de la trialit\'e sur les automorphismes stabi\-lisant $W^+$ et $W^-$. Il est facile de v\'erifier que 
 pour une isom\'etrie $t_1$ de la forme   : 
 $$  \iota \,(u, g) \, (\left( \begin{matrix}a & w \cr \phi & b 
\end{matrix}\right)  ) = \left( \begin{matrix}  u^{-1} a & gw \cr \phi g^{-1}& u b 
\end{matrix}\right) \qquad  u \in F^\times, \  g \in \GL_F(W^+),  
 $$
 l'\'equation (\ref{trialite}) a une solution si et seulement si $u  \, \det  g $ est un carr\'e, la solution \'etant donn\'ee, pour un choix de $\lambda \in F^\times$  tel que 
 $\lambda^2 = u  \, \det  g $, par : 
\begin{equation}\label{GLW} 
 t_2= \iota \, ( \lambda^{-1}u, \lambda^{-1} g), \quad  
 t_3 = \iota \, ( \lambda , u \lambda^{-1}g). 
\end{equation}
Bien entendu on retrouve le fait que les rotations  stabi\-lisant $W^+$ et $W^-$ et fix\'ees par trialit\'e forment un groupe canoniquement isomorphe \`a $\SL_F(W^+)$ (\S \ref{composition}). 

Il s'agit maintenant de choisir une base de $V$ permettant de calculer la trialit\'e sur les sous-groupes radiciels standard. 
En posant $e_{-4} = \left( \begin{smallmatrix}   1 & 0 \cr 0 & 0  
\end{smallmatrix}\right)$ et $e_{4} = \left( \begin{smallmatrix}   0 & 0 \cr 0 & 1  
\end{smallmatrix}\right)$, on obtient une base de Witt $\{e_i\}$ de $V$.  
 Pour $\lambda \in F$ et $i, j \in \{ \pm 1, \pm 2, \pm 3, \pm 4\}$ avec $ i \ne \pm j$, 
on pose 
$$
u_{i,j}(\lambda)(e_i) = e_i + \lambda e_{-j}, \quad  
u_{i,j}(\lambda)(e_j) =  e_j - \lambda e_{-i}, \quad   
u_{i,j}(\lambda)(e_k) = e_k \ \text{ pour } k \ne i, j. 
$$
 
\begin{Lemma}\label{radiciels}  
Si 
$i$ et $j $ sont de signe oppos\'e et distincts de $\pm 4$, $u_{i,j}(\lambda)$  est fix\'e par   trialit\'e. 
Sinon,  notant  
  $(i, i^{+} , i^{++}) $  l'un des triplets $ \, (1,2,3) $,  $(2, 3,1) $,   $(3,1,2) $, les triplets suivants sont reli\'es :   
$$ \ ( u_{-i, -i^{++}} (\lambda), \,  u_{4, i^{+}} (\lambda), \ u_{i^{+}, -4} (\lambda)) \   
\text{ et } 
  \ ( u_{ i,  i^{++}} (\lambda), \, u_{-4,- i^{+}} (\lambda), \  u_{-i^{+},  4} (\lambda)) \ .$$
Il en est de m\^eme dans l'alg\`ebre de Lie $\mathfrak{s   o}_F(V)$ en rempla\c cant 
$u_{i,j}(\lambda)$ par $U_{i,j}(\lambda)= u_{i,j}(\lambda)-I$. 
\end{Lemma} 
\begin{proof} 
Si $i$ et $j $ sont distincts de $\pm 4$, $u_{i,j}(\lambda)$ appartient \`a $\SL_F(W^+)$. 
La suite est sans doute bien connue  (\cite[\S 6]{GY} par exemple),  cependant nous   pr\'ef\'erons inclure une 
version pr\'ecise. 
Il suffit de v\'erifier que pour tous $k, l$ on a 
$t_1(e_k e_l) = t_2(e_k) t_3(e_l)$. On s'appuie sur la table de multiplication   de l'alg\`ebre qui est donn\'ee comme suit, avec  $(i, i^{+} , i^{++}) $ comme ci-dessus et $j $ un entier dans $\{1,2,3\}$ : 
$$
\begin{aligned}
&e_ie_i= e_{-i}e_{-i} = 0  \qquad \  
e_ie_{i^{+}}= -e_{i^{+}}e_{i } = e_{-i^{++}} \qquad  \   e_{-i}e_{-i^{+}}= -e_{-i^{+}}e_{-i } = e_{i^{++}}  
\\
&e_{-j}e_i = -\delta_{ij}e_4  \quad e_{ j}e_{-i} = -\delta_{ij}e_{-4}  \qquad
e_{-4}e_{-4}=e_{-4}  \quad e_{4} e_{4} = e_{4}    \quad
e_{-4}e_{ 4}= e_{ 4}e_{-4}=0   
\\ 
&e_{-4}e_{i}=e_{ i}e_{ 4}=  e_{i} \qquad  e_{-4}e_{-i}= e_{-i} e_{ 4}=0 \qquad 
e_{-i}e_{-4}=e_{ 4}e_{-i} = e_{-i} \quad  e_{i}e_{-4}=e_{ 4}e_ i =0     
\end{aligned}
$$
En particulier, un produit $e_{-i^{+}} e_j$, avec   $j\ne i^+$ et $j \ne -4$,  est non nul si et seulement si 
$j = -i $ ou $j=-i^{++}$, point essentiel du calcul.   
\end{proof}

Pour $\lambda \in F^\times$ et $i \in \{   1,  2,  3,   4\}$, soit 
$d_i(\lambda)$ l'\'el\'ement de $\SO_F(V)$ envoyant 
$e_i$ sur $\lambda e_i  $, $e_{-i}$ sur $ \lambda^{-1}e_{-i}$  et fixant   les autres vecteurs de base. 
La trialit\'e sur ces \'el\'ements est un cas particulier de (\ref{GLW}).  
On d\'efinit de m\^eme $D_i(t) \in  \mathfrak{so}_F(V)$   par $D_i(t)(e_i)= t e_i$, 
$D_i(t)(e_{-i})= - t e_{-i}$ et $D_i(t)(e_j)= 0$ pour $j \ne \pm i$ ($t \in F$). 
On inclut par commodit\'e le calcul suivant : 
\begin{Lemma}\label{diagonaux}  
Les triplets d'\'el\'ements de  $\mathfrak{so}_F(V)$ suivants sont reli\'es  ($i=1,2,3  $) : 
$$
\begin{aligned}
  & ( \ D_i(t), \ D_i(t/2) + \sum_{\begin{smallmatrix}   j \in  \{   1,   2,  3,   4\} \cr 
j \ne i 
  \end{smallmatrix}} D_j(-t/2), \  
  \ D_4(t/2) + \ D_i(t/2) + \sum_{\begin{smallmatrix}   j \in  \{   1,   2,  3,   4\} \cr 
j \ne i, 4  
  \end{smallmatrix}} D_j(-t/2) \   )   , 
  \\
  &( \  D_4(t), \ D_4(t/2) + \sum_{  j \in  \{   1,   2,  3 \}  } D_j(-t/2), \  
  \sum_{   j \in  \{   1,   2,  3,   4\}  
 } D_j( t/2) \   ) . 
  \end{aligned}
$$
 \end{Lemma}
 
 \subsection{Trialit\'e et transformation de Cayley}\label{par225}
 Soit $[\L, n, 0, \beta]$ une strate semi-simple  de $\mathfrak g_2$.   L'alg\`ebre  $\End_F(V)$ 
 admet une filtration par les  
    r\'eseaux $\tilde {\mathfrak A}_i(\Lambda)= \{ g \in   \End_F(V)   
 / \forall k \in \mathbb Z \ g \Lambda(k) \subseteq \Lambda(k+i) \} $ ($i \in \mathbb Z$), et 
 le sous-groupe ouvert compact $\tilde P (\L)= \{ g \in  \GL_F(V) 
 / \forall k \in \mathbb Z \ g \Lambda(k) =  \Lambda(k) \}$ admet une filtration par  
 les pro-$p$-sous-groupes  $$\tilde P^i(\L)= \{ g \in  \GL_F(V) 
 / \forall k \in \mathbb Z \ \  (g-1) \Lambda(k) \subseteq \Lambda(k+i) \} \quad (i \ge 1). $$

  De m\^eme l'alg\`ebre  $\mathfrak{s   o}_F(V)$ 
 admet une filtration par les  
    r\'eseaux $ \mathfrak A_i(\Lambda) = \tilde {\mathfrak A}_i(\Lambda) \cap \mathfrak{s   o}_F(V)$ 
      et 
 le sous-groupe ouvert compact $P (\L)= \tilde P (\L) \cap \O_F(V)$
   admet une filtration par  
 les pro-$p$-sous-groupes   $P^i(\L)= \tilde P^i(\L) \cap \O_F(V)$, contenus dans $\O^\prime_F(V)$.

 D'apr\`es \cite[ \S 11 et Corollaire 11.14]{GY}, le groupe $\Gamma$ agit par trialit\'e sur l'immeuble 
 de $\Spin_F (V) $ de fa\c con compatible \`a son action sur  $\Spin_F (V) $ et l'immeuble de $\G2(F)$ s'identifie  \`a l'ensemble des points de l'immeuble de $\Spin_F (V) $ qui sont fix\'es par $\Gamma$.  Le sous-groupe $\pi^{-1}(P^\prime( \L))$ de $\Spin_F (V) $,  fixateur du point attach\'e \`a $\L$ dans l'immeuble de  $\Spin_F (V) $, est donc, comme ce point, stable par trialit\'e. 
  Notons $s$ l'unique section homomorphe de  $P^1(\L)$ dans 
 $\Spin_F (V)$. Alors  $s(P^1(\L))$,  plus grand  pro-$p$-sous-groupe distingu\'e de 
  $\pi^{-1}(P^\prime( \L))$, est stable par trialit\'e et l'unicit\'e de $s$ permet de d\'efinir de mani\`ere unique l'action de la trialit\'e sur $P^1(\L)$ (\S \ref{par21}).
 Le but de ce paragraphe 
 est de montrer que les sous-groupes  $P^i(\Lambda)$  ($i \ge 1$) sont eux aussi stables par trialit\'e 
 et que les groupes de points fixes obtenus ont les propri\'et\'es esp\'er\'ees 
 (Proposition \ref{diagcom}). 
 
La suite de r\'eseaux $\L$ est scind\'ee par une base de Witt comme ci-dessus \cite[Proposition 8.1]{GY}. Elle est associ\'ee \`a une norme d'alg\`ebre autoduale $\alpha$ sur $V$ et \`a sa  fonction de r\'eseau $\lambda_\alpha$,  et 
il existe un entier $m$ tel que l'on ait  pour tout $i \in \mathbb Z$ : 
$\L(i)= \lambda_\alpha(\frac {i} {m})$ 
(\S \ref{strates.normes}). Notons $\lceil x \rceil$ le plus petit entier sup\'erieur ou \'egal au nombre r\'eel $x$. 
 D'apr\`es \cite[Lemma 3.2]{L} les r\'eseaux $ \mathfrak A_k(\Lambda)$   sont alors  donn\'es par 
 $$
  \mathfrak A_k(\Lambda) = \bigoplus_{i=1}^4 D_i\left(\pF^{\lceil \frac {k} {m} \rceil}\right) \  \oplus \ 
  \bigoplus_{\begin{smallmatrix} i, j \in  \{ \pm 1, \pm 2, \pm 3, \pm 4\} \cr i \ne \pm j
  \end{smallmatrix}} U_{i,j} \, \left(\pF^{\lceil\frac {k} {m}+\alpha(e_i)+\alpha(e_j)\rceil}\right) 
 $$
  Leur partie diagonale est stable  par trialit\'e par le lemme \ref{diagonaux}, puisque les  $D_i$ portent tous la m\^eme valuation. 
 Certaines parties $U_{i,j}$ sont   fix\'ees individuellement par la trialit\'e ; les autres  doivent \^etre regroup\'ees  trois par trois  comme dans le lemme \ref{radiciels}  et la stabilit\'e   par trialit\'e 
 \'equivaut \`a montrer, dans les notations de ce lemme, que 
 $$
 \begin{aligned}
  \alpha(e_{-i})+\alpha(e_{-i++})&= \alpha(e_{4})+\alpha(e_{ i +})= \alpha(e_{ i+})+\alpha(e_{-4}) \\  \text{ et }  \ \ 
  \alpha(e_{i})+\alpha(e_{i++}) =& \alpha(e_{-4})+\alpha(e_{- i +})= \alpha(e_{ -i+})+\alpha(e_{4}). 
 \end{aligned}
 $$  
 Or ceci d\'ecoule du lemme \ref{151} et de la   proposition  \ref{volume0sl} qui nous  assurent que $\alpha(e_{ 4})=\alpha(e_{-4})=0$,  $\alpha(e_{1})+ \alpha(e_{2})+ \alpha(e_{3})=0$ et $ \alpha(e_{i})+ \alpha(e_{-i}) =0$. 
 
 \medskip 
 
Pour $k\ge 1$ la  transformation  de Cayley $X \mapsto  C(X) = (1+\frac X 2) (1-\frac X 2)^{-1}$ 
  est une bijection de $ \mathfrak A_k(\Lambda)$ sur $P^k(\Lambda)$ \cite[2.13]{M}. D'autre part $P^k(\Lambda)$ a une d\'ecomposition d'Iwahori par rapport au tore diagonal et au sous-groupe de Borel standard dans la base de Witt donn\'ee, il est donc engendr\'e par ses intersections avec les sous-groupes $d_i(F^\times)$ et $u_{i,j}(F)$. 
  Enfin, on a 
\begin{equation}\label{cdcu}
   C (D_i(\lambda))= d_i( C(\lambda)) \quad (\lambda \in \pF) \quad \text{ et } 
  \quad 
  C (U_{i,j}(\lambda))= u_{i,j}( \lambda) \quad (\lambda \in F). 
\end{equation}
  Le groupe $P^k(\Lambda)$ est donc stable par trialit\'e, puisqu'il poss\`ede une famille de g\'en\'erateurs stable par trialit\'e :  les $d_i\left(1+\pF^{\lceil \frac {k} {m} \rceil}\right)$ et les $U_{i,j} \, \left(\pF^{\lceil\frac {k} {m}+\alpha(e_i)+\alpha(e_j)\rceil}\right)$. 
 Nous venons de montrer la premi\`ere assertion de la proposition suivante, qui 
 est au groupe $\G2(F)$ ce que les lemmes 3.1 et 3.2 de \cite{S2} sont aux groupes  classiques :

\begin{Proposition}\label{diagcom}
Soient $r$ et $s$ des entiers tels que $1 \le r \le s \le 2r$. 
\begin{enumerate}
\item $ \mathfrak A_r(\Lambda)$   et  $P^r(\Lambda)$  sont stables par trialit\'e. 
\item L'isomorphisme de groupes ab\'eliens : 
$ 
\mathfrak A_r(\Lambda) / \mathfrak A_s(\Lambda)  \stackrel{\simeq}{\longrightarrow}  
P^r(\Lambda) / P^s(\Lambda)
$ 
induit par la bijection de Cayley   commute \`a la trialit\'e. 
  Il induit un isomorphisme de groupes ab\'eliens 
$$
  {\mathfrak A}_r(\Lambda)^{\mathbf \Gamma} /    {\mathfrak A}_s(\Lambda)^{\mathbf \Gamma} 
\stackrel{\simeq}{\longrightarrow} 
  P^r(\Lambda)^{ \Gamma} /  P^s(\Lambda)^{ \Gamma}. 
$$
\item Soit $\psi $ un caract\`ere additif    de $F$ de conducteur $\pF$. 
L'isomorphisme $P(\Lambda)$-\'equivariant de groupes ab\'eliens  
$$
\begin{aligned}
\mathfrak A_{1-s}(\Lambda) / \mathfrak A_{1-r}(\Lambda)   \quad   &\stackrel{\simeq}{\longrightarrow}  \quad
\left(P^r(\Lambda) / P^s(\Lambda)\right)^\wedge \\
b +  \mathfrak A_{1-r}(\Lambda) \quad  &\longmapsto \qquad \psi_b  \qquad  \qquad (\psi_b(x) = \psi\circ\tr (b(x-1))) 
\end{aligned}
$$
commute \`a la trialit\'e :  $\psi_{d\nu (b)}(x) = \psi_b( \nu^{-1} (x))$ pour tout $\nu \in \Gamma$ de diff\'erentielle 
$d\nu \in \mathbf \Gamma$. 
\end{enumerate}
\end{Proposition}
\begin{proof} (ii) 
Le transform\'e de Cayley de $x \in \mathfrak A_r(\Lambda)$ est congru \`a $1+x$ modulo  $P^s(\Lambda)$, l'isomorphisme consid\'er\'e est bien celui de \cite[2.13]{M} et 
 \cite[Lemma 3.1]{S2} ; en particulier la bijection de Cayley devient un   morphisme par passage au quotient.   Soit $\nu $ une application de trialit\'e sur $P^1(\Lambda)$ et 
 $d\nu$ sa diff\'erentielle.   Pour \'etablir que  $  C (d\nu ( x))$ est congru 
 \`a $\nu (  C(x))$ modulo $ P^s(\Lambda) $  pour tout $x \in \mathfrak A_r(\Lambda)$,  il suffit de le montrer sur des g\'en\'erateurs, ce qui est imm\'ediat par (\ref{cdcu}) et les lemmes  \ref{radiciels} et \ref{diagonaux}. 
 
 On obtient alors un isomorphisme 
 entre les groupes de points fixes des quotients sous $\mathbf \Gamma$ et $ \Gamma$ 
 respectivement. Il reste \`a voir que 
 $\mathfrak A_{r}(\Lambda)^{\mathbf \Gamma} / \mathfrak A_{s}(\Lambda)^{\mathbf \Gamma}  \simeq \left[\mathfrak A_{r}(\Lambda) / \mathfrak A_{s}(\Lambda)\right]^{\mathbf \Gamma} $ (et de m\^eme du c\^ot\'e droit), c'est-\`a-dire que les points fixes du quotient peuvent se relever dans $\mathfrak A_{r}(\Lambda)^{\mathbf \Gamma}$. Il suffit encore une fois d'\'ecrire le groupe commutatif  $\mathfrak A_{r}(\Lambda) / \mathfrak A_{s}(\Lambda)$ comme somme directe de 
 sous-groupes fixes ou stables par trialit\'e (la partie diagonale, les parties en $U_{i,j}$ fixes, et les 
 sommes de trois parties en $U_{i,j}$ li\'ees par trialit\'e). 
 
 (iii) L'isomorphisme en question est celui de 
  \cite[Lemma 3.2 (ii)]{S2} et le groupe $\left(P^r(\Lambda) / P^s(\Lambda)\right)^\wedge$ est le dual de Pontrjagin de
 $\left(P^r(\Lambda) / P^s(\Lambda)\right) $ ; l'action naturelle de $\Gamma$ sur un \'el\'ement $\psi $ de ce  dual est 
 bien donn\'ee par $\nu . \psi = \psi\circ \nu^{-1}$. 
 
 Rappelons  que pour tous $X, Y \in  \mathfrak{s  o}_F(V)$ et pour tout 
 $d\nu \in \mathbf \Gamma$ on a 
 $\tr \, X Y = \tr \,  d\nu(X) d\nu(Y)$, o\`u la trace est prise dans $\End_F(V)$. 
 (Ces deux formes bilin\'eaires non d\'eg\'en\'er\'ees sur $\mathfrak{s   o}_F(V)$ d\'eterminent chacune un isomorphisme de $\mathfrak{s   o}_F(V)$-modules entre la repr\'esentation adjointe et la repr\'esentation coadjointe  de $\mathfrak{s   o}_F(V)$,  alg\`ebre de Lie simple : elles sont donc proportionnelles, or elles co\"\i ncident sur $\mathfrak g_2(F)$.) 
 On calcule   alors, pour $x \in P^r(\Lambda) $ : 
 $$
 \begin{aligned}
 \psi_{d\nu (b)}(x)=  \psi\circ\tr \left( d\nu(b) (x-1)   \right) &= 
 \psi\circ\tr \left( b \, d\nu^{-1} (x-1)     \right) \\
 &=\psi\circ\tr \left( b d\nu^{-1}( C^{-1}(x))     \right) \\ 
 &= \psi\circ\tr \left( b C^{-1}(\nu^{-1}( x))  \right) \\
  &= \psi\circ\tr \left( b  (\nu^{-1}( x)-1)   \right)
 = \psi_b( \nu^{-1} (x)). 
 \end{aligned}
 $$ 
\end{proof}
 
 \begin{Remark}\label{Moy}
{\rm  Consid\'erons un \'el\'ement $X$ de $\msl_F(W^+) \subset \lg2(F)$ tel que   $X(e_i)= u_i e_i$ pour $i=1,2,3$ avec $u_i \in F$, $v_F( u_i) \ge 1$,  et $u_1+u_2+u_3 =0$. 
 Sa transform\'ee de Cayley $  C(X)  $   appartient \`a $\GL_F(W^+)$ canoniquement identifi\'e \`a un sous-groupe de $\SO_F(V)$. Ainsi $  C(X)$ appartient \`a $\G2(F)$ si et seulement si sa restriction \`a $W^+$ est de d\'eterminant $1$. Or il est facile de voir que $u_1+u_2+u_3 =0$ n'entra\^\i ne pas $  C(u_1)   C(u_2) C(u_3) =1$ (prendre par exemple $u_1=u_2=u, u_3=-2u$). 
 
 La transform\'ee de Cayley n'applique donc pas l'alg\`ebre de Lie $\lg2(F)$ dans le groupe $\G2(F)$. 
 C'est la raison pour laquelle nous ne pouvons pas invoquer le travail ant\'erieur de Moy ({\it Minimal $K$-types for $\G2$ over a $p$-adic field}, Trans. A.M.S. 305(2) (1988), 517--529), bas\'e sur l'assertion contraire.  La proposition \ref{diagcom} requiert une d\'emonstration. 
 }
 \end{Remark}

 \subsection{Caract\`eres semi-simples autoduaux}\label{par23}
 
 Poursuivons avec notre strate semi-simple $[\L, n, 0, \beta]$, ou plut\^ot, provisoirement, 
 avec     une strate semi-simple  $[\L, n, r, \beta]$ de $\mathfrak{so}_F(V)$. 
 Ce paragraphe  est de fait valide dans le cadre g\'en\'eral d'un groupe classique, c'est-\`a-dire celui de \cite{S5} : c'est une cons\'equence imm\'ediate des constructions faites dans  {\it loc.cit.} pour une strate semi-simple gauche et 
 adapt\'ees   dans \cite[\S 8.2]{D} au cas d'une strate semi-simple autoduale. Dans un souci de simplicit\'e, on se contente de l'\'enoncer dans le contexte et les notations d\'ej\`a d\'efinis.

 Rappelons la d\'ecomposition 
   de $V$ associ\'ee \`a $\beta$ (\ref{decomposition}) :
$$
 V = \left[\perp_{i=0}^s V^i \right] \perp
\left[\perp_{j=1}^k ( V^{s + 2j-1} \oplus V^{s + 2j})\right] 
$$
o\`u chaque $V^i$ pour $i \le s$,  ou $ \  V^{s + 2j-1} \oplus V^{s + 2j} \ $ pour $1 \le j \le k $,  
est non d\'eg\'en\'er\'e  et orthogonal \`a tous les autres, et 
$ V^{s + 2j-1} ,V^{s + 2j}  $ sont totalement isotropes en dualit\'e.  
Pour tout $g \in \GL_F(V^{s + 2j-1})$ on notera $\iota_j(g)$ l'unique \'el\'ement de 
$\SO_F(V)$ prolongeant $g$ et agissant trivialement sur 
$V^a$ pour tout $a$ distinct de $s + 2j-1$ et $ s + 2j$. Ceci d\'efinit une injection canonique 
$$\iota_j : \GL_F(V^{s + 2j-1}) \hookrightarrow \SO_F(V).$$

Les sous-groupes ouverts compacts  
$\widetilde H^i(\beta, \L)$, $\widetilde J^i(\beta, \L)$ ($i \ge 1$) et 
 $\widetilde J(\beta, \L)$  de $\GL_F(V)$   attach\'es \`a la strate 
  sont d\'efinis dans  \cite[\S 3.2]{S4} (apr\`es la d\'efinition moins g\'en\'erale de  \cite[\S 3.6]{BK2}).  
Ces sous-groupes sont inva\-riants par l'automorphisme $\tau$ de points fixes $O_F(V)$ 
\cite[\S 3.6]{S4} \cite[\S 8.2]{D}.  
Notons  
$H^i(\beta, \L)={\widetilde H^i(\beta, \L)}^\tau$, $J^i(\beta, \L)={\widetilde J^i(\beta, \L)}^\tau$ ($i \ge 1$) 
 : ce sont des sous-groupes ouverts compacts  de $\O^\prime_F(V)$. 
 
 \begin{Lemma}{\bf \cite[Corollary 3.12]{S4}}\label{betagamma} 
 Soit 
$[\L, n, r, \beta]$ une strate semi-simple autoduale.  Il existe une strate semi-simple autoduale 
 $[\L, n, r+1, \gamma]$  \'equivalente \`a $[\L, n, r+1, \beta]$ et respectant la 
 d\'ecomposition initiale (\ref{decomposition}) de $V$. On a alors :
 $$
 \begin{aligned}
     H^{i}(\beta, \L)&=    H^{i }(\gamma, \L) 
    \text{ pour }   i  > \frac{r+1}{2} ,  \quad  
     J^{i}(\beta, \L)=    J^{i }(\gamma, \L) 
    \text{ pour }   i  \ge \left[\frac{r+1}{2}\right] , \\ 
    H^{i}(\beta, \L)&= (P^{i}(  \L) \cap G_\beta) \;  H^{[\frac{r+1}{2}] +1}(\gamma, \L) 
 \quad    \text{ pour } 1 \le i  \le \frac{r+1}{2},    \\   
    J^{i}(\beta, \L)&= (P^{i}(  \L) \cap G_\beta) \;  J^{[\frac{r+1}{2}] }(\gamma, \L) 
    \quad    \text{ pour } 1 \le i  \le \left[\frac{r+1}{2}\right]  . 
       \end{aligned}
 $$
 \end{Lemma}
 \begin{proof} 
 L'\'enonc\'e analogue pour les groupes $\tilde H^{i}(\beta, \L)$ et $\tilde J^{i}(\beta, \L)$ 
 est pr\'ecis\'ement \cite[Corollary 3.12]{S4} dont les deux premi\`eres \'egalit\'es
  d\'ecoulent. Pour les suivantes il ne s'agit que de v\'erifier que l'on peut prendre les points fixes de l'involution $\tau$ terme \`a terme dans le produit. Le raisonnement est devenu classique : posons $H = \tilde P^{i}(  \L) \cap \tilde  G_\beta $ et 
 $U = \tilde H^{[\frac{r+1}{2}] +1}(\gamma, \L) $. Alors $H$ normalise $U = \tilde H^{[\frac{r+1}{2}] +1}(\beta, \L) $ donc pour tout 
 $h \in H$ on a $UhU \cap H = (U\cap H) h$. Il reste \`a appliquer \cite[Theorem 2.3]{S2} : 
 $(UH)^\tau = U^\tau H^\tau$.  
 \end{proof}

  Pour $0\le m < r$,   notons $\tilde \CC(\L, m,  \beta)$   l'ensemble des caract\`eres semi-simples de   $\tilde H^{m+1}(\beta, \L)$ 
  \cite[Definition 3.13]{S4}.  Il est stable   par l'involution $\tau$  \cite[\S 8.2]{D}, ce qui permet de  d\'efinir 
   l'ensemble $  \CC(\L, m,  \beta)$  des caract\`eres semi-simples autoduaux de   $ H^{m+1}(\beta, \L)$
   exac\-tement comme dans le cas 
  semi-simple gauche \cite[\S 3.6]{S4} :  c'est l'image par la   corres\-pondance de Glauberman   
   de  l'ensemble $\tilde \CC(\L, m,  \beta)^\tau$ des caract\`eres semi-simples $\tau$-invariants de   $\tilde H^{m+1}(\beta, \L)$.   L'usage de la correspondance de Glauberman permet d'affirmer 
  que la restriction \`a   
  $H^{m+1}(\beta, \L)$  d\'etermine une bijection : 
  $$ \tilde \CC(\L, m,  \beta)^\tau \stackrel{\simeq}\longrightarrow \CC(\L, m,  \beta). 
  $$
Puisque  l'action de 
  $\tau$ conserve   $\tilde \CC(\L, m,  \beta)$,   l'ensemble 
 $  \CC(\L, m,  \beta)$ est aussi l'ensemble des res\-trictions des caract\`eres semi-simples  de  $\tilde H^{m+1}(\beta, \L)$, sans condition d'invariance. En effet, si $\tilde \theta $ appartient \`a $\tilde \CC(\L, m,  \beta)$, il en est de m\^eme de $\tilde \theta^\tau$ ; le produit $\tilde \theta \tilde \theta^\tau$ appartient \`a $\tilde \CC(\L, m,  2 \beta)$ 
 et poss\`ede une unique ``racine carr\'ee'', \'el\'ement de  $\tilde \CC(\L, m,  \beta)$ et ayant m\^eme restriction \`a $H^{m+1}(\beta, \L)$ que $\tilde \theta$. On a d\'etaill\'e ce raisonnement parce qu'il n'est pas valide pour $\G2(F)$, on le verra plus loin. 
 En attendant, le
  lemme ci-dessous,  
 cons\'equence imm\'ediate de \cite[Definition 3.13]{S4}, 
explicite dans le point (i)(b) la  diff\'erence entre caract\`eres semi-simples gauches et   autoduaux :    
\begin{Lemma}\label{defss} Soit 
$[\L, n, r, \beta]$ une strate semi-simple autoduale dans $\mathfrak{so}_F(V)$.  Il existe une strate semi-simple autoduale 
 $[\L, n, r+1, \gamma]$  \'equivalente \`a $[\L, n, r+1, \beta]$ et respectant la d\'ecomposition initiale (\ref{decomposition}) de $V$.  Soit $m$ tel que $0\le m < r$. 
Un caract\`ere semi-simple autodual~$\theta$ de $H^{m+1}(\beta, \L)$  
v\'erifie : 
\begin{enumerate}
	\item (a) Pour $i \le s$ la restriction de $\theta$ \`a 
	$$H^{m+1}(\beta, \L)\cap \SO_F(V^i)= H^{m+1}(\beta_i, \L^i)$$ appartient 
	\`a   $\CC(\L^i, m,  \beta^i)$; 
	
	(b) pour $1 \le j \le k $ la restriction de $\theta$ \`a  
	$$H^{m+1}(\beta, \L)\cap \iota_j \left(\GL_F(V^{s + 2j-1} )\right) = \iota_j \left( \tilde H^{m+1}(\beta^{s + 2j-1}, \L^{s + 2j-1}) \right) $$  
	appartient 
	\`a   $\iota_j \left( \tilde \CC(\L^{s + 2j-1}, m, 2 \beta^{s + 2j-1}) \right) $. 
	\item Si $m' = \max\{ m, [(r+1)/2]\}$ la restriction de $\theta$ \`a $H^{m'+1}(\beta, \L)$ 
	est de la forme $\theta_0 \psi_{\beta-\gamma}$ pour un $\theta_0 \in \CC(\L, m',  \gamma)$. 
\end{enumerate}
\end{Lemma} 
   
 \subsection{Trialit\'e et sous-alg\`ebres de composition}\label{par24}
 Le mod\`ele de l'alg\`ebre d'octonions utilis\'e au paragraphe \ref{par22} permet de d\'ecrire la trialit\'e   sur le stabilisateur dans $G'$ de la d\'ecomposition 
 $V = V^0 \perp (W^+ \oplus W^-)$, pour toute {\it sous-alg\`ebre  de composition hyperbolique de dimension $2$ de $V$} (\ref{GLW}). Nous voulons de la m\^eme fa\c con obtenir une description   de la trialit\'e adapt\'ee \`a chaque type de sous-alg\`ebre de composition de $V$. 
 
 Commen\c cons par {\it une sous-alg\`ebre de composition $V^0$ de dimension $4$} et orthogonal $W$. Choisissons un \'el\'ement $a$ de norme non nulle de $W$ de sorte que $W = V^0a$. Soient $t_1$, $t_2$, $t_3$ des   \'el\'ements de $\SO_F(V)$ stabilisant $V^0$ ; on les \'ecrit sous la forme (\S \ref{composition}) : 
 $$
 t_i(x+ya)= \alpha_i u_ixu_i^{-1} + (\delta_i v_iyv_i^{-1}) \, a \qquad (x,y \in V^0) 
 $$
 avec $i \in \{ 1,2,3\}$, $\alpha_i, \delta_i \in (V^0)^1$, $u_i, v_i \in (V^0)^\times$. 
 En utilisant la formule de doublement (\ref{double}) pour d\'evelopper l'\'equation de trialit\'e  (\ref{trialite}) on constate   que, \`a $t_1$ fix\'e,  (\ref{trialite}) a une solution en $t_2$, $t_3$ si et seulement si $Q(u_1/v_1)$ est un carr\'e. Les deux  solutions de (\ref{trialite}) correspondent alors au choix d'une  racine carr\'ee $\xi $ de $Q(u_1/v_1)$  : 
 \begin{equation}\label{trialitedim4}
\left\{\begin{aligned} t_2(x+ya)&= \xi^{-1}\alpha_1 u_1xv_1^{-1} + (\xi\delta_1 v_1yu_1^{-1}) \,  a ,  \\ 
 t_3(x+ya)&= \xi  v_1xu_1^{-1} + (\xi\delta_1 v_1yu_1^{-1}\alpha_1^{-1})\, a, 
 \end{aligned}\right.
 \qquad   \xi^2 = Q(u_1/v_1) . 
 \end{equation}
 
 \bigskip
 Passons au cas d'{\it une sous-alg\`ebre de composition $V^0$ de dimension $2$ et anisotrope}. Son suppl\'ementaire $W$ est un espace vectoriel de dimension $3$ sur $V^0$ op\'erant par multiplication \`a gauche, et la multiplication \`a gauche par les \'el\'ements de $V^0$ est 
 $V^0$-lin\'eaire sur $W$ (\S \ref{composition})~: c'est cette  structure d'espace vectoriel 
 sur $V^0$ qui est utilis\'ee dans la suite   sauf mention expresse du contraire.
 Il faut commencer par \'etablir une formule de produit adapt\'ee \`a la d\'ecomposition 
 $V = V^0 \perp W$. 
   
   Soit $c$ un \'el\'ement non nul et de trace nulle de $V^0$, de sorte que $V^0=F[c]$. 
Soit $\Phi_g$ l'unique forme $V^0$-hermitienne \`a gauche (i.e. pour la structure d'espace vectoriel obtenue par produit \`a gauche) sur $V$ telle que 
  $f = \tr_{V^0/F} \circ \Phi_g$, et soit $\Phi_d$ son analogue \`a droite (i.e. pour la structure d'espace vectoriel obtenue par produit \`a droite sur $V$). Elles sont 
    donn\'ees par : 
  $$\Phi_g(x,y)= \frac 1 2 \left( f(x,y) + c^{-1} f(cx,y) \right) , \quad 
  \Phi_d(x,y)= \frac 1 2 \left( f(x,y) + c^{-1} f(xc,y) \right)
  \qquad (x, y \in V). 
  $$
 Les deux formes co\"\i ncident sur $V^0$ ; pour $x, y \in  W$ on a $xc=-cx$ et 
 $\Phi_d(x,y)= \overline{\Phi_g (x,y)}$. De $f(xy,z)=f(y, \bar x z)= f(x, z \bar y)$ 
 ($x, y ,x \in V$) on d\'eduit ais\'ement : 
 $$\Phi_g(w w', z)= -\Phi_d(w', wz)= -\Phi_d(w'\bar z, w ) \quad \text{ pour tous } 
 w, w' \in W, \  z \in V. 
 $$
 Soient donc $w, w'$ deux \'el\'ements de $W$, cherchons la d\'ecomposition du produit $w w'$ sur $V^0 \perp W$. La composante sur $V^0$ est bien entendu 
 $$\Phi_g(w w', \un)= -\Phi_d(w', w ) =   -\overline{\Phi_g (w',w)}= -\Phi_g(w ,w' ).$$
 La composante sur $W$ est $p(ww')$ o\`u $p$ est la projection orthogonale de $V$ sur $W$. 
 Nous pr\'ef\'erons en donner une expression plus semblable \`a celle du paragraphe \ref{par22} en termes de produit ext\'erieur.  
 Sur le mod\`ele de  \cite[\S5]{GY} cit\'e plus haut, on identifie   $\wedge^3 W $  \`a $V^0$ en fixant une base orthogonale $\{a, b, ab\}$ de $W$ sur $V^0$ 
 (cf. proposition \ref{volume0su}) et en normalisant par $a \wedge b \wedge ab = Q(ab)$. Cela fournit un isomorphisme de $\wedge^2 W $  sur $W^*$ (dual de $ W$ sur $V^0$) identifiant, 
 pour $w, w' \in W$, l'\'el\'ement 
 $w \wedge w'$ \`a la forme lin\'eaire 
 $z \mapsto w \wedge w' \wedge z$ ($z \in W$). Enfin, cette forme lin\'eaire s'\'ecrit de fa\c con unique sous la forme $z \mapsto \Phi_g(z, w \bar \wedge w')$ et  
 l'\'el\'ement $w \bar \wedge w'$ ainsi d\'efini v\'erifie 
 $$
 (\lambda w )\bar \wedge (\lambda'w') = \overline{\lambda \lambda'} \  (w \bar \wedge w')  \quad 
 \text{ i.e. } \quad 
 (w\lambda  ) \bar \wedge (w'\lambda')  =  \lambda \lambda'   \ (w \bar \wedge w') 
 \qquad  
(\lambda, \lambda' \in V^0). 
$$
 Nous obtenons de la sorte une application bilin\'eaire 
 (pour la structure \`a droite de $W \times W$ et \`a gauche de $W$) : 
 $$ W \times W  \longrightarrow W \qquad (w , w') \mapsto w \bar \wedge w'. $$ 
 La projection $p$ v\'erifie elle aussi : 
 $$
 p((w\lambda )( w'\lambda'))= p( [w( w'\lambda')]\bar \lambda) =
 p(w(\bar \lambda' w') ) \bar \lambda = 
 p( \lambda'(ww')) \bar \lambda = \lambda \lambda' p(ww'). 
 $$
  Ces deux applications bilin\'eaires co\"\i ncident sur les \'el\'ements de la base  $\{a, b, ab\}$ (simple v\'erification) donc sont \'egales. 
La formule de produit recherch\'ee s'\'ecrit : 
\begin{equation}\label{barwedge}
(v_0 + w )\ (v'_0 + w' ) = \  [v_0v'_0 -\Phi_g(w ,w' )] \  +  \  [v_0w' + w v'_0 + w \bar \wedge w'] 
\quad (v_0, v'_0 \in V^0, w , w' \in W). 
 \end{equation}
Soient \`a pr\'esent $t_1$, $t_2$, $t_3$ des   \'el\'ements de $\SO_F(V)$ stabilisant $V^0$ 
et $V^0$-lin\'eaires \`a gauche ; on les \'ecrit sous la forme  : 
 $$
 t_i(v_0+w)= \lambda_i v_0 + g_i w  \qquad (v_0 \in V^0, \  w \in W) 
 $$
 avec $i \in \{ 1,2,3\}$, $\lambda_i \in (V^0)^1= \{x \in V^0/ Q(x)=1\}$, $g_i \in \U(W, \Phi_g)$. Notons que le conjugu\'e de $t_i$ est donn\'e par $\hat t_i(v_0+w)= \bar \lambda_i v_0 + g_i w$.
 On d\'eveloppe \`a l'aide de la formule  de produit (\ref{barwedge}) l'\'equation de trialit\'e  (\ref{trialite}). Elle \'equivaut \`a : 
 $$ \lambda_1=\lambda_2 \lambda_3 \, , \ \  \lambda_1 \Phi_g(w ,w' ) = \Phi_g(g_2w ,g_3w' ) \, , 
 \ \  g_1= \lambda_2g_3 = \bar \lambda_3 g_2 \, , \ \ g_1(w \bar \wedge w') = g_2w \bar \wedge g_3w'
 \, , 
 $$
 pour tous $w, w' \in W$. Comme l'application $\bar \wedge$ v\'erifie $g w \bar \wedge g w'
 = \overline{\det_{V^0}(g)} \, g( w \bar \wedge  w')$ pour tout $g \in \GL_{V^0}(W)$, on  
 conclut   que $(t_1, t_2, t_3)$ comme ci-dessus forment un triplet reli\'e si et seulement si $\lambda_1 \, \overline{\det_{V^0}(g_1)}$ est le carr\'e d'un \'el\'ement
 $\xi = \lambda_3$  de norme $1$ de $V^0$ et 
 \begin{equation}\label{trialitedim2}
\left\{\begin{aligned} t_2(v_0+w)&= \xi^{-1}\lambda_1 v_0 + \xi \, g_1 w ,  \\ 
 t_3(v_0+w)&= \xi   v_0 + \xi \bar \lambda_1 \, g_1 w ,  
 \end{aligned}\right.
 \qquad  \xi \in V^0, \ Q(\xi)=1, \  \xi^2 = \lambda_1 \, \overline{\det_{V^0}(g_1)}   . 
 \end{equation}

 \subsection{Trialit\'e et caract\`eres semi-simples}\label{par25}
Nous sommes maintenant pr\^ets \`a r\'esoudre le probl\`eme de la stabilit\'e des 
caract\`eres semi-simples autoduaux par trialit\'e. 
Soit $[\L, n, 0, \beta]$ une strate semi-simple  de $\mathfrak g_2(F)$ 
et soit $V = V^0 \perp W$ la d\'ecomposition de $V$ corres\-pondante : $V^0$ est le noyau de $\beta$ 
(\S \ref{strates.3}). 
Rappelons que si $V^0$ est de dimension $2 $ et d\'eploy\'ee, $W$ poss\`ede une polarisation compl\`ete canonique $W = W^+ \oplus W^-$, d'o\`u une injection cano\-nique $\iota$ de $\GL_F(W^+)$ dans $\SO_F(W)$.

  \begin{Lemma}\label{ex2.1}
  Soit $[\L, n, r, \beta]$ une strate semi-simple  de $\mathfrak g_2(F)$.
  Les sous-groupes $H^i(\beta, \L)$ et  $J^i(\beta, \L)$, $i \ge 1$,  sont stables par trialit\'e.  
 \end{Lemma}
 \begin{proof}  Comme d'habitude ceci se d\'emontre par induction sur une suite d'approxi\-ma\-tion de la strate. Pour une strate nulle, la proposition   \ref{diagcom} nous assure que les groupes 
 $P^i(\L)  $ sont stables par trialit\'e, celle-ci \'etant bien d\'efinie sur le pro-$p$-groupe $P^1(\L)$ 
 (\S \ref{par21}) qui contient  les sous-groupes consid\'er\'es. 
 Soit  alors $[\L , n, r+1, \gamma ]$  
 une strate semi-simple de $\lg2(F)$   \'equivalente \`a $[\L , n, r+1, \beta ]$   (corollaire \ref{raffinement1}).  On a 
  $J^i(\beta, \L) = (P^i(\L) \cap G_\beta) J^{[\frac{r+1}{2}]} (\gamma, \L)$
   et une relation analogue pour les $ H^i$ (Lemme \ref{betagamma}). Supposons $J^{[\frac{r+1}{2}]} (\gamma, \L)$ 
    stable par trialit\'e.  
 Comme $\beta$ est un \'el\'ement de $\mathfrak{so}_F(V)$ fix\'e par l'action de   $\mathbf \Gamma$, son centralisateur dans $\Spin_F (V) $ est stable par l'action de $\Gamma $ (\S \ref{par21}) et son intersection avec $P^i(\L) $ 
 (via l'unique section de $P^i(\L) $ dans $\Spin_F (V)$) l'est aussi, $J^i(\beta, \L)$ est donc stable par trialit\'e, de m\^eme que $H^i$.  
 \end{proof}

\begin{Theorem}\label{carfixes}
 Soit $\theta$ un caract\`ere semi-simple autodual de 
$H^1(\beta, \Lambda)$. Alors les images de $\theta$ par les applications de trialit\'e sont   des caract\`eres semi-simples autoduaux de 
$H^1(\beta, \Lambda)$
si et seulement si l'une des conditions suivantes est r\'ealis\'ee :  
\begin{enumerate}
	\item $V^0$ est de dimension $4$ ou $8   $ ; 
	\item $V^0$ est anisotrope  de dimension $2$   et la restriction de $\theta$ \`a $ H^1(\beta, \Lambda) \cap \U(W, V^0/F)$ est   triviale sur l'intersection 
	de $ H^1(\beta, \Lambda)$ avec le centre de $\U(W, V^0/F)$ ; 
	\item $V^0$ est d\'eploy\'ee  de dimension $2$    et la restriction de $\theta$ \`a $ H^1(\beta, \Lambda) \cap \iota(\GL_F(W^+))$ est   triviale sur l'intersection de $H^1(\beta, \Lambda)$ 
	avec le centre de $\iota(\GL_F(W^+))$. 
\end{enumerate}
En outre, dans ces conditions, le caract\`ere $\theta$ est en fait {\em fixe} par trialit\'e. 
\end{Theorem}

On appellera {\it caract\`ere semi-simple sp\'ecial } de 
$H^1(\beta, \Lambda)$ un caract\`ere semi-simple autodual de 
$H^1(\beta, \Lambda)$ v\'erifiant une des conditions du th\'eor\`eme. 
On notera $\mathscr C^0(\Lambda, \beta)$ 
l'ensemble  des caract\`eres semi-simples sp\'eciaux de $H^1(\beta, \Lambda)$. 
Ces caract\`eres sont fixes 
 par trialit\'e. 

\begin{proof} 
Commen\c cons par rappeler que par d\'efinition, tout caract\`ere 
semi-simple autodual de $H^1(\beta, \Lambda)$  est trivial sur  $H^1(\beta, \Lambda) \cap \SO(V^0)= P^1(\Lambda\cap V^0)$,  
puisque la strate $[\L, n, 0, \beta]$  est somme d'une strate {\it nulle} sur $V^0$ et de strates simples non nulles sur $W$. C'est pourquoi les conditions donn\'ees sont n\'ecessaires. En effet, si $V^0$ est anisotrope de dimension~$2$, l'intersection de $H^1(\beta, \Lambda)$ avec le centre de $\U(W, V^0/F)$ s'identifie  \`a 
$$(1 +  \mathfrak p_{V^0})^1 = \{ \mu \in   1 +  \mathfrak p_{V^0} / Q(\mu)=1\}$$ via 
$\mu  \mapsto 
[v_0+w \mapsto  v_0 + \mu w ]$   ($v_0 \in V^0$, $ w \in W$). Soit $t_1(v_0+w)= \lambda^2 v_0 +  w$ 
avec $\lambda \in   (1 +  \mathfrak p_{V^0})^1$. D'apr\`es (\ref{trialitedim2}), 
$\hat t_2 : v_0+w \mapsto v_0 + \lambda w $ est image de $t_1$ par trialit\'e, donc $\theta (\hat t_2) $ doit \^etre trivial, c.q.f.d. 
 Le cas (iii) se traite de fa\c con analogue en utilisant (\ref{GLW}). 
 
 Montrons maintenant que ces conditions sont suffisantes, c'est-\`a-dire que {\it si $\nu$ est une application de trialit\'e et $\theta$ un caract\`ere 
semi-simple autodual de $H^1(\beta, \Lambda)$ v\'erifiant (i), (ii) ou (iii), alors $\theta\circ \nu$ est encore un caract\`ere 
semi-simple autodual de $H^1(\beta, \Lambda)$.}  

La proposition
  \ref{diagcom} (iii) est le premier ingr\'edient de la d\'emonstration  puisque la restriction de 
$\theta$ \`a $P^{[n/2]+1}(\L )$ est \'egale \`a  $\psi_\beta$. Elle fournit : 
\begin{equation}\label{bottom}
 \theta\circ \nu(x) = \theta(x)  \quad \text{ pour tout }   x \in P^{[n/2]+1}(\L ). 
 \end{equation}
 
Commen\c cons par le cas (i) o\`u  il n'y a pas de condition suppl\'ementaire sur $\theta$. 
Si $\beta$ est nul,  le caract\`ere  est trivial. Si $\dim V^0=4$,    la strate  $[\L, n, n-1, \beta]$ 
est semi-simple (lemme \ref{centraldim4}) de sorte que 
\begin{equation}\label{H1dim4}
\begin{aligned}
H^1(\beta , \L ) &= \left( P^1(\L^0) \times H^1(\beta_W, \L_W) \right) P^{[n/2]+1}(\L ) \\
 &= \left( P^1(\L^0) \times  (P^1(\L_W)\cap G_{\beta_W}) \right) P^{[n/2]+1}(\L ).
\end{aligned}
\end{equation}
Le caract\`ere $\theta$ est trivial sur $P^1(\L^0) $. D'apr\`es la formule de trialit\'e 
  (\ref{trialitedim4}),  si $t_1$ appartient \`a $\SO_F(V^0) \times \SO_F(W)$, les actions de $t_1$ et  $\nu(t_1)$ sur $V^0$ ou $W$ diff\`erent d'une multiplication \`a droite, de la forme $y \mapsto yz$ ou 
  $y a \mapsto (yz) a $ ($y \in V^0$) dans les notations de la formule, par un \'el\'ement $z$ de $V^0$  et l'on a $Q(z)=1$ puisqu'on reste dans un groupe sp\'ecial orthogonal.

  Supposons   $[\L_W, n, 0, \beta_W]$ simple (cas (\ref{scalairesu})  du paragraphe \ref{strates.4}). Sur $ P^1(\L_W)\cap G_{\beta_W} $ le caract\`ere simple 
  $\theta_{|  H^1(\beta_W, \L_W)}$ 
  factorise par un caract\`ere du d\'eterminant  : $G_{\beta_W} \longrightarrow  
  F[\beta_W]$ \cite[Definition 3.2.1]{BK1}.  Le d\'eterminant de la multiplication \`a droite par $z$ dans $V^0$ vu comme espace vectoriel sur $ F[\beta_W] $ est la norme de $z$ dans l'extension $ F[\beta_W] [z]$, soit 
  $Q(z)=1$. Finalement $\theta \circ \nu = \theta$, d'o\`u la stabilit\'e voulue. 
  
 Sinon (cas (\ref{scalairegl2})  du paragraphe \ref{strates.4}) 
  l'action de $\beta_W$ a deux sous-espaces propres, qui sont  $   W_\lambda=( e^+_{F[c]}V^0)a \, $ et  $ \, W_{-\lambda}= ( e^-_{F[c]} V^0)a \,  $  
 dans les notations de la d\'emonstration du  lemme~\ref{dim4}. Le caract\`ere $ \theta_{|   H^1(\beta_W, \L_W)}$  poss\`ede  une d\'ecomposition d'Iwahori par rapport au sous-groupe de Levi $M_W$ stabilisant ces sous-espaces propres et \`a un sous-groupe parabolique de facteur de Levi $M_W$\cite[\S 5.3]{S5}.
 Sur l'intersection avec  $M_W$, 
 identifi\'ee \`a  $\tilde H^1(\lambda, \L_{W_\lambda})$, il factorise par un caract\`ere du d\'eterminant.  Une multiplication \`a droite 
 par $ z \in V^0$ conserve $e^+_{F[c]}V^0$ et $e^-_{F[c]} V^0$ et un petit calcul dans    $V^0$, d\'ecompos\'e comme d'habitude en 
 $V^0 = F[c] \perp F[c]b$ avec $Q(b) \ne 0$, montre que le d\'eterminant de sa restriction \`a 
 $e^+_{F[c]}V^0$ est $Q(z)$. On a donc comme pr\'ec\'edemment $\theta \circ \nu = \theta$.

  \bigskip 
  Pla\c cons-nous \`a pr\'esent dans le cas (ii).  Soit $r$, $0 \le r \le n$ le plus grand entier tel que la strate  $[\L, n, r, \beta]$ soit semi-simple : on raisonne par induction
   (d\'ecroissante) sur $r$ pour montrer qu'un caract\`ere semi-simple autodual 
   trivial sur l'intersection 
	de $ H^1(\beta, \Lambda)$ avec le centre de $\U(W, V^0/F)$
  est fixe par trialit\'e. Si $r=n$ la strate est nulle et $\theta$ est trivial. 
Sinon, soit 
  $[\L, n, r+1, \gamma]$ une strate semi-simple de $\mathfrak g_2(F)$ 
  \'equivalente \`a 
   $[\L, n, r+1, \beta]$ et de m\^eme type (th\'eor\`eme \ref{remontee}). 
 D'apr\`es le lemme     \ref{betagamma} on peut \'ecrire   
 $$
H^1(\beta , \L ) = \left( P^1(\L ) \cap G_\beta \right) H^{[\frac{r+1}{2}] +1}(\gamma, \L)  , 
$$ 
produit de  deux sous-groupes stables par trialit\'e.
 D'apr\`es le lemme  \ref{defss} (ii), la restriction de $\theta$ \`a 
 $H^{[\frac{r+1}{2}] +1}(\gamma, \L) =H^{[\frac{r+1}{2}] +1}(\beta, \L) $ est de la forme 
 $\theta_0 \psi_{\beta-\gamma}$ pour un $\theta_0 \in \CC(\L, m',  \gamma)$~; elle est  fixe par trialit\'e vu  (\ref{bottom}) et l'hypoth\`ese inductive (noter que ${\beta-\gamma}$ est de trace nulle sur $V^0$ donc $\psi_{\beta-\gamma}$ est trivial sur le centre de $\U(W, V^0/F)$).  
 Ensuite vu le lemme  \ref{central}  : 
$$   P^1(\L ) \cap G_\beta   =  P^1(\L^0) \times \left( P^1(\L_W)\cap U(W, V^0/F)
\cap G_\beta \right)  .$$
   Soit $t_1 $ un \'el\'ement de ce sous-groupe.  
 D'apr\`es la formule de trialit\'e  (\ref{trialitedim2}) 
  les actions de $t_1$ et  $\nu(t_1)$ sur   $W$  diff\`erent d'une multiplication \`a gauche  par un \'el\'ement $z$ de $V^0$ de norme $1$, c'est-\`a-dire d'un \'el\'ement du centre de $U(W, V^0/F)$.  
 L'hypoth\`ese assure donc que  $\theta \circ \nu $ et $ \theta$ co\"\i ncident sur ce sous-groupe, c.q.f.d.

  \bigskip
Le cas (iii) se traite de la m\^eme mani\`ere en utilisant le lemme \ref{centralGL}  au lieu du lemme 
\ref{central} 
:    
d'apr\`es la formule de trialit\'e (\ref{GLW}), si $t_1 \in O^\prime(V)$ conserve $V^0$, $W^+$ et $W^-$, les actions de $t_1$ et  $\nu(t_1)$ sur   $W^+$  diff\`erent d'un scalaire, d'o\`u l'\'egalit\'e voulue 
  $\theta \circ \nu = \theta$. 
  \end{proof}

\setcounter{section}{2}
  
\Section{Caract\`eres semi-simples de $\G2(F)$ et entrelacement} \label{par3} 

Soit $[\L, n, 0, \beta]$ une strate semi-simple  de $\mathfrak g_2(F)$. Nous lui associons une famille de ca\-rac\-t\`eres semi-simples $\bar \theta$ du groupe  $\oH^1(\beta, \L) = H^1(\beta, \L) \cap  \G2(F)$  dont nous d\'eterminons l'entrelacement et le comportement par transfert (\ref{par31}).  
 A chaque  tel caract\`ere    est attach\'ee une unique   repr\'esentation  irr\'educ\-tible  $\bar \eta$ de $\oJ^1(\beta, \L)$  telle  que $\bar \eta$ contienne $\bar \theta$ (\ref{par32}). Comme dans un groupe lin\'eaire ou classique, l'espace d'entrelacements de $\bar \eta$ est de dimension $0$ ou $1$ (\ref{par33}). 

\subsection{Entrelacement des caract\`eres semi-simples de $\G2(F)$}\label{par31}

Soit $[\L, n, 0, \beta]$ une strate semi-simple  de $\mathfrak g_2(F)$. 
  Les sous-groupes $H^i(\beta, \L)$ et  $J^i(\beta, \L)$, $i \ge 1$,  sont stables par trialit\'e 
  (lemme  \ref{ex2.1}) et  l'on pose 
  $\oH^1(\beta, \L)= \widetilde H^1(\beta, \L)\cap \bar G =
H^1(\beta, \L)^\Gamma $ et  $\oJ^1(\beta, \L)=\widetilde J^1(\beta, \L)\cap \bar G =
J^1(\beta, \L)^\Gamma $. 
Ce sont des sous-groupes ouverts compacts de $\G2(F)$. 
Dans la suite, nous supprimerons   
$(\beta, \L)$ des notations pour certaines d\'emonstrations.

Comme nous l'avons vu dans le th\'eor\`eme \ref{carfixes}, 
le groupe $\Gamma$ d'automorphismes  de $P^1(\L)$ op\`ere sur 
$H^1(\beta, \L)$ et sur l'ensemble   de ses caract\`eres, mais ne conserve  pas l'ensemble des 
caract\`eres semi-simples autoduaux de $H^1(\beta, \L)$ 
: seul l'ensemble des caract\`eres semi-simples {\it sp\'eciaux} de $H^1(\beta, \L)$ est conserv\'e, et du reste fix\'e,   par trialit\'e. 
Cela ne nous emp\^eche pas de prendre 
  mod\`ele sur 
\cite{S3} et d'utiliser la correspondance de Glauberman  entre  l'ensemble des 
classes d'\'equivalence $\Gamma$-invariantes de repr\'esentations irr\'eductibles  de $H^1(\beta, \L)$    et  l'ensemble des classes d'\'equivalence de repr\'esentations irr\'educ\-tibles
de $\bar H^1(\beta, \L)$ \cite[\S 2]{S3}. 
On d\'efinit   l'ensemble  $\bar{\mathscr C} (\Lambda, \beta)$ des {\it caract\`eres semi-simples} de $\bar H^1(\beta, \L)$   comme l'image 
par cette corres\-pondance    de l'ensemble des caract\`eres semi-simples de $H^1(\beta, \Lambda)$ qui sont $\Gamma$-invariants, c'est-\`a-dire des  caract\`eres semi-simples sp\'eciaux. 
L'image par la correspondance de Glauberman d'une repr\'esentation de dimension $1$ est sa restriction, 
donc la restriction des caract\`eres est   une {\it bijection} : 
$$ 
\mathscr C^0(\Lambda, \beta) \quad 
\stackrel{\simeq}\longrightarrow  \quad  \bar{\mathscr C} (\Lambda, \beta)
$$

 \begin{Notation}\label{theta} Soit $\bar{\mathscr C} (\Lambda, \beta)$ l'ensemble des caract\`eres semi-simples de $\bar H^1(\beta, \L)$. Si 
   $\bar \theta $ appartient \`a  $\bar{\mathscr C} (\Lambda, \beta)$ on note $\theta$ le 
  ca\-rac\-t\`ere 
semi-simple sp\'ecial  de $  H^1(\beta, \L)$ de restriction $\bar \theta$ \`a $\bar H^1(\beta, \L)$, et $\widetilde \theta$ le ca\-rac\-t\`ere 
semi-simple $\tau$-invariant de $  \widetilde H^1(\beta, \L)$ de restriction $\theta$ \`a 
$  H^1(\beta, \L)$.
\end{Notation}

Passons aux propri\'et\'es essentielles d'entrelacement et de transfert, en commen\c cant par 
le fait dit 
   ``d'intersection simple''  : 

\begin{Lemma}\label{simpleinter}{\bf \cite[Lemme 2.6)]{S5}.}
Notons $P^1_\beta(\L) = P^1(\L) \cap G_\beta$. 
 Soit $x  \in   G_\beta$. Alors 
 $$ P^1(\L) x  P^1(\L) \; \cap \; G_\beta \; 
 = P^1_\beta(\L) x  P^1_\beta(\L).$$
 \end{Lemma}

 \begin{proof} 
 Celle de {\it loc. cit.} reste valide avec  notre d\'efinition un peu plus large de strate semi-simple gauche. 
 \end{proof}

\begin{Proposition}\label{entrelacement}
Soit $\bar \theta$ un   caract\`ere semi-simple  de $\bar H^1(\beta, \L)$.
L'entrelacement de $\bar \theta$ dans $\G2(F)$ est \'egal \`a 
$\bar J^1(\beta, \L) \; \bar G_\beta \; \bar J^1(\beta, \L)$. 
\end{Proposition}

\begin{proof} Reprenons l'argument de \cite[Proposition 3.27]{S4}, \`a base de correspondance de Glauberman. La premi\`ere \'etape consiste \`a passer de $\widetilde G$ \`a $G$ : l'entrelacement  de $\theta$ dans $G$ est l'intersection avec $G$ de celui de $\widetilde \theta$, soit $\widetilde J^1\widetilde G_\beta \widetilde J^1 \cap G = J^1 G_\beta  J^1$
(\cite[Corollaire 2.5]{S3}, \cite [Th\'eor\`eme 3.22]{S4}, \cite[\S 2]{S2}) et son entrelacement dans $G^\prime$ est $J^1 G_\beta^\prime  J^1$.

La  trialit\'e n'est pas d\'efinie sur le groupe $G^\prime$ :  nous devons relever  $\theta$ en un caract\`ere $s(\theta)$ du sous-groupe $s(H^1)$ de $\Spin_F (V)$, o\`u $s$ est la section de $P^1(\L)$ (\S \ref{par225}). L'unicit\'e  d'une section homomorphe d'un pro-$p$-sous-groupe (\S \ref{par21}) entra\^\i ne : 
$$ 
\forall \; g \in \Spin_F (V) \quad 
s\left(H^1\cap \pi(g)^{-1}H^1 \pi(g) \right) = s(H^1) \cap 
g^{-1} s(H^1) g .
$$
Donc l'entrelacement de $s(\theta)$ est image inverse de celui de $\theta$, c'est-\`a-dire $s(J^1) \pi^{-1}( G_\beta^\prime  ) s(J^1)$.

Le caract\`ere $\bar \theta$ s'obtient aussi bien par restriction de $s(\theta)$ \`a $\bar H^1$, c'est-\`a-dire par une corres\-pondance de Glauberman appliqu\'ee \`a l'action de $\Gamma$ sur $s(H^1)$ : alors \cite[Corollaire 2.5]{S3} nous dit que l'entrelacement de $\bar \theta$ est  $\left(s(J^1) \pi^{-1}( G_\beta^\prime  ) s(J^1) \right)\cap \bar G = \left(s(J^1) \pi^{-1}( G_\beta^\prime  ) s(J^1) \right)^\Gamma$. 

Il n'y a plus qu'\`a \'etablir l'\'egalit\'e $ \left(s(J^1) \pi^{-1}( G_\beta^\prime  ) s(J^1) \right)^\Gamma= \bar J^1 \; \bar G_\beta \; \bar J^1$. Il suffit pour cela d'appliquer  le Th\'eor\`eme 2.3 de  \cite{S2} \`a $\Gamma$, groupe r\'esoluble d'ordre $6$ premier \`a $p$ (voir la remarque finale de {\it loc. cit.} \S 2) agissant sur  $  \Spin_F (V)$, au pro-$p$-sous-groupe $U = s(J^1)$ et \`a $H = \pi^{-1}( G_\beta^\prime  )$, stable par trialit\'e puisque $\beta$ l'est. On est ainsi ramen\'e \`a la v\'erification de l'hypoth\`ese de ce th\'eor\`eme~:  
 $$
\forall \; g \in \pi^{-1}( G_\beta^\prime  ) \quad 
\left(s(J^1) \; g \; s(J^1) \right) \cap  \pi^{-1}( G_\beta^\prime  ) = s(J^1 \cap G_\beta) \;  g \; s(J^1\cap G_\beta).
$$ 
Remarquons que $J^1 \cap G_\beta = P^1_\beta(\L)$. On a 
(cf. \cite[Th\'eor\`eme 4.7]{S2}) : 
 $$
 \aligned
 s(J^1 \cap G_\beta) \;  g \; s(J^1\cap G_\beta)
 &\subseteq 
 \left(s(J^1) \; g \; s(J^1) \right) \cap  \pi^{-1}( G_\beta^\prime  )  \\
 &\subseteq 
 \left(s(P^1(\L)) \; g \; s(P^1(\L)) \right) \cap  \pi^{-1}( G_\beta^\prime  )
 \endaligned
 $$
 La projection de ce dernier sous-ensemble sur $G^\prime$ est contenue dans 
 $\left(P^1(\L) \; \pi(g) \; P^1(\L) \right) \cap   G_\beta^\prime$, 
 qui est \'egal \`a $P^1_\beta(\L) \; \pi(g) \; P^1_\beta(\L) $ par le lemme \ref{simpleinter}. L'\'egalit\'e voulue en r\'esulte. 
 \end{proof}

La propri\'et\'e dite ``de transfert'' s'\'etend elle aussi aux caract\`eres semi-simples de $G_{2}(F)$. 

\begin{Proposition}\label{transfertG2}
Soient $[\Lambda,n,0,\beta]$ et $[\Lambda',n',0,\beta]$ deux strates semi-simples de $\lg2(F)$. Il existe une bijection canonique 
$$\tau_{\Lambda,\Lambda',\beta} : \bar{\mathscr C} (\Lambda, \beta)  \longrightarrow \bar{\mathscr C} (\Lambda', \beta)$$
qui transforme un caract\`ere semi-simple $\bar \theta$ de $ \bar{\mathscr C} (\Lambda, \beta)$ en l'unique caract\`ere semi-simple $\bar \theta'$ de $ \bar{\mathscr C} (\Lambda', \beta)$ qui est entrelac\'e \`a $\bar \theta$ par un \'el\'ement de $\bar G_{\beta}$. Dans ce cas, tout \'el\'ement de $\bar G_{\beta}$ entrelace les deux caract\`eres.
\end{Proposition}

\begin{proof} Les deux strates $[\Lambda,n,0,\beta]$ et $[\Lambda',n',0,\beta]$ sont des strates semi-simples dans $\mathfrak{gl}_F(V)$. Par \cite[Proposition 3.2]{S5}, il existe une bijection canonique $\tau_{\Lambda,\Lambda',\beta}$ entre les caract\`eres semi-simples de $ \widetilde H^1(\beta, \L)$ et ceux de $ \widetilde H^1(\beta, \L')$ d\'efinie par : si $\theta$ est un caract\`ere semi-simple de $ \widetilde H^1(\beta, \L)$, $\theta'=\tau_{\Lambda,\Lambda',\beta}(\theta)$ est le caract\`ere semi-simple de $ \widetilde H^1(\beta, \L')$ entrelac\'e \`a $\theta$ par un (tout) \'el\'ement de $\widetilde G_{\beta}$. Cette bijection commute \`a l'involution $\tau$ (les \'el\'ements de $G_{\beta}$ qui entrelacent deux caract\`eres entrelacent aussi leurs images sous l'action de $\tau$) donc se restreint en une bijection de $\tilde \CC(\L,  \beta)^\tau$ sur $\tilde \CC(\L',  \beta)^\tau$. En composant avec la correspondance de Glauberman, on obtient une bijection $\tau_{\Lambda,\Lambda',\beta}$ des caract\`eres semi-simples autoduaux de $H^1(\beta,\Lambda)$ sur ceux de $H^1(\beta,\Lambda')$ encore caract\'eris\'ee par  (\cite[cor. 2.5]{S3}) :
\begin{itemize}
\item[] $\tau_{\Lambda,\Lambda',\beta}(\theta)=\theta' $ si et seulement si $\theta$ et $\theta'$ sont entrelac\'es par un (tout) \'el\'ement de $G_{\beta}$. 
\end{itemize}
Par cons\'equent, compte tenu de la caract\'erisation des caract\`eres sp\'eciaux de $H^1(\beta,\Lambda)$ et $H^1(\beta,\Lambda')$ (th\'eor\`eme \ref{carfixes}), $\tau_{\Lambda,\Lambda',\beta}$ se restreint en une bijection de $\mathscr C^0(\Lambda, \beta)$ sur $\mathscr C^0(\Lambda', \beta)$, 
d'o\`u l'\'enonc\'e par restriction \`a $\oH^1(\beta,\Lambda)$ et $\oH^1(\beta,\Lambda')$.
\end{proof}

\subsection{Extensions de Heisenberg}\label{par32}
La suite est classique : on passe du caract\`ere semi-simple $\bar \theta$ de  $\bar H^1(\beta, \L)$ \`a une repr\'esentation $\bar \eta$ de  $\bar J^1(\beta, \L)$, soit par une construction directe \`a la Heisenberg, soit par correspondance de Glauberman \`a partir de l'unique repr\'esentation irr\'eductible $\eta$ de  $J^1(\beta, \L)$ contenant $\theta$.  

\begin{Proposition}\label{eta}
Soit $\bar \theta$ un   caract\`ere semi-simple  de $\bar H^1(\beta, \L)$.
  Il existe une unique repr\'e\-sen\-tation irr\'eductible $\bar \eta$ de  $\bar J^1(\beta, \L)$ contenant $\bar \theta$. Sa restriction \`a   $\bar H^1(\beta, \L)$ est multiple de $\bar \theta$ et son  
	 entrelacement est 
	$\bar J^1(\beta, \L) \; \bar G_\beta \; \bar J^1(\beta, \L)$. 
	 Pour $g \in \bar G_\beta $, la dimension de 
	l'espace d'entrelacements  
	$\Hom_{\bar J^1 \cap g^{-1} \bar  J^1 g} \; (\bar \eta, \bar \eta^g)$ est 
	\'egale \`a $1$.  
 \end{Proposition}

\begin{proof} On rappelle la notation \ref{theta}. Soit $\widetilde \eta$ l'unique repr\'esentation irr\'eductible  de  $\widetilde J^1$ contenant $\widetilde \theta$ ; son entrelacement est $\widetilde J^1 \widetilde G_\beta \widetilde J^1$ \cite[Corollaire 3.25]{S4}. Par unicit\'e, $\widetilde \eta$ est fix\'ee par $\tau$ \`a \'equivalence  pr\`es. Une premi\`ere correspondance de Glauberman lui 
associe  donc une unique (classe de) repr\'esentation irr\'educ\-tible $\eta$ de $ J^1$ contenant $\theta$ 
 \cite[Th\'eor\`eme 2.1 et Lemme 2.3]{S3}.  Comme $\theta$ est invariant par trialit\'e, la classe de $\eta$ l'est aussi et une deuxi\`eme correspondance de Glauberman lui associe  $\bar \eta$, 
unique (classe de) repr\'esentation irr\'educ\-tible  de $ \bar J^1(\beta, \L)$ contenant $\bar \theta$
(m\^emes motifs). Comme $\bar J^1(\beta, \L)$ entrelace 
$\bar \theta$, la restriction de $\bar \eta $ \`a  
$\bar J^1(\beta, \L)$ est en fait multiple de $\bar \theta$. 

Par ailleurs, la restriction \`a $\bar J^1$ de la forme altern\'ee non d\'eg\'en\'er\'ee $(x,y) \mapsto \theta([x,y])$ 
sur $J^1/ H^1$ d\'efinit une forme altern\'ee non d\'eg\'en\'er\'ee $(x,y) \mapsto \bar\theta([x,y])$ sur 
$\bar J^1/ \bar H^1$ (cf. lemme \ref{GammaPerp}). Il existe donc une unique (classe de) repr\'esentation 
irr\'eductible de $\bar J^1$ contenant $\bar \theta$ : c'est 
$\bar \eta$.

L'entrelacement de $\bar \eta$ est \'evidemment contenu dans celui de $\bar \theta$. Pour obtenir l'inclusion oppos\'ee on est une fois de plus g\^en\'e par le fait que la trialit\'e n'agit pas sur $G$ lui-m\^eme. On rel\`eve donc $ \eta$ en une repr\'esentation $s(\eta)$ de $s(J^1)$ : $\bar \eta$ s'obtient aussi bien par correspondance de Glauberman \`a partir de $s(\eta)$. Or l'entrelacement de $s(\eta)$ est l'image inverse de celui de  $\eta$ soit $ s(J^1) \pi^{-1}( G_\beta^\prime  ) s(J^1) $, gr\^ace \`a la   preuve de \ref{entrelacement} et \`a \cite[Lemme 2.4]{S3}. Les deux m\^emes ingr\'edients fournissent   l'entrelacement voulu pour $\bar \eta$.   

En fait,  seule l'inclusion de l'entrelacement dans  
	$\bar J^1(\beta, \L) \; \bar G_\beta \; \bar J^1(\beta, \L)$ est utilis\'ee dans la 
	d\'emonstration du dernier point, que l'on reporte au paragraphe suivant 
	car elle est assez longue : nous ne pouvons pas  imiter  les d\'emonstrations existantes, qui adaptent  celle de   \cite[Proposition 5.1.8]{BK1}. En effet elles se placent toutes dans l'alg\`ebre de Lie, or nous ne disposons pas d'une application commode de  $\mathfrak g_2(F)$ 
dans $\G2(F)$  puisque la transformation de Cayley ne joue pas ce r\^ole (Remarque \ref{Moy}).
\end{proof}

\subsection{Dimension des espaces d'entrelacement }\label{par33}

Reste \`a \'etablir la propri\'et\'e cruciale de la repr\'esentation $\bar\eta$ : pour  $g \in \bar G_\beta $,  
$$\dim \Hom_{\bar J^1 \cap g^{-1} \bar  J^1 g} \; (\bar \eta, \bar \eta^g) = 1 .$$

Commen\c cons par introduire des notations, pour ce paragraphe uniquement, qui nous permet\-tront de travailler indiff\'eremment dans  un groupe lin\'eaire sur $F$, 
  un groupe classique sur $F$   ou $G_2(F)$.  On note    $\bs G$  un tel groupe, puis $\bs H^1$, $\bs J^1$,    $\bs \theta$   et $\bs \eta$ les groupes, caract\`ere semi-simple et extension de Heisenberg d\'efinis dans $\bs G$ relativement \`a une strate semi-simple $[\L, n, 0, \beta]$  de  $\Lie \bs G $ 
fix\'ee. Soit $g \in \bs G$ commutant \`a $\beta$. On d\'efinit : 
\begin{itemize}
	\item $\bs \theta^g$, $\bs \eta^g$ les repr\'esentations de $g^{-1}\bs H^1 g$ et $g^{-1}\bs J^1 g$ 
 conjugu\'ees de $\bs \theta$, $\bs \eta$ ; 
 \item $\bs K_g = \ker (\bs \theta) \cap  g^{-1}\ker (\bs \theta) g$, 
 $\bs H_g = \bs H^1 \cap  g^{-1}\bs H^1 g$
 et $\bs J_g = \bs J^1 \cap  g^{-1}\bs J^1 g$ ; 
 \item $\bs L_g$ le sous-groupe de $\bs J_g$ engendr\'e par $ \bs H^1 \cap  g^{-1}\bs J^1 g$  et $ \bs J^1 \cap  g^{-1}\bs H^1 g$, 
  c'est-\`a-dire le produit 
$( \bs H^1 \cap  g^{-1}\bs J^1 g)( \bs J^1 \cap  g^{-1}\bs H^1 g)$ ; 
\item $\bs M_g$ le noyau de la forme altern\'ee 
$ (x, y )    \mapsto  \bs \theta([x, y]) $ sur $\bs J_g / \bs H_g  $    vu comme espace vectoriel sur $\mathbb F_p$. On a 
 $\bs K_g \lhd \bs H_g \lhd \bs L_g \lhd \bs M_g \lhd \bs J_g . 
 $ 
\end{itemize}

\begin{Proposition}\label{essentiel} 
La dimension de 
$\Hom_{\bs J_g}(\bs \eta, \bs \eta^g)$ est \'egale \`a $ 1 $ si et seulement si  
\begin{enumerate}
	\item $\bs M_g=\bs L_g$, 
	\item $(\bs H^1 \bs M_g)^\perp =  \bs H^1 \bs J_g $, 
	\item $ ( (g^{-1} \bs H^1g) \bs M_g)^\perp =  (g^{-1} \bs H^1g ) \bs J_g$. 
\end{enumerate} 
\end{Proposition} 
\begin{proof}  
La restriction de $\bs \eta$ \`a $ \bs H^1 \cap  g^{-1}\bs J^1 g$ est scalaire, multiple de $\bs \theta$, 
celle de  $\bs \eta^g$ \`a $ \bs J^1 \cap  g^{-1} \bs H^1 g$ est scalaire, multiple de $\bs \theta^g$, or 
  $\bs \theta$ et $\bs \theta^g$ co\"\i ncident sur l'intersection $\bs H_g$. La restriction 
   \`a $\bs L_g$  de toute composante irr\'eductible commune   \`a 
  $\bs \eta{|\bs J_g}$ et     $\bs \eta^g{|\bs J_g}$ 
est donc scalaire et op\`ere par le  caract\`ere $\bs \theta_g$ de $\bs L_g$ qui est l'unique caract\`ere prolongeant les deux pr\'ec\'edents.

Le groupe fini $\bs M_g/\bs K_g$ est commutatif, c'est le centre de $\bs J_g/\bs K_g$. Le  caract\`ere $\bs \theta_g$ de $\bs L_g / \bs K_g$ a donc $[\bs M_g : \bs L_g]$ prolongements \`a $\bs M_g$. Soit $X$ l'ensemble de ces prolongements et 
$\chi \in X$,  la repr\'esentation 
$\cInd_{\bs M_g}^{\bs J_g} \chi $ est  multiple d'une unique repr\'esentation irr\'eductible
  $h(\chi)$, de degr\'e $[\bs J_g : \bs M_g]^{1/2}$ et multiplicit\'e 
$[\bs J_g : \bs M_g]^{1/2}$. Les $h(\chi)$, $\chi \in X$,  sont deux \`a deux non \'equivalentes puisque  leurs caract\`eres centraux sont distincts. En outre  
$$
\cInd_{\bs L_g}^{\bs J_g} \bs \theta_g =\sum_{\chi \in X} \cInd_{\bs M_g}^{\bs J_g} \chi =  \sum_{\chi \in X} [\bs J_g : \bs M_g]^{1/2} \   h(\chi) , 
$$
donc 
toute 
composante irr\'eductible commune  \`a $\bs \eta$ et $\bs \eta^g$  est l'une des repr\'esentations $h(\chi)$.

Le sous-groupe $ \bs H^1 \bs M_g$  de $\bs J^1$ est totalement isotrope pour la forme $\bs \theta([x, y])$. Pour $\chi \in X$,    les caract\`eres $\bs \theta$ de $\bs H^1$ et $\chi$ de $\bs M_g$ 
co\"\i ncident sur $\bs H^1 \cap \bs M_g = \bs H^1 \cap g^{-1}\bs J^1 g$   donc  d\'efinissent un unique caract\`ere $\bs \theta\chi$ de $ \bs H^1 \bs M_g$ ($\bs M_g$ entrelace $\bs \theta$). Soit alors $Y$ un sous-espace totalement isotrope maximal de $\bs J^1$ contenant $\bs H^1 \bs M_g$ et $\xi$ un caract\`ere de $Y$ prolongeant  $\bs \theta\chi$. 
La repr\'esentation $ \bs \eta$ est \'equivalente \`a $\cInd_Y^{\bs J^1} \xi$ :  sa restriction \`a 
$\bs M_g$ contient $\chi$ donc sa restriction \`a $\bs J_g$ contient  $h(\chi)$. 
On  calcule  la multiplicit\'e de  $h(\chi)$ dans $\bs \eta_{|\bs J_g}$ de la fa\c con suivante : 
$$
\aligned 
&[\bs J_g : \bs M_g]^{1/2}   <\bs \eta, h(\chi)>_{\bs J_g} 
  = <\bs \eta,  \cInd_{\bs M_g}^{\bs J_g} \chi)>_{\bs J_g} 
= <\bs \eta, \chi >_{\bs M_g}= < \cInd_Y^{\bs J^1} \xi, \bs \theta\chi >_{\bs H^1\bs M_g} 
\\
&\qquad \qquad = < \cInd_Y^{\bs J^1} \xi, \cInd_{\bs H^1\bs M_g}^{\bs J^1} \bs \theta\chi >_{\bs J^1} 
= < \cInd_Y^{\bs J^1} \xi, \sum_{\xi^\prime}\cInd_{Y}^{\bs J^1} \xi^\prime >_{\bs J^1}  
\endaligned 
$$
o\`u la somme porte sur les prolongements $\xi^\prime$  de $\bs \theta\chi$ \`a $Y$. 
Chacune de ces induites est isomorphe \`a  $\bs \eta$ et il reste 
$$
<\bs \eta, h(\chi)>_{\bs J_g}  =  [\bs J_g : \bs M_g]^{- 1/2} [Y: \bs H^1\bs M_g] = d_1 
$$
La multiplicit\'e  $d_2$ de $h(\chi)$ dans $\bs \eta_{|\bs J_g}$ est donn\'ee par une formule analogue, en travaillant avec le sous-groupe 
$ g^{-1} \bs H^1 g \bs M_g    $ de $g^{-1} \bs J^1 g$. 

L'ensemble des composantes irr\'eductibles communes \`a $\bs \eta$ et $\bs \eta^g$ est 
donc exactement l'ensemble des  $h(\chi)$, $\chi \in X$ et 
la dimension de 
$\Hom_{\bs J_g}(\bs \eta, \bs \eta^g)$ est $d_1d_2  [\bs M_g : \bs L_g] $, une puissance de $p$. 
Cette  dimension est \'egale \`a $1$ si et seulement si 
$\bs M_g=\bs L_g$ et $d_1=d_2=1$. 

Reprenons l'expression de $d_1$ (le cas de $d_2$ est \'evidemment semblable). On a  
$[\bs J_g : \bs  M_g]= [\bs H^1 \bs J_g : \bs H^1 \bs M_g]$ et $[Y: \bs H^1\bs M_g][  \bs H^1\bs M_g:\bs H^1] = [\bs J^1: \bs H^1 ]^{ 1/2}$ d'o\`u 
$$d_1^2 = \frac{ [\bs J^1: \bs H^1 ]}{ [\bs H^1 \bs J_g : \bs H^1\bs M_g][  \bs H^1\bs M_g:\bs H^1]^2} 
=\frac{ [\bs J^1: \bs H^1 ]}{[\bs J^1:  (\bs H^1\bs M_g)^\perp] [\bs H^1 \bs J_g : \bs H^1\bs M_g][ \bs  H^1\bs M_g:\bs H^1] } .$$ 
Comme $\bs H^1 \subset \bs H^1 \bs M_g \subset \bs H^1 \bs J_g 
\subset (\bs H^1\bs M_g)^\perp  \subset \bs J^1$ il reste 
$d_1^2 = [  (\bs H^1\bs M_g)^\perp: \bs H^1 \bs J_g  ]$.  \end{proof}

Le lemme suivant est une simple formalisation de la remarque  \cite[p. 135]{S3}. 

\begin{Lemma}\label{GammaPerp} 
Soit $\Gamma$ un groupe d'ordre $2$ ou $3$ agissant sur un espace symplectique 
$V$ sur $\mathbb F_p$ dont on note $<,>$ la forme,  fix\'ee par $\Gamma$. On suppose $p$ diff\'erent de $2$ et $3$. 
Pour tout sous-espace $X$ de $V$ stable sous l'action de $\Gamma$, 
l'orthogonal de $X^\Gamma$ dans $V^\Gamma$ est \'egal \`a 
$  (X^\perp)^\Gamma$. 
En particulier, le noyau de la restriction de la forme \`a $V^\Gamma$ est form\'e des  points fixes du noyau de la forme sous $\Gamma$. 
\end{Lemma}

\begin{proof}  
Le groupe $\Gamma$ est engendr\'e par un \'el\'ement $s$ d'ordre $2$ ou $3$. Le polyn\^ome caract\'eristique de $s$ agissant sur $V$ est  \'egal  \`a $X^2-1$ ou $X^3-1$ de sorte que 
$V$ se d\'ecompose en somme directe des deux sous-espaces $V_1$, sous-espace propre de $s$ pour la valeur propre $1$, et $V_s$, noyau de $s+1$ ou de $s^2+s+1$ selon le cas. 
Ces sous-espaces sont orthogonaux pour la forme symplectique : si $x \in V_1$ et 
$y \in V_s$ alors 
$$
\aligned 
\text{si } s^2=1  \qquad  2<x,y> &= <x, y > + <x, sy> =0 ; 
\\
 \text{si } s^3=1  \qquad   3<x,y> &= <x, y > + <x, sy> + <x, s^2,y> =0 . 
\endaligned 
$$
Le sous-espace $X$ se d\'ecompose de m\^eme en $X = X_1 \overset \perp \oplus X_s$  avec $X_1 = X^\Gamma$. Son orthogonal dans $V$ est la somme directe de l'orthogonal de 
$X_1$ dans $V_1$ et de l'orthogonal de $X_s$ dans $V_s$ d'o\`u l'assertion. 
\end{proof}

\begin{Corollary} 
Soit $\Gamma$ un groupe d'ordre $2$ ou $3$ agissant sur le groupe ambiant $\bs G$ en stabilisant les groupes $\bs H^1$ et $\bs J^1$ et le caract\`ere $\bs \theta$. 
Soit $g \in G^\Gamma$ centralisant  $\beta$. 
Soit $\bs H^{1 \Gamma}$ et $\bs J^{1 \Gamma}$ les groupes de points fixes de $\Gamma$ dans $\bs H^1$ et $\bs J^1$ respectivement, soit  $\bs \theta^\Gamma$ la restriction de $\bs \theta$ \`a $\bs H^{1 \Gamma}$ 
et soit $\bs \eta^\Gamma$ la repr\'esentation irr\'eductible de 
$\bs J^{1 \Gamma}$ contenant $\bs \theta^\Gamma$.  
\begin{enumerate}
	\item Les sous-groupes de points fixes de $\Gamma$ : $\bs J_g^\Gamma$, $\bs M_g^\Gamma$, $\bs L_g^\Gamma$, $\bs H_g^\Gamma$ et $\bs K_g^\Gamma$,  sont exactement les groupes obtenus par les d\'efinitions du d\'ebut de ce paragraphe   \`a partir de $\bs H^{1 \Gamma}$ et $\bs J^{1 \Gamma}$. 
	\item Si $\dim \Hom_{\bs J_g}(\bs \eta, \bs \eta^g) = 1$, alors 
	$\dim \Hom_{\bs J_g^\Gamma}(\bs \eta^\Gamma, \bs \eta^{\Gamma g}) = 1$. 
\end{enumerate}
\end{Corollary}
\begin{proof} 
La premi\`ere assertion, pour  $\bs J_g^\Gamma$,    $\bs H_g^\Gamma$ et $\bs K_g^\Gamma$, d\'ecoule de $g \in G^\Gamma$. Pour $\bs L_g^\Gamma$ 
elle 
r\'esulte de l'application 
du th\'eor\`eme 2.3 de \cite{S2}.
Enfin, pour  $\bs M_g^\Gamma$, elle provient de l'application du   lemme \ref{GammaPerp} \`a l'espace symplectique $\bs J_g/\bs H_g$.

Utilisons maintenant la proposition \ref{essentiel}. 
L'hypoth\`ese de la deuxi\`eme assertion implique 
    $  \bs M_g =  \bs L_g$, d'o\`u   
$\bs M_g^\Gamma = \bs L_g^\Gamma$   ; 
    $(\bs H^1\bs M_g)^\perp =  \bs H^1 \bs J_g $ donc 
par le lemme \ref{GammaPerp}   
$(  \bs H^{1 \Gamma}  \bs M_g^\Gamma )^\perp = 
  \bs H^{1 \Gamma}  \bs J_g^\Gamma   $~; 
  $ ( (g^{-1} \bs H^1g) \bs M_g)^\perp =  (g^{-1} \bs H^1g ) \bs J_g$ donc  
  $( (g^{-1}  \bs H^{1 \Gamma} g) \bs M_g^\Gamma)^\perp =  (g^{-1}  \bs H^{1 \Gamma}g )   \bs J_g^\Gamma$ ;  
d'o\`u le r\'esultat. 
\end{proof}

Il n'y a plus qu'\`a conclure. La propri\'et\'e d'entrelacement voulue,  
 $\dim \Hom_{\bs J_g}(\bs \eta, \bs \eta^g) = 1$, est valide   pour les caract\`eres semi-simples de $GL(n,F)$ lui-m\^eme : la d\'emonstration  de  
\cite[Proposition 3.31]{S4}   s'applique mot pour mot dans ce cas.  
Le corollaire ci-dessus  nous permet de propager cette propri\'et\'e   de 
$GL(n,F)$ \`a un groupe classique sur $F$ d\'efini comme groupe de points fixes d'une involution (si $p \ne 2$). Le cas de $SO(8,F) $ s'obtient  \`a partir de $O(8,F)$ puisque les sous-groupes   consid\'er\'es sont des pro-$p$-groupes,   contenus dans $SO(8,F) $. Enfin on passe de $SO(8,F)$ \`a $G_2(F)$ (si $p \ne 2,3$) via la trialit\'e et le corollaire \`a nouveau.

\setcounter{section}{3}
\Section{Repr\'esentations supercuspidales de $\G2 (F)$}\label{beta}

Soient $[\L, n, 0, \beta]$ une strate semi-simple non nulle de $\mathfrak g_2(F)$ au sens du paragraphe \ref{strates.2}, $\bar \theta$ un caract\`ere semi-simple de $\bar H^1(\beta,\Lambda)$ et $\bar \eta$ son extension de Heisenberg \`a $\bar J^1(\beta,\Lambda)$ (prop. \ref{eta}). On d\'efinit le groupe $\bar J(\beta, \L)$ sur lequel vivent les types semi-simples construits \`a partir de ces donn\'ees puis on s\'electionne des prolongements de $\bar \eta$ \`a ce groupe (\S \ref{types.beta}) \`a partir desquels on construit  des repr\'esentations supercuspidales  de $\G2(F)$ dans les cas o\`u les sous-groupes  $\bar P(\L)$ et $\bar G_\beta$ satisfont certaines conditions (\S \ref{types.cuspidales}). La m\'ethode est une variation sur celle de S. Stevens \cite{S5}.

\subsection{$\beta$-extensions}\label{types.beta}

Pour tout sous-groupe $H$ de $\SO_F(V)$ on note $H'$ son intersection avec le groupe orthogonal
 r\'eduit $G'=\O'_F(V)$. 
Soit $P^\prime_\beta(\L)= P^\prime(\L)\cap G_\beta$ et $P^i_\beta(\L)= P^i(\L)\cap G_\beta$, $i \ge 1$. 
Le sous-groupe   $\pi^{-1}(P^\prime( \L))$  de $\Spin_F (V) $ est stable par trialit\'e (\S \ref{par225}) ainsi que son intersection $\pi^{-1}(P^\prime_\beta( \L))$ avec le centralisateur de $\b$. 
 
On pose alors  $\bar P( \L)= \left(\pi^{-1}(P^\prime( \L))\right)^\Gamma $ ; c'est un sous-groupe ouvert compact maximal   de 
$\G2(F)$, le fixateur du point de l'immeuble attach\'e \`a $\L$. On pose enfin $\bar P_\beta( \L)= \bar P( \L) \cap G_\beta $.

 \begin{Lemma}\label{trialite2}
\begin{enumerate}
\item  
On a $J^\prime(\beta, \L)= P^\prime_\beta(\L) J^1(\beta, \L)$. Son image inverse  $\pi^{-1}(J^\prime(\beta, \L))$  dans $\Spin_F (V) $ est stable par trialit\'e. 

\item
Soit  $\bar J(\beta, \L)= \left(\pi^{-1}(J^\prime(\beta, \L))\right)^\Gamma $. On a   
$\bar J(\beta, \L) = \bar P_\beta(\L) \bar J^1(\beta, \L)$, en particulier :  $\bar J(\beta, \L) / \bar J^1(\beta, \L) \simeq 
 \bar P_\beta(\L) / \bar P^1_\beta(\L)$. 
\end{enumerate}
\end{Lemma}

 \begin{proof} 
 (i) 
  La d\'ecomposition correspondante dans $\widetilde G$ est connue \cite[\S 3]{S4}, celle dans $G$ aussi (elle r\'esulte de l'application de \cite[\S2]{S2} comme au paragraphe \ref{par31}), or  $J^1(\beta, \L)$  est contenu dans $G^\prime$. 
 
 (ii)  Voir   la preuve de  la Proposition \ref{entrelacement}.
 \end{proof}

\begin{Proposition}\label{kappa} Il existe des repr\'esentations $\bar \kappa$ de $\bar J(\beta, \L)$ prolongeant $\bar \eta$ et dont la res\-triction \`a un pro-$p$-sous-groupe de Sylow de $\bar J(\beta, \L)$ poss\`ede le m\^eme entrelacement que $\bar \eta$. \\
Deux telles repr\'esentations diff\`erent d'une torsion par un caract\`ere de $\bar P_{\beta}(\Lambda)/\bar P_{\beta}^1(\Lambda)$ trivial sur  tous ses sous-groupes unipotents.
\end{Proposition}

\begin{proof}  Commen\c cons par d\'ecrire un pro-$p$-sous-groupe de Sylow de $\bar J(\beta, \L)$. Gr\^ace \`a (\ref{trialite2}(ii)), il suffit de d\'eterminer un $p$-sous-groupe de Sylow de $\oP_{\beta}(\Lambda)/\oP_{\beta}^1(\Lambda)$.

On reprend la classification des strates semi-simples d\'ecrite au d\'ebut du paragraphe \ref{strates.4}. Selon les diff\'erents cas, on note :
\begin{enumerate}
\item []dans le cas (\ref{planhyperbolique}),  $V'=W^+$ et $H=\SL (V')$ ; \quad dans le cas (\ref{scalairegl2}), $V'=W_{\lambda}$ et $H=\GL(V')$~;
\item[]dans le cas (\ref{extensionquadratique}), $V'=W$ et $H=\SU (V')$ ; \quad dans le cas (\ref{scalairesu}), $V'=W$ et $H=\U(V')$.
\end{enumerate}
Dans chaque cas, $(\Lambda_{V'}, n,0, \beta_{V'})$ est une strate semi-simple dans l'alg\`ebre de Lie de $H$ et les lemmes \ref{dim2} et \ref{dim4} permettent d'affirmer que la restriction de $V$ \`a $V'$ induit un isomorphisme de $\bar P_{\beta}(\Lambda)/\bar P^1_{\beta}(\Lambda)$ sur $\tP_{\beta_{V'}}(\Lambda_{V'})\cap H/ \tP^1_{\beta_{V'}}(\Lambda_{V'})\cap H$.\\
Le corollaire 2.9 de \cite{S5} assure l'existence d'une suite de r\'eseaux $\Lambda^{'m}$ dans $V'$ normalis\'ee par $F[\beta_{V'}]^\times$, autoduale si $H$ est unitaire, telle que 
  $\tP_{\beta_{V'}}^1(\Lambda^{'m}_{V'})\cap H/\tP_{\beta_{V'}}^1(\Lambda_{V'})\cap H$ 
 soit  
 un $p$-sous-groupe de Sylow de $\tP_{\beta_{V'}}(\Lambda_{V'})\cap H/\tP^1_{\beta_{V'}}(\Lambda_{V'})\cap H$. 
Quitte \`a translater ou doubler $\Lambda^{'m}$ (cas (\ref{planhyperbolique}) ou cas (\ref{scalairesu})), $\Lambda^{'m}$ se prolonge en une suite de r\'eseaux $\Lambda^m$ dans $V$ correspondant \`a un point de l'immeuble de $\oG$ et normalis\'ee par $F[\beta]^\times$ (prop. \ref{volume0sl} \`a \ref{dimension4}). Alors :
  $$\oP_{\beta}^1(\Lambda)\subset \oP_{\beta}^1(\Lambda^m)\subset \oP_{\beta}(\Lambda^m)\subset \oP_{\beta}(\Lambda) $$
et $\oP_{\beta}^1(\Lambda^m)/\oP_{\beta}^1(\Lambda)$ est un $p$-sous-groupe de Sylow de $\oP_{\beta}(\Lambda)/\oP_{\beta}^1(\Lambda)$. D'apr\`es le lemme  \ref{lemme28}, nous pouvons choisir $\Lambda^m$ pour qu'en outre, $\widetilde{\mathfrak A}_0(\L^m)$ soit inclus dans $\widetilde{\mathfrak A}_0(\L)$. Nous choisissons un tel $\Lambda^m$ et nous posons $\oJ^1(\beta ;\Lambda^m ,\Lambda):=\oP^1_{\beta}(\Lambda^m)\oJ^1(\beta,\Lambda)$. C'est un pro-$p$-sous-groupe de Sylow de $\oJ(\beta,\Lambda)$.

\smallskip

 Construisons d'abord un prolongement $\bar\eta_{s}$ de $\bar \eta$ \`a   $\oJ^1(\beta ;\Lambda^m ,\Lambda)$ qui poss\`ede le m\^eme entrelacement que $\bar \eta$. Etant donn\'e que la strate consid\'er\'ee est autoduale, il convient de ``remonter'' jusqu'\`a $\tilde G$ afin d'\'etendre les r\'esultats de la proposition 3.7 de \cite{S5}.

On reprend les notations \ref{theta}. On pose :
$\bar \theta_{m}=\tau_{\Lambda,\Lambda^m,\beta}(\bar \theta)$. Alors : $\theta_{m}=\tau_{\Lambda,\Lambda^m,\beta}(\theta)$ et $\tilde\theta_{m}=\tau_{\Lambda,\Lambda^m,\beta}(\tilde\theta)$.
On note $\eta_{m}$, $\tilde \eta$ et $\tilde \eta_{m}$ les repr\'esentations irr\'eductibles de $J^1(\beta,\Lambda^m)$, $\tJ^1(\beta,\Lambda)$ et $\tJ^1(\beta,\Lambda^m)$ contenant $\theta_{m}$, $\tilde\theta$ et $\tilde\theta_{m}$ respectivement. 

Comme $ \widetilde{\mathfrak A}_0(\L^m)\subset \widetilde{\mathfrak A}_0(\L)$, la proposition 3.12  de \cite{S5} assure l'existence d'une unique repr\'esen\-ta\-tion irr\'eductible $\tilde \eta_{s}$ du groupe  $\tJ^1(\beta ; \Lambda^m,\Lambda):=\tP^1_{\beta}(\Lambda^m)\tJ^1(\beta,\Lambda)$ prolongeant $\tilde \eta$ et dont l'induite \`a $\tP^1(\Lambda^m)$ est \'equivalente \`a l'induite de $\tilde\eta_{m}$ \`a ce m\^eme groupe. 

Puisque $\tilde \theta$ est $\tau$-invariant, il en va de m\^eme de $\tilde \theta_{m}$ ($\tau$ commute au transfert), puis de $\tilde \eta$ et $\tilde \eta_{m}$ et par suite, de $\tilde \eta_{s}$. On note $\eta_{s}$ l'image de $\tilde \eta_{s}$ par la correspondance de Glauberman appliqu\'ee \`a l'action de $\tau$ sur $\tilde J^1(\beta ; \Lambda^m,\Lambda)$. C'est une repr\'esentation irr\'eductible de $J^1(\beta;\Lambda^m,\Lambda):=\tJ^1(\beta; \Lambda^m,\Lambda)^\tau$. En fait, $\tP^1(\Lambda^m)\cap \tJ^1(\beta,\Lambda)$ est un pro-$p$-groupe avec $p\neq 2$ et $J^1(\beta;\Lambda^m,\Lambda)$ n'est autre que $P^1(\Lambda^m)J^1(\beta,\Lambda)$. 

La repr\'esentation $\eta_{s}$ est un prolongement de $\eta$ dont l'induite \`a $P^1(\Lambda^m)$ est \'equivalente \`a l'induite de $\eta_{m}$ \`a ce m\^eme groupe \cite[Theorem 2.2]{S3}. Ces propri\'et\'es d\'efinissent $\eta_{s}$ uniquement (m\^eme argument que \cite[Proposition 5.1.14]{BK1}). 

Comme pr\'ec\'edemment, puisque $\theta$ et $\theta_{m}$ sont fixes par trialit\'e, $\eta_{s}$ est aussi fixe par trialit\'e. On d\'efinit $\bar \eta_{s}$ comme l'image de $\eta_{s}$ par la correspondance de Glauberman appliqu\'ee \`a l'action de $\Gamma$ sur $ J^1(\beta;\Lambda^m,\Lambda)$. Les m\^emes arguments que pr\'ec\'edemment ($p\neq 2,3$) montrent que $\bar \eta_{s}$ est l'unique prolongement de $\bar \eta$ \`a $\bar J^1(\beta;\Lambda^m,\Lambda)$ dont l'induite \`a $\oP^1(\Lambda^m)$ est \'equivalente \`a l'induite de $\bar\eta_{m}$ \`a ce m\^eme groupe. 

Connaissant l'entrelacement de $\bar \eta_{m}$ (proposition \ref{eta}) et  sachant par la d\'emonstration de la proposition \ref{entrelacement} que pour tout \'el\'ement $g$ de $\oG_{\beta}$,
$$\oJ^1(\beta,\Lambda)g\oJ^1(\beta,\Lambda)\cap\oG_{\beta} =\oP^1_{\beta}(\Lambda)g\oP^1_{\beta}(\Lambda),$$
donc que $\oJ^1(\beta;\Lambda^m,\Lambda)g\oJ^1(\beta; \Lambda^m,\Lambda)\cap\oG_{\beta} = \oP^1_{\beta}(\Lambda^m)g\oP^1_{\beta}(\Lambda^m)$, l'argument de \cite[Proposition 5.1.19]{BK1} montre que la repr\'esentation $\bar \eta_{s}$ est entrelac\'ee par tous les \'el\'ements de $\oG_{\beta}$.

\smallskip	

Il reste \`a prolonger $\bar \eta_{s}$ en une repr\'esentation de $\bar J(\beta,\Lambda)$. En reprenant le raisonnement de \cite[Theorem 4.1]{S5}, il suffit de montrer que tout caract\`ere $\varphi$ de $\oP^1_{\beta}(\Lambda^m)/\oP^1_{\beta}(\Lambda)$ qui est entrelac\'e par $\oP_{\beta}(\Lambda)$ se prolonge en un caract\`ere de ce groupe.

Or $\oP^1_{\beta}(\Lambda^m)/\oP^1_{\beta}(\Lambda)$ s'identifie \`a $\oP^1_{\beta_{V'}}(\Lambda^{'m})/\oP^1_{\beta_{V'}}(\Lambda_{V'})$ qui, \`a son tour, s'identifie \`a un produit de groupes r\'eductifs sur des corps finis de caract\'eristique diff\'erente de 2 et 3. On conclut en suivant la d\'emonstration de \cite[Lemma 3.9]{S5}.
\end{proof}

\subsection{Repr\'esentations supercuspidales}\label{types.cuspidales} 

Soit $[\L, n, 0, \beta]$ une strate semi-simple  de $\mathfrak g_2(F)$. Si elle est nulle, on note $\bar \kappa$ la repr\'esentation triviale de $\bar J(\beta,\Lambda)$ ; sinon, on consid\`ere un caract\`ere semi-simple $\bar \theta$ de $\bar H^1(\beta,\Lambda)$, $\bar \eta$ son extension de Heisenberg \`a $\bar J^1(\beta,\Lambda)$ et $\bar \kappa$ un prolongement de $\bar \eta$ \`a $\bar J(\beta, \L)$ construit au paragraphe pr\'ec\'edent. 

Soit $\bar \rho$ une repr\'esentation cuspidale de $\bar J(\beta,\L) / \bar J^1\beta,\L)\simeq \oP_{\beta}(\Lambda)/\oP^1_{\beta}(\Lambda)$. Notons $\oP^\circ _{\beta}(\Lambda)$ la composante connexe de $\oP_{\beta}(\Lambda)$ et  $\bar J^\circ(\beta,\Lambda):=\oP^\circ _{\beta}(\Lambda)\oJ^1(\beta,\Lambda)$. La restriction de $\bar \rho$ \`a $  \oP^\circ_\beta(\Lambda)/\oP^1_{\beta}(\Lambda)$ contient donc une repr\'esentation irr\'eductible cuspidale $\bar \rho^\circ$. On note encore $\bar \rho$ et $\bar \rho^\circ$ les rel\`evements de 
$\bar \rho$ et $\bar \rho^\circ$ \`a $\bar J(\beta,\L)$ et $\bar J^\circ(\beta,\Lambda)$ respectivement.

Soit  $\bar \lambda$ la repr\'esentation $\bar \kappa \otimes \bar \rho$ de $\bar J(\beta, \L)$. La paire $(\oJ(\beta,\Lambda), \bar\lambda)$ ainsi obtenue est un {\it type semi-simple} pour $\oG$.

\begin{Theorem}\label{cuspidales}
Soit $(\oJ(\beta,\Lambda),\bar \lambda)$ un  type semi-simple pour $\oG$.  Si $\bar P^\circ_\beta(\L)$ est un sous-groupe parahorique maximal de $\bar G_\beta$ dont le normalisateur dans $\oG_{\beta}$ est $\oP_{\beta}(\Lambda)$ alors la repr\'esentation induite  $\cInd_{\bar J(\beta, \L)}^{\bar G} \bar \lambda $ est une repr\'esentation irr\'eductible supercuspidale de $\oG$. 
\end{Theorem}

\begin{proof} Il s'agit de calculer l'entrelacement de  $\bar \kappa \otimes \bar \rho$, calcul analogue \`a la d\'emons\-tra\-tion de la proposition \cite[6.18]{S5}. Si $g$ est un \'el\'ement de $\oG$ qui entrelace $\bar \kappa \otimes \bar \rho$, le calcul de l'entrelacement de $\bar \eta$ permet de consid\'erer que  $g$ appartient \`a $\oG_{\beta}$ et qu'il entrelace la repr\'esentation $\bar \kappa \otimes \bar \rho^\circ$. Puis, l'argument de \cite[Proposition 5.3.2]{BK1} dans lequel on utilise que l'entrelacement de la restriction de $\bar \kappa$ \`a $\oJ(\beta;\Lambda^m,\Lambda)$ est \'egal \`a celui de $\bar \eta$ montre que $g$ entrelace la restriction de $\bar \rho^\circ$ \`a $\oJ(\beta;\Lambda^m,\Lambda)$. Alors, de la proposition \cite[1.1]{S5} sachant que $\bar P^\circ_\beta(\L)$ est un sous-groupe parahorique maximal de $\bar G_\beta$, on d\'eduit que $g$ appartient au normalisateur dans $\oG_{\beta}$ de  $\bar P^\circ_\beta(\L)$, c'est-\`a-dire $\oP_{\beta}(\L
 )$. Ainsi l'entrelacement de $\bar \lambda$ est \'egal \`a $\oJ(\beta,\Lambda)$.
\end{proof}

 \begin{biblio}{}

\bibitem{BL} P. Broussous and B. Lemaire,  {\it Building  of $GL(m,D)$ and
centralizers}, Transformation Groups   \textbf{7}(1) (2002), 15--50.

\bibitem{BS} P. Broussous and S. Stevens,  {\it Buildings of classical groups and
centralizers of Lie algebra elements}, J. of Lie Theory \textbf{19} (2009), no. 1, 55--78.
 
\bibitem{BT0} F. Bruhat et J. Tits, {\it Sch\'emas en groupes et immeubles des groupes classiques sur un corps local, 1\`ere partie : le groupe lin\'eaire g\'en\'eral}, Bull. Soc. Math. France 
\textbf{112} (1984), 259--301. 

\bibitem{BT} F. Bruhat et J. Tits, {\it Sch\'emas en groupes et immeubles des groupes classiques sur un corps local, 2\`eme partie : groupes unitaires}, Bull. Soc. Math. France
\textbf{115} (1987), 141--195.

\bibitem{BK1} C. J. Bushnell and P. C. Kutzko, {\it 
The admissible dual of ${\rm GL}(N)$ via compact open subgroups}, Annals of Mathematics Studies  129,  Princeton University Press, Princeton, NJ, 1993.

\bibitem{BK2} C. J. Bushnell and P. C. Kutzko,  {\it Semi-simple types in $GL(n)$},  
 Compositio Math. \textbf{119} (1) (1999),  53--97.

\bibitem{D} J.-F. Dat, {\it Finitude pour les repr\'esentations lisses de groupes $p$-adiques},
  J. Inst. Math. Jussieu \textbf{8} (2) (2009),   261--333. 

\bibitem{GY} W. T. Gan and J.-K. Yu, {\it 
 Sch\'emas en groupes et immeubles des groupes exceptionnels sur un corps local. I. Le groupe $G_2$},  Bull. Soc. Math. France \textbf{131} (3) (2003),   307--358.

\bibitem{Kim} Ju-Lee Kim, {\it  Supercuspidal representations: an exhaustion theorem},  J. Amer. Math. Soc. \textbf{20} (2)(2007),   273--320. 

\bibitem{L} B. Lemaire, {\it  Comparison of lattice filtrations and Moy-Prasad filtrations for classical groups}, J. Lie Theory \textbf{19} (1) (2009),   29--54.

\bibitem{M} L. Morris, {\it Tamely ramified supercuspidal representations of classical groups. I. Filtrations } ,  Ann. Sci. Ecole Norm. Sup. (4) \textbf{24} (6) (1991),   705--738.

\bibitem{M2} L. Morris, {\it Level zero $\bf G$-types},  Compositio math.  \textbf{118} (2) (1999),   135--157.

\bibitem{OM} O. T. O'Meara, {\it    Introduction to quadratic forms},  Die Grundlehren der mathematischen Wissenschaften  117,  Springer-Verlag 1963.  

\bibitem{RS} S. Rallis and G. Schiffmann, {\it Theta correspondence associated to $G_2$},  Amer. J. Math. \textbf{111} (5) (1989),   801--849. 

\bibitem{SV} T. Springer and F. Veldkamp, {\it Octonions, Jordan algebras and exceptional groups},  Springer Monographs in Mathematics, Springer-Verlag  2000.

\bibitem{S2} S. Stevens, {\it Double coset decompositions and intertwining},
manuscripta math.  \textbf{106} (2001), 349--364.

\bibitem{S3} S. Stevens, {\it Intertwining and supercuspidal types for $p$-adic classical groups},
Proc. London Math. Soc. \textbf{83} (2001), 120--140.

\bibitem{S4} S. Stevens, {\it Semisimple characters for {$p$}-adic
classical groups}, Duke Math.\ J.\,  \textbf{127} (1) (2005), 123--173.

\bibitem{S5} S. Stevens, {\it The supercuspidal representations of {$p$}-adic classical groups},   Invent. Math. \textbf{172} (2) (2008),   289--352. 

\bibitem{W} J. S. Wilson, {\it Profinite groups}, London Mathematical Society Monographs  New Series, 19,  The Clarendon Press, Oxford University Press 1998. 

\bibitem{Yu} J.-K. Yu, {\it 
Construction of tame supercuspidal representations}, 
J. Amer. Math. Soc. \textbf{14} (3) (2001),  579--622. 

\end{biblio}

\bigskip

\small

Corinne Blondel\hfill Laure Blasco \break \indent
C.N.R.S. -  Groupes, repr\'esentations et g\'eom\'etrie,
\hfill  D{\'e}partement de Math{\'e}matiques
\break\indent
Institut de Math{\'e}matiques de Jussieu
\hfill  et U.M.R 8628 du C.N.R.S.\break\indent
Universit{\'e} Paris 7 - Case 7012\hfill Universit{\'e} Paris-Sud 11, B{\^a}timent 425 \break\indent
F-75205 Paris Cedex 13 \hfill  F-91405 Orsay Cedex \break\indent
France \hfill France \break\indent

\smallskip
blondel@math.jussieu.fr\hfill laure.blasco@math.u-psud.fr \break\indent
\end{document}